\documentclass[a4paper]{article}
\usepackage{amsthm,amssymb,amsmath,enumerate,graphicx,psfrag,color}
\usepackage{hyperref} 
\hypersetup{
colorlinks=true,  
linkcolor=black,    
urlcolor=red      
}
\usepackage{tikz}
\usepackage{enumerate}
\usepackage{refcount}
\usepackage{subcaption} 

\usepackage[utf8x]{inputenc}
\usepackage{longtable}
\usepackage{arydshln} 

\usetikzlibrary{backgrounds}

\newtheorem{definition}{Definition}

\newtheorem{theorem}[definition]{Theorem}

\newtheorem{lemma}[definition]{Lemma}

\newtheorem{question}[definition]{Question}

\usetikzlibrary{calc}
\usetikzlibrary{backgrounds,positioning}
\usetikzlibrary{shapes.geometric}

\tikzstyle{hvertex}=[solid,thick,circle,inner sep=0.cm, minimum size=2mm, fill=white, draw=black]
\tikzstyle{lv}=[thick,circle,inner sep=0.cm, minimum size=1em, fill=white, draw=black]
\tikzstyle{v}=[thick,circle,inner sep=0.cm, minimum size=1.5mm, fill=white, draw=black]
\tikzstyle{smallvx}=[thick,circle,inner sep=0.cm, minimum size=1.5mm, fill=white, draw=black]
\tikzstyle{hedge}=[very thick]
\colorlet{hellgrau}{black!30!white}
\colorlet{dunkelgrau}{black!50!white}
\colorlet{hellblau}{blue!20!white}
\colorlet{hellrot}{red!40!white}
\tikzstyle{convcols}=[color=white,fill=hellblau]
\tikzstyle{labvx}=[solid,thick,circle,inner sep=0.3mm, minimum size=2mm, fill=white, draw=black]
\tikzstyle{ced}=[hedge,densely dashed]

\newcommand{\svx}{node[labvx]{\texttt s}}
\newcommand{\tvx}{node[labvx]{\texttt t}}
\newcommand{\xvx}{node[labvx]{\texttt x}}

\renewcommand{\phi}{\varphi}

\newcommand{\nestle}{\leq_\text{\rm nest}}
\newcommand{\gestle}{\geq_\text{\rm nest}}

\newcommand{\bigO}{\ensuremath{O}}

\newcommand{\comment}[1]{}
\newcommand{\N}{\mathbb N}
\newcommand{\R}{\mathbb R}

\newcommand{\cP}{\mathcal{P}}

\newcommand{\cF}{\mathcal{F}}

\newcommand{\cC}{\mathcal{C}}

\newcommand{\cE}{\mathcal{E}}

\newcommand{\cB}{\mathcal{B}}
\newcommand{\cW}{\mathcal{W}}

\renewcommand{\epsilon}{\varepsilon}

\newcommand{\emtext}[1]{\text{\em #1}}

\newcommand{\mymargin}[1]{%
  \marginpar{%
    \begin{minipage}{1.3\marginparwidth}\small%
      \begin{flushleft}%
        #1%
      \end{flushleft}%
   \end{minipage}%
  }%
}%

\newcommand{\EP}{Erd\H os and P\'osa}
\newcommand{\EPP}{Erd\H os-P\'osa~property}

\newcommand{\sm}{\setminus}

\newcounter{type}

\title{$K_4$-subdivisions have the edge-Erd\H os-P\'osa~property}
\author{Henning Bruhn \and Matthias Heinlein}
\date{}

\begin{document}
\maketitle
\begin{abstract}
We prove that every graph $G$ contains either $k$ edge-disjoint $K_4$-subdivisions
or a set $X$ of at most $\bigO(k^8  \log k)$ edges
such that $G-X$ does not contain any $K_4$-subdivision.
This shows that $K_4$-subdivisions have the edge-\EPP.
\end{abstract}

\section{Introduction}
In 1965, \EP\ proved a duality between packing and covering of cycles in graphs:

\begin{theorem}[\EP\ 1965, \cite{EP65}]\label{epthm}
For every graph $G$ and every integer $k$,
the graph $G$ either contains $k$ disjoint cycles
or a set $X\subseteq V(G)$ with $|X|=\bigO(k \log k)$
such that $X$ meets every cycle.
\end{theorem}

Other graph families admit similar dualities: 
a family $\cF$ of graphs 
has the \emph{\EPP}
if there is a \emph{bounding function} $f:\N\to\R$
such that every graph $G$ either contains $k$ disjoint subgraphs 
each isomorphic to an element of $\cF$,
or $G$ contains a \emph{hitting set} $X$ of at most $f(k)$ vertices
such that $G-X$ contains no subgraph isomorphic to an element of $\cF$.
Thus, Theorem \ref{epthm} states that the family of cycles has the \EPP\ with $f = \bigO(k \log k)$.

The most general result about the \EPP\ concerns the family of $H$-expansions,
for a fixed graph $H$, where a graph is an  \emph{$H$-expansion}
if it contains $H$ as a minor:

\begin{theorem}[Robertson, Seymour, \cite{RS86}] \label{metaThm}
The family of $H$-expansions has the \EPP\
if and only if $H$ is planar.
\end{theorem}

Note that the theorem includes Theorem~\ref{metaThm}, if $H$ is chosen to be a triangle.

What happens if we pack edge-disjoint cycles
rather than (vertex-)disjoint cycles? It turns out that there is an 
edge-version of Erd\H os and P\'osa's theorem. In order to generalise
in the same way as above, 
we say that a family $\mathcal F$ of graphs has the 
\emph{edge-Erd\H os-P\'osa property}
if there is  $f:\mathbb N\to\mathbb R$ 
such that for every integer $k$ and every graph $G$, 
there are either $k$ \emph{edge-disjoint} subgraphs in $G$ that lie in $\mathcal F$,
or there is an \emph{edge set} $F$ of size $|F|\leq f(k)$
such that $G-F$ does not contain any subgraph from $\mathcal F$.

Compared to the ordinary (vertex-)\EPP\
there are only a few classes known to have the edge-\EPP:
cycles have the edge-property (see, for instance, \cite[Exercise 9.5]{Die10}),
as do long cycles~\cite{BHJ17} and as do $\theta_r$-expansions
(see Raymond, Sau and Thilikos~\cite{RST13}), where $\theta_r$ is the multigraph consisting of $r$ parallel edges.

It is striking to note that these results are all special instances
of a hypothetical edge-version of Theorem~\ref{metaThm} -- we just
have to take $H$ as a triangle, a long cycle or as $\theta_r$.

\begin{question}\label{edgeVersion}
	Is there an edge-version of Theorem~\ref{metaThm}? That is, do $H$-expansions have the edge-\EPP\ if $H$ is a planar graph?
\end{question}


We answer the question for $H=K_4$, the smallest 
$2$-connected graph $H$ for which Question~\ref{edgeVersion} was open.
We prove that $K_4$-subdivisions do, indeed, have the edge-\EPP.
(As $K_4$ is a cubic graph,
it makes no difference whether we consider expansions or subdivisions.)


\begin{theorem}
\label{mainThm}
Either $G$ contains $k$ edge-disjoint $K_4$-subdivisions or 
there is a set $Y\subseteq E(G)$ of size $\bigO(k^8 \log k)$
such that $G-Y$ does not contain any $K_4$-subdivision.
\end{theorem}

Our initial motivation was to study a special case of Question~\ref{edgeVersion}
in order to develop techniques that could be helpful in 
a more general setting. A modest but useful insight in this respect
is spelled out
in Section~\ref{reductionOneVertex}, where we prove that 
one may normally assume that there is a (vertex) hitting set consisting
of a single vertex. 

Our original aim, however, failed. While working on Theorem~\ref{mainThm},
we discovered planar graphs
such that their expansions do not have the edge-\EPP.
In fact, even if $H$ is a ladder of large enough length or a subcubic 
tree of large enough pathwidth,
the family of $H$-expansions does not have the edge-property.
We work out these negative results in a subsequent paper~\cite{BHJ18trees}.

That the answer of Question~\ref{edgeVersion} is ``no'' makes it, in 
our eyes, even more interesting to study which $H$-expansions
have the edge-property. 

\bigskip

Not only certain $H$-expansions have the \EPP, but also other types
of graph classes. 
For example, cycle of even length have the \EPP\ (see~\cite{Tho88} and~\cite{DNL87})  and also the edge-property \cite{BHJ18paths},
while odd cycles have neither the vertex- nor the edge-property~\cite{DNL87}.
Thomassen~\cite{Tho88}, and later Wollan~\cite{Wol11} considered
cycles with more general modularity constraints: 
they proved  that cycles of length~$\equiv 0 \pmod{m}$ for all $m$, resp.\ 
cycles of length~$\not\equiv 0 \pmod{m}$ for odd $m$
have the vertex-\EPP. The edge-versions of these results are still open.

Stretching the definition a bit, we can interprete Menger's theorem 
as saying that $A$--$B$-paths have the \EPP, as well as the edge-property. 
The theorems of Gallai~\cite{Gal61} and of Mader~\cite{Mad78}
show that $A$-paths have the ordinary as well as the edge-property. 

Another type of cycle that has attracted a good amount of interest~\cite{KKM11,KK12, PW12}
are $A$-cycles, cycles that each contain at least one vertex from a fixed set $A$. 
The class of $A$-cycles has both versions of \EPP\ as well; see in particular~\cite{PW12}.
Huynh, Joos and Wollan~\cite{HJW16}
generalised the concept of an $A$-cycle by considering 
non-zero cycles in graphs whose edges are endowed with labels 
from two groups. They in particular characterise when these cycle have the \EPP.
Bruhn, Joos and Schaudt~\cite{BJS18} investigate labelled $H$-expansions.

An overview on Erd\H os-P\'osa properties can be found in the survey article  
of Raymond and Thilikos~\cite{RT16}, and also in~\cite{BHJ18paths}.

\section{Preliminaries}

All graphs in this paper are finite, simple, undirected graphs.
We use standard graph theory notation as used in Diestel's textbook \cite{Die10}.
We recall some definitions and concepts that we use often.

Let $P$ be a path with endvertices $u$ and $v$.
Then we say that $P$ is a $u$--$v$-path.
If $a,b\in V(P)$ are two vertices of $P$, 
the subpath of $P$ between $a$ and $b$ is denoted by $aPb$.
If $w$ is a neighbour of the endvertex $u$ and $w\notin V(P)$, then $wuP$ denotes the concatenation of the path $wu$
and the path $P$.

A path is \emph{trivial} if it contains no edge and a block is \emph{trivial} if it consists of a single edge.

Recall the definition of the vertex- and edge-\EPP\ given in the introduction.
The set $X$ that meets every $\cF$-graph in $G$ is called \emph{hitting set (for $\cF$-graphs in $G$)}.
If we want to specify if $X$ consists of vertices or edges, we write \emph{vertex hitting set} or \emph{edge hitting set}.

\subsection{Series-parallel graphs}
\label{seriesparallel}

There are several almost identical notions of what a series-parallel graph
is. Here, we use what some call \emph{two-terminal series-parallel}. 
If $G$ is a graph and $s,t$ are two (not necessarily distinct) vertices in $G$,
we say that $G$
is \emph{series-parallel with terminals $s,t$} if it 
\begin{itemize}
\item is equal to the vertex $s$ if $s=t$; or
\item is equal to the edge $st$ ( if $s\neq t$); or
\item if there is a series-parallel graph $G_1$ with terminals $s,r$ and 
a series-parallel graph $G_2$ with terminals $r,t$ such that 
$G=G_1\cup G_2$ and $G_1\cap G_2=\{r\}$; or
\item if there are series-parallel graphs $G_1,G_2$ both with terminals $s,t$  such that 
$G=G_1\cup G_2$ and $V(G_1\cap G_2)=\{s,t\}$.
\end{itemize}
In the third case, $G_1 \cup G_2$ is a  \emph{series decomposition} of $G$, and in the last
case $G_1\cup G_2$ is a  \emph{parallel decomposition} of $G$.
Note that the terminals of a series-parallel graph are distinct if it contains an edge.

By a \emph{diamond} we mean a graph consisting of three internally disjoint $a$--$b$-paths for two distinct vertices $a,b$.
Let $u_1,u_2\in V(G)$. We call a subgraph $D'\subseteq G$ 
a \emph{$u_1$--$u_2$-diamond} if it is the union of a diamond $D=P_1\cup P_2\cup P_3$ 
and disjoint paths $Q_1,Q_2$ such that $Q_i$ connects $u_i$ and the interior of $P_i$
and is disjoint from $P_j$ for $j\neq i$.

Here we give a basic lemma about series-parallel graphs and diamonds.
We need it in the proof of Lemma \ref{moduleHasBlueprint}.

\goodbreak
\begin{lemma}\label{diamonds} $\quad$
\begin{enumerate}[\rm (i)]
\item Let $G$ be a $2$-connected graph, and let $st$ be an edge of $G$. 
Then $G$ does not contain a $K_4$-subdivision if and only if 
 $G$ is series-parallel with terminals $s,t$.
\item A series-parallel graph with terminals $s,t$
does not contain an $s$--$t$-diamond.
\end{enumerate}
\end{lemma}
\begin{proof}
For (i): Eppstein (Lemma~9 in \cite{Epp92}) proves that a $2$-connected graph that is series-parallel
with some terminals is also series-parallel if any two adjacent vertices are chosen as terminals. 
The rest of the statement is elementary.

For (ii): Suppose there is an $s$--$t$-diamond $D$ in a series-parallel graph $G$
with terminals $s,t$. We add an $s$--$t$-path $P$ that is internally disjoint from $G$, and observe that the resulting
graph $G'$ is still series-parallel (with terminal $s,t$) as $G'=G\cup P$ is a parallel decomposition. 
The graph $G'$, however, contains the $K_4$-subdivision  $D\cup P$, which is impossible, by~(i).
\end{proof}

\section{First reductions}
\label{sec:reductions}

Our aim in this article is to prove Theorem~\ref{mainThm}.
Let $G$ be a graph and assume that it does not contain $k$ edge-disjoint $K_4$-subdivisions.
Then it does not contain $k$ vertex-disjoint $K_4$-subdivisions either.
Thus, Theorem~\ref{metaThm}
yields a set $X\subseteq V(G)$ of size bounded in $k$
such that $G-X$  contains no $K_4$-subdivision.
Aboulker et al.~\cite{AFHJRS17} 
showed that  $|X|=\bigO(k\log k)$ is sufficient.
Our proof uses this vertex hitting set to construct 
an edge hitting set.

We  give a brief overview of the proof.
In Section~\ref{reductionOneVertex} we start by showing that we may 
assume to have a vertex hitting set that consists of a single vertex~$x$. 
In a second step, in Section~\ref{twoconnSection}, 
we reduce to the case that $G-x$ is $2$-connected.

This implies, in particular,  
that $G-x$ is series-parallel and has thus an ear decomposition 
that is \emph{nested}; see Section~\ref{section:neds}. 
If this nested ear decomposition contains many ears that contain a neighbour of $x$, 
then
we will find $k$ edge-disjoint $K_4$-subdivisions (Section~\ref{largeS}).
If, on the other hand, the number of such ears  is bounded (Section~\ref{smallS}),
we can do induction over that number. This case will take up most of the effort.

\subsection{Single vertex hitting set}
\label{reductionOneVertex}

As a first step we reduce to graphs with a vertex hitting set that consists of a single vertex.

\begin{lemma}\label{singlehitlem}
Let $\mathcal F$ be a class of graphs that has the vertex-Erd\H os-P\'osa property
with bounding function~$g$. Let $h:\mathbb N\to\mathbb R$ be a function such that 
for every $k$ and for every graph $G$ that has a vertex $x$ such that $G-x$ does not contain any subgraph of $\mathcal F$
it holds that: either $G$ contains $k$ edge-disjoint subgraphs from $\mathcal F$
or there is an edge set $F$ of size $|F|\leq h(k)$ that meets every subgraph from $\mathcal F$.
Then $\mathcal F$ has the edge-Erd\H os-P\'osa property with bounding function $f=gh$.
\end{lemma}

\begin{proof}
We proceed by induction on the size of a vertex hitting set for $\mathcal F$. 
Let $G$ be a graph. For the edge-Erd\H os-P\'osa property we can clearly 
assume that $G$ does not contain any $k$ edge-disjoint (or vertex-disjoint) subgraphs 
from $\mathcal F$. 
Therefore, $G$ has a vertex hitting set 
for $\mathcal F$, and we choose such a set $X$ as small as possible. If $X$ is empty then the empty 
set is an edge hitting set for all subgraphs  in $\mathcal F$. 

If $X$ is not empty, pick $x\in X$
and observe that $G'=G-x$ has a smaller vertex hitting set (namely $X'=X\sm\{x\}$) than $G$. 
By induction, there is thus an edge set $F'$ of size $|F'|\leq |X'|\cdot h(k)$
such that $G'-F'$ is devoid of subgraphs from $\mathcal F$.
Consequently, every subgraph of $G''=G-F'$ that is contained in $\mathcal F$ must meet the vertex~$x$.
By assumption, there is thus an edge set $F''$ of size $|F''|\leq h(k)$ such that $G''-F''$ does 
have any subgraphs in $\mathcal F$. In total, we see that $F'\cup F''$ is an edge hitting set
for $\mathcal F$ in $G$ of a size of 
\[
|F'\cup F''|\leq h(k)+|X'|h(k)=|X|h(k)\leq g(k)h(k),
\] 
where the last inequality follows from the minimal choice of $X$ and the fact that $g$ is a bounding function for 
the vertex-Erd\H os-P\'osa property.
\end{proof}

\subsection{Reduction to the case that $G-x$ is a block}
\label{twoconnSection}

Having reduced to graphs $G$ with a single-vertex hitting set $\{x\}$, we further reduce 
to the case when $G-x$ is (essentially) $2$-connected. 
Independently of the results in this subsection, we prove 
in Sections~\ref{section:neds}--\ref{smallS} the following lemma:

\begin{lemma}
\label{lem:twoConn}
Let $x$ be a vertex in a graph $G$ such that every $K_4$-subdivision in $G$ meets $x$, and let
each  neighbour of $x$ in $G$ have degree~$2$. If $G-N(x)-x$ is $2$-connected, then 
$G$ either contains $k$ edge-disjoint $K_4$-subdivisions
or an edge  set of size $\bigO(k^4)$ that meets every $K_4$-subdivision.
\end{lemma}

With the help of Lemma \ref{lem:twoConn} we show the main result of  this subsection:

\begin{lemma}\label{conredlem}
Let $G$ be a graph with a vertex $x$ such that $G-x$ does not contain any $K_4$-subdivision. Then 
for every $k$, the graph $G$ either
contains $k$ edge-disjoint $K_4$-subdivisions or there is an edge set $F$ of size $|F|=\bigO(k^7)$
that meets every $K_4$-subdivision. 
\end{lemma}

Once we have proved the two lemmas above, we have proved the main theorem:
\begin{proof}[Proof of Theorem~\ref{mainThm}]
Lemma \ref{conredlem} provides a function $h$ as in Lemma \ref{singlehitlem}
for the family $\cF$ of $K_4$-subdivisions.  
The function $g$ for $\cF$ is in $\bigO(k\log k)$ by \cite{AFHJRS17}.
Hence, the hitting set function in Theorem \ref{mainThm} is $k \log k \cdot \bigO(k^7)=\bigO(k^8 \log k)$.
\end{proof}

The remainder of this  section is devoted to the proof  of Lemma~\ref{conredlem}.
Before we start, let us denote the size of the hitting set $F$ in Lemma~\ref{conredlem}
by $f(k)$. Then the function $f$ we will find will satisfy $f(k)\geq k^7$, which implies 
that  
\begin{equation}\label{goodf}
f(k_1)+f(k_2)\le f(k_1+k_2-1)\emtext{ for all positive integers }k_1,k_2.
\end{equation}
We fix throughout the rest of the section a positive integer $k$, and 
a graph $G$ that has a vertex $x$ such that $G-x$ does not contain any $K_4$-subdivision.
We also assume that
\begin{equation}\label{nok}
\emtext{
$G$ does not contain $k$ edge-disjoint $K_4$-subdivisions.}
\end{equation}
Moreover, we may assume  that every edge $e$ of $G$ is contained in a $K_4$-subdivision;
any edge $e$ that is not contained  in one may simply be omitted
as a hitting set in $G-e$ is still a hitting set in $G$.

Suppose that $G$ contains several blocks. 
Every $K_4$-subdivision is contained in a single block.
Because every edge of $G$ is contained in a $K_4$-subdivision, every block contains a $K_4$-subdivision.
Hence, we can apply induction on every block of $G$.
By~\eqref{goodf}, 
the union of the hitting sets of each block has size smaller than $f(k)$.
Therefore, we may assume from now that
\begin{equation}\label{blab}
\emtext{$G$ is $2$-connected, and thus $G-x$ is connected.}
\end{equation}

Furthermore, we can assume that $G-x$ is not $2$-connected: 
otherwise, after subdividing the edges incident with~$x$, we can directly apply Lemma~\ref{lem:twoConn},
which then finishes the proof of Lemma~\ref{conredlem}.

In the rest of the section, whenever we speak of a \emph{block}, without specifying
of which graph, 
 we mean a 
block \emph{of $G-x$}.
Pick an arbitrary cutvertex $r^*$ of $G-x$, and use it to define a partial order on the blocks
by setting $B\geq B'$ for two blocks $B,B'$  if either $B=B'$ or if there is a $B$--$r^*$-path
in $G$ that passes through an edge of $B'$. (It is easy to verify that then every $B$--$r^*$-path
meets $E(B')$.) 

For a block $B$, 
we write $u_B$ for the unique cutvertex of $G-x$ that lies in $B$ and that 
separates $B$ from $r^*$. We, furthermore, denote for any block $B$  by $G_{\geq B}$
the subgraph of $G$ induced by $\bigcup_{B'\geq B}B'\cup\{x\}$,
where the union is over all blocks $B'$  with $B'\geq B$,
and where we exclude the edge $u_Bx$ from $G_{\geq B}$, should it exist. 
We also define $G_{\not\geq B}$ as the subgraph of $G$ induced by $\bigcup_{B'\not\geq B}B'\cup\{x\}$.
Note that if $u_Bx$ is an edge of $G$, then it lies in $G_{\not\geq B}$ but not in $G_{\geq B}$. 
Thus,  $G=G_{\geq B}\cup G_{\not\geq B}$ is an edge-disjoint union, and indeed  $G_{\geq B}$ and $G_{\not\geq B}$
meet precisely in $\{x,u_B\}$.

A block $B$ is \emph{essential} if there is a $K_4$-subdivision $K$ such that $K\cap B$
contains a cycle. Any such $K_4$-subdivision then \emph{makes $B$ essential}.
Note that every $K_4$-subdivision makes precisely one block essential.
For an essential block $B$ there is a (unique) $\leq$-largest block $A$ such that 
$G_{\geq A}$ contains a $K_4$-subdivision that makes $B$ essential. 
That block $A$ is the \emph{baseblock of $B$}, or simply a \emph{baseblock} if 
there is a block for which it is the baseblock. Clearly, $A\leq B$, and 
note that $A=B$ may happen. 

Here is a short overview which steps we take to prove Lemma \ref{conredlem}.
Lemmas \ref{esslem}, \ref{esspathlem} and \ref{bblocklem} provide some basic properties of essential blocks and baseblocks.
Using them, Lemma~\ref{fewbaseblocks} shows that we cannot have $3k$ baseblocks in $G$
because otherwise we would find $k$ edge-disjoint $K_4$-subdivisions.
Then, our aim is to bound the number of essential blocks that belong to the same baseblock.
Lemmas~\ref{diamondchain1},~\ref{moreesslem} and~\ref{diamondchain2}
prove that almost all essential blocks that have the same baseblock have a certain simple structure.
Then, Lemma \ref{fewessblocks} uses this structure to find a small set of edges
that makes almost all such essential blocks \emph{inessential} 
(the set meets all $K_4$-subdivisions that make the block essential).
As now, there are only few baseblocks and only few essential blocks per baseblock left,
we can focus on every essential block 
and either find $k$ edge-disjoint $K_4$-subdivisions there or an edge hitting set of bounded size (done by Lemma \ref{lem:twoConn}). 
The union of these sets over all remaining essential blocks is the hitting set for the whole graph.

We start the proof of Lemma~\ref{conredlem} with a couple of lemmas about essential blocks and baseblocks.

\begin{lemma}\label{esslem}
Let $B$ be an essential block, and let $u$ be a cutvertex of $G-x$ contained in $B$.
Let $C$ be a component of $G-\{u,x\}$ that is disjoint from $B$,
and let $K$ be a $K_4$-subdivision that makes $B$ essential.  
If $P$ is a $u$--$x$-path internally disjoint from $B$, then $(K-C)\cup P$ 
contains a $K_4$-subdivision that makes $B$ essential.
\end{lemma}
\begin{proof}
Whether $x$ is a branch vertex of $K$ or whether it lies on a subdivided edge does not change the fact
that all other branch vertices lie in the unique non-trivial block of $K-x$, and thus in $B$. As $\{u,x\}$
separates $B$ from $C$ in $G$, it follows that $K\cap C$ is either empty or contains a subpath of 
a subdivided edge of $K$, that then is continued to $x$ and to $u$. Replacing this subpath, if it exists,
by $P$ results in a new $K_4$-subdivision that again contains a cycle in $B$, i.e.\ that makes $B$ essential.
\end{proof}

\begin{lemma}\label{esspathlem}
Let $B$ be an essential block, let $A$ be the baseblock of $B$,
and let $K$ be a $K_4$-subdivision that makes $B$ essential. Then:
\begin{enumerate}[\rm (i)]
\item\label{essi} if $B=A$ then there are two 
internally disjoint $u_B$--$x$-paths in $G_{\geq B}$;
\item\label{essii} if $B>A$ then $K\cap G_{\geq B}$ contains  a $u_B$--$x$-path; and
\item\label{essiii} if $B>A$ then $K\cap G_{\ngeq B}$ contains a $u_B$--$x$-path
that meets $A$.
\end{enumerate}
\end{lemma}
\begin{proof}
We recall that $K\cap B$ contains a cycle, and thus the whole unique non-trivial block of $K-x$. 
We start with the proof of~\eqref{essiii}. As $B>A$, it follows that $A\subseteq G_{\ngeq B}$. 
By definition of a baseblock, and as $B\neq A$, 
 the $K_4$-subdivision $K$ must meet $A$ and must have an edge outside $G_{\geq B}$, and thus in 
 $G_{\ngeq B}$. As $G_{\ngeq B}$ is separated from $B$ by $\{u_B,x\}$, this means 
that $K\cap G_{\ngeq B}$ is a subpath $Q$ of a subdivided edge of $K$. In particular, $Q$ is a $u_B$--$x$-path
that meets $A$.

Next, we treat~\eqref{essii}.
Let $C$ be a cycle in $K\cap B$. 
Then there are at least two $C$--$x$-paths contained in $K$ that meet only in $x$. Of these, only one may 
pass through $u_B$; the other, $P$ say, is disjoint from $u_B$ and therefore contained in $G_{\geq B}$.
Since, by~\eqref{essiii}, $K$ also contains  $u_B$ and since it is connected, we therefore 
find  a $u_B$--$x$-path in $K\cap G_{\geq B}$.

For~\eqref{essi}, consider a $K_4$-subdivision $K'$ that makes $B$ essential and that is contained
in $G_{\geq B}$. Such a $K'$ exists as $B$ is its own baseblock. As above, $K'\cap B$ contains a 
cycle $C'$, and two 
 $C'$--$x$-paths that meet only in $x$. As $K'\subseteq G_{\geq B}$ (as $A=B$), the two paths 
are contained in  $G_{\geq B}$. 
As $B$ is $2$-connected, there are two $u_B$--$C'$-paths that only meet in $u_B$.
These paths can be extended disjointly from each other on $C'$ 
to the start vertices of the $C'$--$x$-paths so that we obtain internally 
disjoint $u_B$--$x$-paths in $G_{\geq B}$.
\end{proof}

\begin{lemma}\label{bblocklem}$\,$
\begin{enumerate}[\rm (i)]
\item\label{blocki}
Let $A$ be a baseblock, and let $B'$ be some block with $B'>A$. Then $G_{\geq A}\cap G_{\not\geq B'}$
contains a $u_A$--$x$-path.
\item\label{blockii} Let $A$ be the baseblock of an essential block $B>A$, and let $B'$ be some block with $A<B'<B$. Then 
there is no $u_B$--$x$-path in $G_{\geq B'}\cap G_{\ngeq B}$. 
\item\label{blockiii} Let $A$ be the baseblock of an essential block $B$ and let $A'$ be 
another block such that $A<A'<B$. Then $A'$ cannot be a baseblock.
\end{enumerate}
\end{lemma}
\begin{proof}
Let $B$ be an essential block such that $A$ is the baseblock of $B$,
and let $K$ be a $K_4$-subdivision that makes $B$ essential, and that is contained in $G_{\geq A}$
(such a $K$ exists, by definition of a baseblock).

Assume first that $B$ and $B'$ are incomparable, or that $B=A$. In either case, 
Lemma~\ref{esspathlem}~\eqref{essi} or~\eqref{essii}
yields a $u_B$--$x$-path in $G_{\geq B}\cap G_{\ngeq B'}$. We extend the path 
through $G_{\geq A}\cap G_{\ngeq B'}$ to a $u_A$--$x$-path.

Now assume that $B$ and $B'$ are comparable and that $B>A$.
Lemma~\ref{esspathlem}~\eqref{essiii} yields a $u_B$--$x$-path $P$ in $K\cap G_{\ngeq B}$ that meets $A$.
In particular, $P\subseteq  G_{\geq A}\cap G_{\ngeq B}$ as $K\subseteq G_{\geq A}$.
Let $a$ be the last vertex of $P$ in $A$. 
Then $aP$ is an $A$--$x$-path in $G_{\geq A}\cap G_{\ngeq B'}$. As $A$ is connected, 
we can extend it to a $u_A$--$x$-path through $A$.
This proves~\eqref{blocki}.

Statement~\eqref{blockii} follows from Lemma~\ref{esslem}: if there was such a path, then we would 
find a $K_4$-subdivision in $G_{\geq B'}$ that makes $B$ essential, in contradiction to that 
$A<B'$ is the baseblock of $B$.

Let us now prove~\eqref{blockiii}.
Suppose that $A'$ is a baseblock. By~\eqref{blocki}, there is then a $u_{A'}$--$x$-path 
in $G_{\geq A'}\cap G_{\ngeq B}$, which implies that there is also a $u_{B}$--$x$-path 
in $G_{\geq A'}\cap G_{\ngeq B}$. This, however, is impossible as~\eqref{blockii} shows.
\end{proof}

The next lemma bounds the number of baseblocks.

\begin{lemma}\label{fewbaseblocks}
There are fewer than  $3k$ distinct baseblocks.
\end{lemma}
\begin{proof}
We prove by induction on $\ell$ that $G$ contains $\ell$ edge-disjoint $K_4$-sub\-di\-visions if $G$
contains at least $3\ell$ distinct baseblocks. By~\eqref{nok}, it then follows that $G$ cannot have $3k$
or more distinct baseblocks. 

Pick a $\leq$-largest baseblock $A_1$, and set $G'=G_{\not\geq A_1}$.
The restriction of the partial order of the blocks of $G-x$ onto the blocks of $G'-x$
is a partial order of the blocks of $G'-x$. 
Thus, we can speak of baseblocks with respect to $G'$.

We claim that $G'-x$ contains at least $3(\ell-1)$ blocks (of $G'-x$) that are baseblocks
with respect to $G'$. Then, by induction, $G'$  contains $\ell-1$ edge-disjoint $K_4$-subdivision. 
Together with one $K_4$-subdivision contained in $G_{\geq A_1}$ 
(that exists as $A_1$ is a baseblock)
we  obtain the desired $\ell$
edge-disjoint $K_4$-subdivisions.  

To prove the claim, consider a  baseblock $A$ (with respect to $G$) with $A\neq A_1$,
and let $B$ be an essential block such that $A$ is its baseblock.
First assume that  $A\not\leq A_1$. From the choice of $A_1$ it then follows 
that $A_1$ and $A$ are incomparable and thus 
\[
G_{\geq A}=G\big[\bigcup_{B'\geq A}B'\cup\{x\}\big]-u_{A}x\subseteq 
G\big[\bigcup_{B'\not\geq A_1}B'\cup\{x\}\big]=G_{\not\geq A_1}=G'.
\]
By definition of a baseblock, there is a
$K_4$-subdivision $K$ that makes $B$ essential (in $G$) and that is contained
in $G_{\geq A}$.
Then $K$  
 is still contained 
in $G'$, and thus still makes $B$ essential in $G'$, 
which in turn means that $A$ is still a baseblock with respect to $G'$.

Next, assume that $A<A_1$. 
Let $A_2$ be the $\le$-largest baseblock such that $A_2<A_1$.
If it exists, let $A_3$ be the $\le$-largest baseblock such that $A_3<A_2$. Otherwise set $A_3:=A_2$.
We show that $A$ is still a baseblock in $G'$ if $A\not\in \{A_2,A_3\}$.
The choice of $A$ implies $A<A_3$.
Lemma~\ref{bblocklem}~\eqref{blockiii} implies $B \ngtr A_3$.
Then, with $A_2$ in the role of $A$, and $A_1$ in the role of $B'$,
we apply Lemma~\ref{bblocklem}~\eqref{blocki} in order to find a $u_{A_2}$--$x$-path
in $G_{\geq A_2}\cap G_{\ngeq A_1}$. 
The path can be extended to a $u_B$--$x$-path $P$ in
$G_{\geq A}\cap G_{\ngeq A_1}$. We follow this by an application of Lemma~\ref{esslem}
to $B,P$ and a $C$ that contains $G_{\geq A_1}-x$. In this way, we find a $K_4$-subdivision 
in $G_{\geq A}\cap G_{\ngeq A_1}$ that makes $B$ essential. 
As a consequence, $A$ is still a baseblock in $G'$.

To conclude, all but at most three of the baseblocks with respect to $G$ (namely all but $A_1$, $A_2$ and $A_3$ if they exist) 
are still baseblocks with respect to $G'$. This proves the claim 
and the lemma. 
\end{proof}

Knowing that $G-x$ only contains few baseblocks 
we aim to bound the number of essential blocks that have the same block as their baseblock. 
This will be done in Lemma~\ref{fewessblocks}. 

For the rest of the section, we call a vertex $v\in V(B)$ of a block $B$ a \emph{gate} of $B$
if $v$ has neighbours outside  $B$.
Note that gates are either cutvertices of $G-x$ or neighbours of $x$.
Furthermore, as $G$ is $2$-connected by~\eqref{blab}, it follows 
for every gate $v$ of $B$ that there is a $v$--$x$-path 
that is internally disjoint from $B$.

\begin{lemma}\label{diamondchain1}
Let $B$ be an essential block with exactly two gates $u,v$, and let $\mathcal P$
be a set of edge-disjoint $u$--$v$-paths contained in $B$. 
Then there is a $u$--$v$-diamond  in $B$ 
that is edge-disjoint from all but at most five of the paths in $\mathcal P$.
\end{lemma}
\begin{proof}
Let $K$ be a $K_4$-subdivision that makes $B$ essential. 
Hence, the unique non-trivial block of $K-x$ lies in $B$.
As $B$ has only two gates, namely $u,v$, it follows that 
$x$ must lie on a subdivided edge of the $K_4$-subdivision $K$, which means that $K\cap B=:D$ is 
a $u$--$v$-diamond.
Let $s,t$ be the two vertices of $D$ that are linked by three internally 
disjoint paths in $D$. Note that $s,t,u,v$ must be four distinct vertices. 
We can express $D$ as $D=D_u\cup D_M\cup D_v$ where for $y\in \{u,v\}$, 
we let $D_y$ be the unique tree in $D$ that contains $y$ and that has leaves $s,t$, 
and where $D_M$ is the $s$--$t$-path in $D$ that is separated from $u$ and $v$ by $\{s,t\}$.

Note that $\{s,t\}$ separates any two of $D_u,D_v,D_M$ in $B$
because otherwise we would find a $K_4$-subdivision in $B\subseteq G-x$.
In particular, $B$ is the edge-disjoint union of connected subgraphs $B_u,B_v,B_M$, where,
for $y\in\{u,v\}$, $B_y$ is connected and contains $y$, and where any two of them meet 
precisely in $s,t$.

We will construct graphs $D'_u\subseteq B_u, D'_v\subseteq B_v,D'_M\subseteq B_M$
such that $D'=D'_u\cup D'_v\cup D'_M$ is a $u$--$v$-diamond that is edge-disjoint from all but
at most five paths in $\cP$.
We start with $D'_M$.
If some path $P$ in $\mathcal P$ shares an edge with $D_M$, then, as $\{s,t\}$ separates $B_M$ 
from the rest of $B$,  
the path $P$ contains an $s$--$t$-subpath $sPt\subseteq B_M$.
In this case put $D'_M=sPt$; if there is no such $P\in\mathcal P$,
put $D'_M=D_M$. 

Next, we construct $D'_u$ in such way that it is edge-disjoint from all but at most two paths of $\cP$. 
The construction of $D'_v$ is analogous.

At least one of the paths in $\mathcal P$ contains an $s$--$u$-subpath that avoids $t$
or a $t$--$u$-subpath that avoids $s$; 
let us assume the former is the case. Define $\mathcal S$ as the set of all $s$--$u$-subpaths
of paths in $\mathcal P$ such that the subpath avoids $t$ 
and put $S=\bigcup_{P\in\mathcal S}V(P)\sm\{s\}$. Observe that both $\mathcal S$
as well as $S$ are non-empty. We also define $\mathcal T$ as the set of all $t$--$\{s,u\}$-subpaths 
of paths in $\mathcal P$ that meet $B_u-\{s,t\}$. To ensure that $\mathcal T$ is also non-empty, 
we add a path $P^*$ to $\mathcal T$ consisting of a single edge between $t$ and a neighbour in $B_u-\{s,t\}$ (such a neighbour exists as $D$ contains one).
We point out that every path in $\mathcal P$ that meets $B_u-\{s,t\}$ is represented by a subpath in $\mathcal S$
or in $\mathcal T$, or both.
Put $T=\bigcup_{P\in\mathcal T}V(P)\sm\{s,t\}$ and observe that $T\neq\emptyset$.

As $B_u-\{s,t\}$ is connected, there is a shortest $S$--$T$-path $Q$ in $B_u-\{s,t\}$.
Pick $P_S\in\mathcal S$
and $P_T\in\mathcal T$ such that both meet $Q$. (Note that possibly $Q$ consists of a single vertex, e.g.~$u$.)
The minimal choice of $Q$ implies that it is edge-disjoint from every $P\in \cP$.
Then $P_S\cup Q\cup P_T\subseteq B_u$ is edge-disjoint from all but at most two paths in $\mathcal P$
(note that this is also the case if $P_T=P^*$).
The graph $P_S\cup Q\cup P_T$ contains a $\subseteq$-minimal tree that has leaves $s$ and $t$
 and that contains $u$;
we pick this tree as $D'_u$. 
\end{proof}

\begin{lemma}\label{moreesslem}
Let $B$ be an essential block, and let $A<B$ be its baseblock. Then:
\begin{enumerate}[\rm (i)]
\item if $B_1,B_2\geq B$ are two essential blocks that have $A$ as a baseblock, then they are comparable; and
\item if $B'> B$ is an essential block that has $A$ as a baseblock, then $B$ has only two gates.
\end{enumerate}
\end{lemma}
\begin{proof}
For~(i), suppose that $B_1$ and $B_2$ are incomparable. Then, by Lemma~\ref{esspathlem}~\eqref{essii},
there is a $u_{B_2}$--$x$-path in $G_{\geq B_2}\subseteq G_{\ngeq B_1}$, 
which we can extend to a $u_{B_1}$--$x$-path 
in $G_{\geq B}\cap G_{\ngeq B_1}$. This, however, violates Lemma~\ref{bblocklem}~\eqref{blockii}
(with $B_2$ in the role of $B$).

For~(ii), suppose that $B$ has three gates, which implies that there are three disjoint $B$--$x$-paths
in $G$. Only one of these can start in $u_{B}$, and only one can pass through $u_{B'}$. Thus, there 
must be a $B$--$x$-path contained in $G_{\geq B}\cap G_{\ngeq B'}$. Extending this to a $u_{B'}$--$x$-path,
we obtain a contradiction to Lemma~\ref{bblocklem}~\eqref{blockii}.
\end{proof}

\begin{lemma}\label{diamondchain2}
Let $B$ be an essential block and $A<B$ be its baseblock and let $\cB$ be the set of essential blocks $B'\ge B$
that have $A$ as baseblock.
If there is a set $\mathcal C$ of  $k+5$ edge-disjoint cycles 
such that each passes through $x$ and contains an edge from $B$, then 
$|\cB|\le k$.
\end{lemma}
\begin{proof}
Suppose that $r:=|\cB|\ge k+1$.
By Lemma~\ref{moreesslem}~(i), the essential blocks in $\mathcal B$ form a chain $B=B^1<\ldots<B^r$.
By Lemma~\ref{moreesslem}~(ii), the blocks $B^1,\ldots,B^{r-1}$ all have at most two gates.

We claim that 
\begin{equation}\label{throughOtherBlocks}
\emtext{
each cycle in $\mathcal C$ shares an edge with  each of $B^1,\ldots,B^{r}$.  
}
\end{equation}
Suppose not, and pick $j$ minimal such that some cycle $C\in\mathcal C$ fails to 
contain an edge of $B^j$. Since $C$ contains an edge of $B=B^1$, it follows that $j\geq 2$.
Denote by $P$ the $u_{B^{j-1}}$--$x$-subpath of $C$ that is contained in $G_{\geq B^{j-1}}$,
and observe that $P$ meets $B^j$ at most in $u_{B^j}$. 
Consequently, we find a 
$u_{B^j}$--$x$-path in $G_{\geq B^{j-1}}\cap G_{\ngeq B^j}$ contrary to Lemma~\ref{bblocklem}~\eqref{blockii}.
This proves the claim. 

Claim~\eqref{throughOtherBlocks} implies that every cycle $C_1,\ldots, C_{k+5} \in \cC$  
passes through every $u_{B^j}$ for $j=1,\ldots, r$.
For $i=1,\ldots, k+5$ and $j=1,\ldots, r-1$, let $P_i^j$ be the $u_{B^j}$--$u_{B^{j+1}}$-subpath of $C_i$ that meets $B^j$.
We apply Lemma~\ref{diamondchain1} to every block $B^j$ for $j=1,\ldots, r-1$
with $\{P_1^j\cap B^j,\ldots, P_{k+5}^j\cap B^j\}$ 
as the set of edge-disjoint paths between the two gates $u_{B^j}$ and $v_{B^j}$ of $B^j$
and obtain a $u_{B^j}$--$v_{B^j}$-diamond $\bar{D}^j$ in $B^j$
that is edge-disjoint from all but at most five paths $P_i^j$.
After renaming the paths $P_i^j$, we may assume that $\bar{D}^j$ is edge-disjoint from $P_{6}^j,\ldots, P_{k+5}^j$.
Then, $D^j=\bar{D}^j \cup v_{B^j}P_1^j$, is a $u_{B^j}$--$u_{B^{j+1}}$ -diamond 
that is edge-disjoint from $P_{6}^j,\ldots, P_{k+5}^j$.

Consider $i\in\{1,\ldots, k\}$, and let 
$Q^i$ be the $u_{B^r}$--$x$-subpath of $C_i$ that is edge-disjoint from $B^1$,
and let $R^i$ be the $u_{B^1}$--$x$-subpath of $C_i$ that is edge-disjoint from  $B^1$.
We set 
\[
L^i= D^i\cup \bigcup_{j=1,j\neq i}^{r-1} P_{i+5}^j  \cup Q^i \cup R^i
\]
Note that, as $r\geq k+1$ there are enough $D^i$ to define all $L^i$.
Observe that every $L^i$ is a $K_4$-subdivision, 
and that any two $L^i$ are edge-disjoint.
We obtain  $k$ edge-disjoint $K_4$-subdivisions in this way,
which we had excluded in~\eqref{nok}. Thus $r=|\mathcal B|\leq k$.
\end{proof}

Let $B$ be an essential block. 
An edge set $F$ \emph{makes $B$ inessential} if $F$ meets 
every $K_4$-subdivision that makes $B$ 
essential.
For the proof of the next lemma we use a result of Mader about $S$-paths:

\begin{theorem}[Mader~\cite{Mad78}]\label{gallaithm}
Let $S$ be a vertex set in a graph $H$. Then there are either $k$ edge-disjoint $S$-paths in $H$
or there is an edge set of size at most $2k-2$ that meets every $S$-path in $H$. 
\end{theorem}

\begin{lemma}\label{fewessblocks}
Let $A$ be a baseblock. Then 
 there is an edge set $F_A$ of size $|F_A|\leq 33k^2$ such that 
all but at most $5k^2$ of the essential blocks with $A$ as baseblock are made
inessential by $F_A$.
\end{lemma}

\begin{proof}
Let $B_1,\ldots, B_N$ be the set of $\leq$-minimal blocks with baseblock $A$
such that $B_i\neq A$ for all $i\in\{1,\ldots, N\}$. In particular
\begin{equation}\label{incompclm}
\emtext{$B_1,\ldots, B_N$ are pairwise incomparable.} 
\end{equation}

Let $Z''$ denote the set of edges 
between $x$ and vertices in $\bigcup_{i=1}^NG_{\geq B_i}$, and let 
$Z'$ be the set of edges incident with $x$
that are not contained in $Z''$. 
For every $i\in\{1,\ldots, N\}$, let $y_i$ be a new vertex. Form $G'$ from $\bigcap_{i=1}^NG_{\not\geq B_i}$
by adding $y_i$ and the edge $e_i=y_iu_{B_i}$ for every $i\in\{1,\ldots, N\}$. 
Set $Y=\{y_1,\ldots,y_N\}$.

We apply the edge-version of Menger's theorem to $Y$ and $x$ in $G'$. Assume first that there are $k$ edge-disjoint
$Y$--$x$-paths $P_1,\ldots, P_k$ in $G'$, where we may assume that $P_i$ starts in $y_i$. 
Then, for $i=1,\ldots, k$, we apply Lemma~\ref{esslem} with $u_{B_i}P_i\subseteq G_{\not\geq B_i}$
and obtain a $K_4$-subdivision that is contained in $G_{\geq B_i}\cup u_{B_i}P_i$. 
By~\eqref{incompclm}, these $k$ $K_4$-subdivisions are all pairwise edge-disjoint, which 
we had excluded in~\eqref{nok}.

So, let us treat the case when there is an edge set $F'$ of size $|F'|\leq k-1$ such that 
any $Y$--$x$-path in $G'$ meets $F'$.
We claim
\begin{equation}\label{F'sep}
\emtext{
if $e_i\notin F'$, then $B_i$ is separated from
$x$ in $G-F'-Z''$. 
}
\end{equation}
Indeed, suppose that $e_i\notin F'$ but there is a $B_i$--$x$-path $P$
in $G-F'-Z''$.
Then, the last edge of $P$ must lie in $Z'$, and the penultimate vertex must lie outside 
any $G_{\geq B_j}$. In particular, $P$ passes through $u_{B_i}$ by~\eqref{incompclm}.
As a result, the interior of  $u_{B_i}P$ lies outside any $G_{\geq B_j}$. Thus $y_iu_{B_i}P$
is a $Y$--$x$-path in $G'$ and  must be met by $F'$. As $e_i\notin F'$,
it follows that $F'$ meets $P$, which is impossible.
Thus, the claim is proved.

Next, we apply Theorem~\ref{gallaithm} to $Y$ in $G'$. Assume first that the theorem 
yields $k$ edge-disjoint $Y$-paths $Q_1,\ldots, Q_k$ in $G'$, where we may assume 
that $Q_i$ starts in $y_{2i-1}$ and ends in $y_{2i}$. 
By Lemma~\ref{esspathlem}~\eqref{essii}, there is 
for every $i\in \{1,\ldots,k\}$ a $u_{B_{2i}}$--$x$-path $R_i$ contained in $G_{\geq B_{2i}}$. 
By~\eqref{incompclm}, the paths $R_i$ are pairwise edge-disjoint,
and since $Q_j\subseteq G'$, the paths $R_i$ are also edge-disjoint
from every path $Q_j$.  Lemma~\ref{esslem} 
yields for every $i\in\{1,\ldots, k\}$ a $K_4$-subdivision 
contained in $G_{\geq 2i-1}\cup u_{B_{2i-1}}Q_iu_{B_{2i}}\cup R_i$. 
These $k$ $K_4$-subdivisions are, again by~\eqref{incompclm},
pairwise edge-disjoint, which contradicts~\eqref{nok} again.

Thus, by Theorem~\ref{gallaithm}, there is an edge set $F''$ of size $|F''|\leq 2k-2$ such that
every $Y$-path in $G'$ meets $F''$. We show:
\begin{equation}\label{exceptJ}
\begin{minipage}[c]{0.8\textwidth}\em
for every $e_i\in F''$ there is at most one $j$ with $e_j\notin F''$ such that 
$F''$ does  not separate $u_{B_i}$ and $u_{B_j}$ in $G-x$.
\end{minipage}\ignorespacesafterend 
\end{equation} 
Suppose that there are two such indices $j$, namely $j'$ and $j''$. Thus, there is
a $u_{B_{j'}}$--$u_{B_i}$-path $P$ in $G-x$ and a $u_{B_i}$--$u_{B_{j''}}$-path $Q$ in $G-x$,
both of which avoid $F''$. By~\eqref{incompclm} it follows that $P,Q\subseteq G'$. 
But then $y_{j'}PQy_{j''}$ contains a $Y$-path in $G'-F''$, which is impossible.

Denote by $J$ the set of all such $j$ as in~\eqref{exceptJ}, and note that $|J|\leq |F''|\leq 2k-2$.
Set $I=J\cup \{i: e_i\in F'\cup F''\}$. 
Note that
\begin{equation}\label{sizeofI}
|I|< 5k\emtext{ and }|F'\cup F''|\leq 3k.
\end{equation}
For $s\in\{1,\ldots, N\}$, we define $\mathcal B_s$
as the set of all essential blocks $B$ with baseblock $A$ such that $B\geq B_s$. 
 The rest of the proof consists of two steps:
First, we prove that for all $r\not\in I$, 
the set $F'\cup F''$ makes all essential blocks $B\in \cB_r$ inessential.
Then we prove that for $r\in I$, there are either few essential blocks $B\in \cB_r$
or there is a small set of edges that makes all $B\in \cB_r$ inessential.

Consider $r\notin I$. That means 
 $r\notin J$ and also $e_r\notin F'\cup F''$.
 Let $K$ be a $K_4$-subdivision
that makes a block $B\in\mathcal B_r$ essential (i.e.\ $B\geq B_r$). 
We show that $K$ contains an edge from $F'\cup F''$, which means that $F'\cup F''$ makes $B$ inessential.

Since $A<B_r$ it follows from Lemma~\ref{bblocklem}~\eqref{blocki} that $K$ contains 
a $u_{B_r}$--$x$-path $R$ in $G_{\not\geq B_r}$. 
If the last edge of $R$, the one incident with $x$, lies in $Z'$, then $R$ is edge-disjoint from $Z''$,
and thus, by~\eqref{F'sep}, meets $F'$. If, on the other hand, the last edge of $R$ lies in $Z''$
then, by~\eqref{incompclm}, $R$ contains a $u_{B_r}$--$u_{B_i}$~path $S$ for some $i\neq r$. 
The path $S$ lies in $G'$ (since the $u_{B_s}$ are cutvertices in $G-x$). 
Since the path $y_ru_{B_r}Su_{B_i}y_i$ has to 
meet $F''$ by definition of $F''$, it follows that either $S$ and thus $R$ meets $F''$, in which 
case we are done, or that $F''$ contains $e_r=y_ru_{B_r}$ (which we had excluded) or that $F''$ contains $e_i=y_iu_{B_i}$.
The latter, however, is also impossible as this would imply $r\in J$. 
Therefore, $F'\cup F''$ makes $B$ inessential. 
We have proved:
\begin{equation}\label{Iiness}
\emtext{
for $r\notin I$,
each essential block in $\mathcal B_r$ is made inessential by $F'\cup F''$.
}
\end{equation}

Now, consider $r\in I$.
First assume that there are $k+5$ edge-disjoint cycles each of which meets $x$ and contains an 
edge of $B_r$. In this case, 
we apply Lemma~\ref{diamondchain2}, obtain $|\cB_r|\le k$ and set $F_r=\emptyset$.
Otherwise, there is an edge set $F_r$ of size $|F_r|\leq k+5$ that separates $u_{B_r}$ 
from $x$ in either $G_{\geq B_r}$ or in $G_{\ngeq B_r}$. 
Consider some $B\in\mathcal B_r$, and a $K_4$-subdivision $K$ that makes $B$ essential.
We see with Lemma~\ref{esspathlem}~\eqref{essii} and~\eqref{essiii}
that $K$ 
contains a $u_{B}$--$x$-path in $G_{\ge B}$ and one in $G_{\ngeq B}$.
In particular, the union of these two paths must be met by $F_r$. 

Thus, we have
for $r\in I$, that either there is an edge set $F_r$ of size at most $k+5$ that makes every $B\in\mathcal B_r$
inessential, or $|\mathcal B_r|\leq k$.
We set 
\[
F_A=F'\cup F''\cup \bigcup_{r\in I} F_r,
\]
and observe with~\eqref{Iiness} that $F_A$ makes all but $|I|\cdot k+1\leq 5k^2$ 
essential blocks with baseblock $A$
inessential (the $+1$ 
is due to the fact that $A$ itself might be essential). With~\eqref{sizeofI} we see that
the of size $F_A$ is at most $3k+5k\cdot (k+5)\le 5k^2+28k\le 33k^2$
\end{proof}

Let $F$ be the union of the $F_A$ as in the previous lemma, where we range over all baseblocks $A$,
and denote by $\mathcal B$ the set of all essential blocks that are not made inessential by~$F$. 
Then, by Lemma~\ref{fewbaseblocks} and Lemma \ref{fewessblocks}
\begin{equation}\label{sizeofF}
|F|\leq 3k\cdot 33k^2\le 100 k^3\emtext{ and }|\mathcal B|\leq 3k\cdot 5k^2=15k^3
\end{equation}

Consider a block $B\in\mathcal B$, i.e.\ an essential block that is not made inessential by $F$. 
Let $U_B$ be the set of gates of $B$. That is, $U_B$ is the union of the neighbours of $x$ in $B$
with the set of cutvertices of $G-x$ that lie in $B$.
For each gate $u\in U_B$, fix a maximal set $\mathcal P_u$ of edge-disjoint $u$--$x$-paths
that are internally disjoint from $B$ and a minimal edge set $F_u$ that meets every such $u$--$x$-path
that is internally disjoint from $B$. By Menger's theorem, we have $|\mathcal P_u|=|F_u|$. 

We form a graph $G_B$ as follows.
Start with $G[B\cup\{x\}]$, and
 for each  $u\in U_B$, add 
$|\mathcal P_u|$ many internally disjoint $u$--$x$-paths of length~$2$; 
let the set of these be $\mathcal P'_u$.
We observe that $G_B$ satisfies the requirements of Lemma~\ref{lem:twoConn}:
the neighbours of $x$ have degree~$2$, 
$G_B-N_{G_B}(x)-x$ is $2$-connected, and  every $K_4$-subdivision of $G_B$ still needs 
to contain $x$. 
Using Lemma \ref{lem:twoConn} we can prove
\begin{lemma}
\label{claim:gb}
For each $B\in\mathcal B$, 
there is an edge set $Y_B\subseteq E(G)$ of size $|Y_B|= \bigO(k^4)$ that makes $B$ inessential (in $G$).
\end{lemma}
\begin{proof}
Suppose there are $k$ edge-disjoint $K_4$-subdivisions in $G_B$. We turn these into $K_4$-subdivisions
of $G$  by substituting 
any path in $\mathcal P'_u$, for $u\in U_B$,
 by a distinct path in $\mathcal P_u$. As $G$ is assumed to have fewer than $k$ edge-disjoint $K_4$-subdivisions,
by~\eqref{nok}, we obtain a contradiction. 

Thus, an application of Lemma~\ref{lem:twoConn} to $G_B$ yields an edge hitting set $Y'_B$ of $G_B$
of size $\bigO(k^4)$. We may assume that $Y'_B$ is minimal subject to inclusion.
As a consequence of minimality, 
for each $u\in U_B$,  whenever $Y'_B$ contains an edge from one of the paths in $\mathcal P'_u$
it contains edges from all of the paths in $\mathcal P'_u$. Note that $|F_u|=|\mathcal  P'_u|$.
We turn $Y'_B$ into an edge set $Y_B\subseteq E(G)$
by replacing any edge in any path from $\mathcal P'_u$, $u\in U_B$, by $F_u$. 
Then $|Y'_B|=|Y_B|$. 

Suppose there is a $K_4$-subdivision $K$ of $G$ that makes $B$ essential but is not met by $Y_B$.
In particular, if $K$ has a subdivided edge that leaves $B$ through a vertex $u\in U_B$ 
(and then continues on to $x$), then 
$F_u$ is disjoint from $Y_B$. Thus, we may replace the part of the subdivided edge between $u$ and $x$
by a path in $\mathcal P_u'$, and by doing so for all subdivided edges that leave $B$ 
through  a vertex in $U_B$, we
obtain a $K_4$-subdivision in $G_B$ that is disjoint from $Y'_B$, which is impossible. 
\end{proof}

We may finally finish the proof of the main lemma of this section.

\begin{proof}[Proof of Lemma \ref{conredlem}]
In view of Lemma~\ref{claim:gb}, the set
\[
Y:=F\cup \bigcup_{B\in\mathcal B}Y_B
\]
meets every $K_4$-subdivision of $G$. 
Recalling~\eqref{sizeofF}, we see that its size is 
\(
|Y|\leq 100k^3 + 15k^3\cdot \bigO(k^4)  = \bigO(k^7).
\)
\end{proof}

\section{Nested ear decompositions}
\label{section:neds}

To finish the proof of our main theorem (Theorem \ref{mainThm}) it remains to prove Lemma~\ref{lem:twoConn},
which we will do in the course of the next three sections. Before we start, however, 
we describe a structural tool for series-parallel graphs that was found by Eppstein~\cite{Epp92}.

An \emph{ear decomposition} of a graph $G$ is a sequence $E_1,\ldots, E_n$
of non-trivial paths 
such that 
\begin{enumerate}[\rm (i)]
\item $G=\bigcup_{i=1}^nE_i$; and
\item $E_i$ is a $\left (\bigcup_{j=1}^{i-1}E_j \right)$-path for every $i=2,\ldots,n$. 
\end{enumerate}
The paths $E_i$ are the \emph{ears} of the ear decomposition. The ear $E_1$ is the \emph{first ear}.
The endvertices of the \emph{first ear} $E_1$ are the \emph{terminals}
of the ear decomposition.

If $E_i$ and $E_j$ are ears with $i<j$
 such that both endvertices of $E_j$ lie in $E_i$ and such that 
no ear $E_{i'}$ with $i'<i$ contains both endvertices of $E_j$,
then $E_j$ is \emph{nested in} $E_i$. If $E_j$ is nested in $E_i$,
then we write $I(E_j)$ for the subpath of $E_i$ between the two endvertices
of $E_j$ and call it the \emph{nest interval} of $E_j$.

\begin{figure}[ht]
\centering
\begin{tikzpicture}
\node[hvertex] (s) at (0,0){};
\node[hvertex] (t) at (8,0){};
\draw[hedge,bend left=10] (s) to node[hvertex,pos=0.1] (p){} node[hvertex,pos=0.2] (q){} node[hvertex,pos=0.3] (r){} node[hvertex,pos=0.4] (rr){} node[hvertex,pos=0.5] (m){} node[hvertex,pos=0.7] (y){} node[hvertex,pos=0.9] (z){} node[pos=0.6,auto]{$E_1$} (t);
\draw[hedge,bend left=80] (s) to node[hvertex,pos=0.6] (a){} node[hvertex,pos=0.9] (b) {} node[pos=0.3,auto]{$E_2$} (m);
\draw[hedge,bend left=80] (y) to node[pos=0.9,label=right:$E_7$]{} (z);
\draw[hedge,bend left=85,auto] (y) to node[pos=0.5]{$E_8$} (t);
\draw[hedge,bend left=80] (a) to node[hvertex,pos=0.6] (c){} node[pos=0.9,auto]{$E_5$} (b);
\draw[hedge,bend left=90] (a) to node[pos=0.5,auto]{$E_6$} (c);
\draw[hedge,bend left=80] (p) to node[pos=0.5,auto]{$E_3$} (q);
\draw[hedge,bend left=80] (r) to node[pos=0.5,auto]{$E_4$} (rr);
\end{tikzpicture}
\caption{A nested ear decomposition}\label{nedfig}
\end{figure}
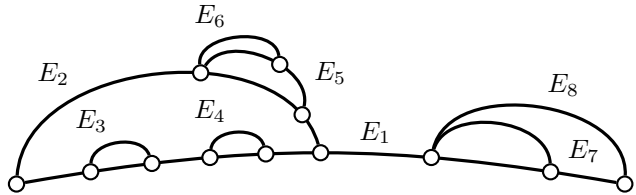

Eppstein~\cite{Epp92} introduced \emph{nested ear decompositions}: these are ear 
decompositions $E_1,\ldots, E_n$ such that 
\begin{enumerate}[\rm (i)]
\item for every $j\in\{2,\ldots, n\}$ 
there is a $i<j$ such that $E_j$ is nested in $E_i$; and
\item if $E_{j'}$ and $E_j$ are both nested in $E_i$, then $I(E_j)$ and $I(E_{j'})$
are either edge-disjoint or one contains the other.   
\end{enumerate}
For the sake of brevity we write \emph{NED} for \emph{nested ear decomposition} from now on.

\begin{theorem}[Eppstein~\cite{Epp92}] \label{eppstein}
A graph with at least one edge is series-parallel with terminals $s,t$ if and only if it admits a 
NED with terminals $s,t$.
\end{theorem}

Whenever we consider a (nested) ear decomposition of a series-parallel graph
we will implicitly assume that the terminals of the decomposition and of the graph 
are the same.\bigskip

We repeat Lemma~\ref{lem:twoConn} here  for the reader's convenience:
\newtheorem*{lem4}{Lemma~\ref{lem:twoConn}}

\begin{lem4}
Let $x$ be a vertex in a graph $G$ such that every $K_4$-subdivision in $G$ meets $x$, and let
each  neighbour of $G$ have degree~$2$. If $G-N(x)-x$ is $2$-connected, then 
$G$ either contains $k$ edge-disjoint $K_4$-subdivisions
or an edge  set of size $\bigO(k^4)$ that meets every $K_4$-subdivision.
\end{lem4}

We fix $G$ as in the lemma 
throughout the rest of the article. We also define the vertex set $X$ as $N_G(x)\cup\{x\}$.
Moreover, we note that, as $G-X$ is $2$-connected and does not contain any $K_4$-subdivision,
it follows by Lemma \ref{diamonds} (i) that $G-X$ is series-parallel (with some terminals). 
In conclusion, we will use the following properties throughout the remainder of this article:
\begin{equation}\label{standardG}
\begin{minipage}[c]{0.8\textwidth}\em
$G$ has a vertex $x$, the set $N(x)\cup\{x\}$ is $X$, 
all vertices in $X\sm\{x\}$ 
have degree~$2$, $G-X$ is $2$-connected and series-parallel, but $G$ contains a
 $K_4$-subdivision.
\end{minipage}\ignorespacesafterend 
\end{equation} 


For the proof of Lemma~\ref{lem:twoConn}, we will often work in certain subgraphs $H$ of $G$ such 
that $H-X$ is series-parallel. To simplify notation somewhat we will say that a $\cE$
is a \emph{NED of $H$} if it is one of $H-X$ (if $H-X$ contains an edge). 
Let $F$ be an ear of such a NED $\cE$ of some subgraph $H$ of $G$ (such that $H-X$ is
series-parallel with the same terminals as $\cE$). Let the endvertices of $F$ be $u$ and $v$. 
We define $\overline F$ as the union of $F$ with all $(F-\{u,v\})$--$x$-paths in $G$ of length~$2$. 
That is, to form $\overline F$ we start with $F$, then we add all edges between $F-\{u,v\}$ and $X$
(together with the corresponding endvertices in $X$) and finally we add all edges between the newly
added vertices in $X$ and $x$ (together with $x$ if we added any edge at all).

An ear $F$ of $\mathcal E$ is an \emph{$x$-ear} if it contains a vertex from $N(X)$ in its interior, 
or equivalently, if $\overline F$ contains $x$. 
The NED $\mathcal E=E_1,\ldots,E_n$ is \emph{good} if 
\begin{enumerate}[\rm (i)]
\item among all NEDs of $H$ the number of $x$-ears in $\mathcal E$ is maximal; and
\item subject to~(i) $x$-ears appear as early as possible, that is, 
the binary number $b_1\ldots b_n$ is maximal, where
 $b_i=1$ if $E_i$ is an $x$-ear and $b_i=0$ otherwise.
\end{enumerate}

In the next lemma we collect some properties of good NEDs.
\begin{lemma}\label{goodNEDlem}
Let $H\subseteq G$ be a subgraph such that $H-X$ is series-parallel and contains at least one edge.
Let $\mathcal E$ be a good NED of $H$, and let $E, F$ be ears of $\mathcal E$
such that $F$ is nested in $E$. Then the following holds:
\begin{enumerate}[\rm (i)]
\item Let $P$ be the interior of $I(F)$. 
If $E$ is an $x$-ear but $F$ is none, then either $P$ is disjoint from $N(X)$
or $P$ contains all of $N(X)\cap E$ except possibly the endvertices of~$E$; 
\item If $F$ is an $x$-ear, then so is $E$.
\end{enumerate}
\end{lemma}
\begin{proof}
Let $\mathcal E=E_1,\ldots, E_N$, and let $E=E_i$ and $F=E_j$, which implies $i<j$.

(i) Suppose that $P$ contains a vertex from $N(X)$ and that, at the same time, 
some vertex of the interior of $E_i$ lies in $N(X)\sm V(P)$. 
Let $u,v$ be the endvertices of $E_j$. 
Form the paths $E'_i=E_iuE_jvE_i$ and $E'_j=uE_iv=I(E_j)$, and observe that both contain a vertex from $N(X)$
in their respective interiors. Moreover, the sequence
\[
\mathcal E'=E_1,\ldots, E_{i-1},E'_i,E'_j,E_{i+1},\ldots ,\hat{E_j},\ldots,E_n,
\]
where $\hat{E_j}$ indicates that $E_j$ is omitted, is a NED with more $x$-ears than $\mathcal E$, which is impossible. 

(ii) Suppose that $E_j$ is an $x$-ear but $E_i$ is not. As in~(i), form the paths
$E'_i=E_iuE_jvE_i$ and $E'_j=uE_iv=I(E_j)$. Observe that now $E'_i$ contains a vertex of $N(X)$
in its interior. Thus, the $x$-ears of the NED $\mathcal E'$ from (i) appear earlier than in $\cE$, 
contradicting the choice of $\cE$.
\end{proof}

If $H-X$ contains an edge, we define the \emph{$x$-ear number} of  $H$
as the number of $x$-ears in a good NED of $H$  
(by definition the number of $x$-ears is the same in all good NEDs of $H$
and equals the maximal number of $x$-ears in any NED of $H$).
If $H-X$ consists of a single vertex,
we define the $x$-ear number of $H$ as~$0$.

We use the $x$-ear number of $G$ to distinguish between two major cases in the 
proof of Lemma~\ref{lem:twoConn}: If the $x$-ear number is large, we will always
find $k$ edge-disjoint $K_4$-subdivisions; if it is small, then both outcomes will still be 
possible, but we will be able to do an induction on the $x$-ear number.
We first treat the case when the $x$-ear number is large.

\section{Many $x$-ears}
\label{largeS}

We prove in this section:
\begin{lemma}\label{manylem}
If the $x$-ear number of $G$ is at least $200k$, for some integer $k$,
 then $G$ contains $k$ edge-disjoint $K_4$-subdivisions.
\end{lemma}

Let $H\subseteq G$ be a subgraph of $G$ such that $H-X$ is series-parallel
and let $\cE$ be a good NED of $H$.
By Lemma~\ref{goodNEDlem}~(ii) the union of all $x$-ears in $\cE$ is a connected graph,
i.e.\ every $x$-ear except for the first ear is nested in another $x$-ear.
Whether an $x$-ear is nested in another $x$-ear or not defines a relation whose 
transitive closure is 
a partial order $\nestle$ on the $x$-ears. 
The first ear is always the unique minimal element in this partial order.
If $F\nestle F'$ holds for two $x$-ears $F$ and $F'$, the ear $F'$ is a \emph{$\nestle$-descendant} of $F$.
Note that we define the $\nestle$-relation only for $x$-ears.

\begin{lemma}\label{alleNebeneinander}
Let $F_1,\ldots, F_{\ell}$ be $x$-ears that are nested in an ear $E$
 such that all nest intervals are edge-disjoint.
Let $P$ be the minimal subpath of $E$ that contains 
$\bigcup_{i=1}^\ell I(F_i)$.
Then $\bigcup_{i=1}^\ell\overline F_i\cup P$ contains $\lfloor \frac{\ell}{3}  \rfloor$ 
edge-disjoint $K_4$-subdivisions.
\end{lemma}
\begin{proof}
We show the statement first for $\ell=3$. The general statement then follows by 
partitioning $F_1,\ldots,F_\ell$ into groups of three with consecutive nest intervals. 

For $i=1,2,3$, let $s_i,t_i$ be the endvertices of $F_i$ 
and assume that they appear in the order $s_1<t_1\le s_2<t_2\le s_3<t_3$ on $E$.
For $i=1,2,3$, let $v_i$ be a vertex in $F_i\cap N(X)$
and $y_i$ the unique neighbour of $v_i$ in $X$.
Let $C$ be the cycle $F_2\cup I(F_2)$.
Then, there are three $C$--$x$-paths $P_1,P_2,P_3$ that only meet in $x$,
namely $P_1=s_2Et_1F_1v_1y_1x$, $P_2=v_2y_2x$ and $P_3=t_2Es_3F_3v_3y_3x$.
The union $C\cup P_1\cup P_2\cup P_3$ is a $K_4$-subdivision 
contained in $P\cup \overline F_1 \cup \overline F_2 \cup \overline F_3$ for $P=s_1Et_3$;
see Figure~\ref{fig:dreiNebeneinander}.
\end{proof}

\begin{lemma}
\label{threeAreEnough}
Let $H$ be a subgraph of $G$ such that $H-X$ is series-parallel and $\cE$ be a good NED of $H$.
Let $E$ be an $x$-ear with at least seven $\nestle$-descendants. Then $H$ contains a $K_4$-subdivision.
\end{lemma}
\begin{proof}
Let $\cF$ be the set consisting of $E$ and its $\nestle$-descendants.
If there is an ear $F\in \cF$
that has three distinct ears $F_1,F_2,F_3\in \cF\setminus \{E\}$ nested in $F$,
their nest intervals are either edge-disjoint (Figure~\ref{fig:dreiNebeneinander})
or the nest intervals are nested (Figure~\ref{fig:ineinander}).
In the first case we find a $K_4$-subdivision with $x$ as branch vertex by Lemma~\ref{alleNebeneinander}.
In the second case, let $F_2$ be the ear with $I(F_2)\subseteq I(F_1)$.
Set $C=F_2\cup I(F_2)$ and let $v$ be an $X$-neighbour in the interior of $F_1$.
Then, there are three $C$--$v$-paths that only meet in $v$,
namely one from an $X$-neighbour in $F_2$ via $x$
and two from the endvertices of $F_2$ via $I(F_1)-I(F_2)$ and $F_1-v$.
Hence, the union of $C$ and these paths is a $K_4$-subdivision.

We therefore may assume that every $F\in\cF$ has at most two $\nestle$-descendants.
If there were no ears $F_1,F_2,F_3$ such that $F_1$ is nested in $E$,
$F_2$ is nested in $F_1$ and $F_3$ is nested in $F_2$,
then $E$ would have at most six $\nestle$-descendants.
Therefore, those ears exist (looks similar as Figure~\ref{fig:stacked}, however some paths there could be trivial) 
and we can construct a $K_4$-subdivision in 
$\overline{F_1}\cup \overline{F_2}\cup \overline{F_3}$ as follows.

Let $C=F_2\cup I(F_2)$.    
Let $u$ be a neighbour of $X$ in $F_2$ and let $P_1$ be the $u$--$x$-path of length 2.
Then, there is an endvertex $v$ of $F_3$ with $v\neq u$ and let $P_2$ start in $v$, follow $F_3$ until it reaches a neighbour of $X$ in $F_3$ and then continues via $X$ to $x$.
If $I(F_2)$ contains a neighbour of $X$, let $w$ be this neighbour and $P_3$ be the $w$--$x$-path of length 2.
Otherwise, $F_1-I(F_2)$ contains a neighbour of $X$ and we chose $w$ as an endvertex of $F_2$ closest to this neighbour of $X$.
In this case, we let $P_3$ start in $w$, follow $F_1$ up to this neighbour of $X$ and finally reach $x$ via $X$.
Then, the paths $P_1,P_2,P_3$ are $C$--$x$-paths that only meet in $x$
and therefore, $C\cup P_1\cup P_2\cup P_3$ is a $K_4$-subdivision.
\end{proof}

\begin{figure}
\tikzstyle{xed}=[hedge,color=dunkelgrau]
\centering
\begin{subfigure}[b]{0.45\textwidth}
\begin{tikzpicture}
\node[hvertex] (a) at (0,0){};
\node[hvertex] (b) at (5,0){};
\draw[hedge,bend left=10] (a) to 
node[hvertex,pos=0.1] (a2){} 
node[hvertex,pos=0.3] (a3){} 
node[hvertex,pos=0.4] (a4){}
node[hvertex,pos=0.6] (a5){}
node[hvertex,pos=0.7] (a6){}
node[hvertex,pos=0.9] (a7){} 
(b);
\draw[hedge,bend left = 70] (a2) to node[hvertex,pos=0.5](c1){} (a3);
\draw[hedge,bend left = 70] (a4) to node[hvertex,pos=0.5](c2){} (a5);
\draw[hedge,bend left = 70] (a6) to node[hvertex,pos=0.5](c3){} (a7);

\node[hvertex,label=right:$x$](x) at (3,1.5){};
\draw[xed] (c1)--(x) node[midway,hvertex]{};
\draw[xed] (c2)--(x) node[midway,hvertex]{};
\draw[xed] (c3)--(x) node[midway,hvertex]{};
\end{tikzpicture}
\caption{Side by side}
\label{fig:dreiNebeneinander}
\end{subfigure}
\begin{subfigure}[b]{0.45\textwidth}
\begin{tikzpicture}
\node[hvertex] (a) at (0,0){};
\node[hvertex] (b) at (5,0){};
\draw[hedge,bend left=10] (a) to 
node[hvertex,pos=0.2] (a2){} 
node[hvertex,pos=0.4] (a3){} 
node[hvertex,pos=0.6] (a4){}
node[hvertex,pos=0.8] (a5){}
(b);
\draw[hedge,bend left = 70] (a2) to node[hvertex,pos=0.5](c1){} (a5);
\draw[hedge,bend left = 70] (a3) to node[hvertex,pos=0.5](c2){} (a4);

\node[hvertex,label=right:$x$](x) at (4,1.5){};
\draw[xed] (c1)--(x) node[midway,hvertex]{};
\draw[xed,bend right=20] (c2) to node[near end,hvertex]{} (x);
\end{tikzpicture}
\caption{Nested}
\label{fig:ineinander}
\end{subfigure}
\begin{subfigure}[b]{0.45\textwidth}

\begin{tikzpicture}
\node[hvertex] (a) at (0,0){};
\node[hvertex] (b) at (4,0){};
\draw[hedge,bend left=10] (a) to node[hvertex,pos=0.1] (c){}  node[hvertex,pos=0.9] (d){} (b);
\draw[hedge,bend left = 70] (c) to node[hvertex,pos=0.3](e){} node[hvertex,pos=0.7](f){} node[hvertex,pos=0.85](x1){} (d);
\draw[hedge,bend left = 90] (e) to node[hvertex,pos=0.2](g){} node[hvertex,pos=0.65](x2){} node[hvertex,pos=0.8](h){} (f);
\draw[hedge,bend left = 90] (g) to node[hvertex,pos=0.5](x3){} (h);

\node[hvertex,label=right:$x$](x) at (4,2){};
\draw[xed] (x)--(x1) node[midway,hvertex]{};
\draw[xed] (x)--(x3) node[midway,hvertex]{};
\draw[xed] (x)--(x2) node[midway,hvertex]{};
\end{tikzpicture}
\caption{Stacked}
\label{fig:stacked}
\end{subfigure}
\caption{$K_4$-subdivisions}
\label{fig:fewDescendants}
\end{figure}
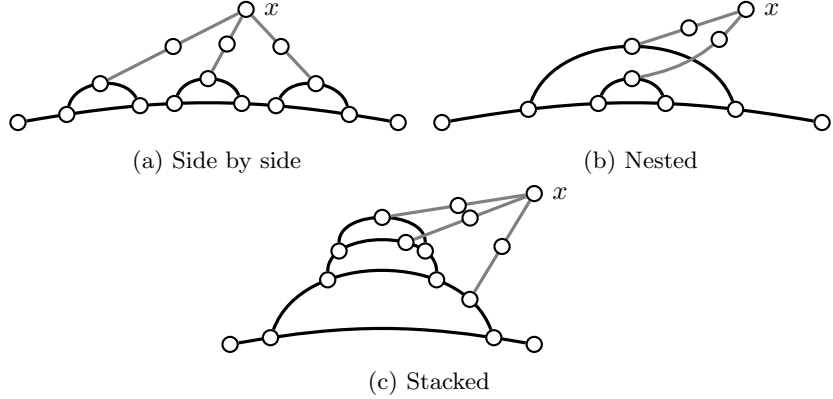

Let $\mathcal E$ be a good NED of some subgraph $H$ of $G$.
We refine the $\nestle$-relation  on the set of $x$-ears to a relation $\leq$.
We  first define the $\le$-relation on specific pairs $(F,F')$ of $x$-ears,
and then take  the transitive closure of the relation.
For two distinct $x$-ears $F,F'$ of $\mathcal E$ set 
\begin{equation}
\begin{minipage}{0.7\textwidth} \label{poset}
$F < F' \quad$ if $F'$ is nested in $F$ 
or there is an ear $E$ such that $F'$ and $F$ are nested in $E$ with $I(F')\subsetneq I(F)$
or $I(F')=I(F)$ but $F$ appears earlier in $\cE$ than $F'$.
\end{minipage}
\end{equation}
Let $\le$ denote the partial order obtained from the transitive closure of \eqref{poset}.
If $F<F'$, the $x$-ear $F'$ is a \emph{$\le$-descendant} of $F$. 

\begin{lemma}
\label{lem:manyears}
Let $H$ be a subgraph of $G$ that contains a $K_4$-subdivision and such that $H-X$ is series-parallel.
Let $\mathcal E$ be a good NED of $H$ with $\lambda$ $x$-ears.
Then $\bigcup_{F\in\mathcal E}\overline F$ 
contains at least $\lfloor \frac{\lambda}{200} \rfloor+1$ edge-disjoint
$K_4$-subdivisions.
\end{lemma}
\begin{proof}
Let $E$ be the first ear in $\cE$.
Let $\ell=\lfloor \frac{\lambda}{200} \rfloor+1$, hence $\lambda=200(\ell-1)+r$ with $0\le r\leq 199$. 

If $\ell=1$, the statement holds because $H$ contains at least one $K_4$-subdivision, by assumption.
Hence, we assume $\ell\ge 2$ and thus, $\lambda\ge 200$ holds. 
Among all $x$-ears $F$ that have at least 13 $\le$-descendants (the first ear $E$ is one of them),
let $E^*$ be a $\le$-maximal one.
As a result, every $\le$-descendant of $E^*$ has at most 12 $\le$-descendants itself.

If 
$E^*$ has at least  $78\ell$ $\leq$-descendants,
then it has at least $6\ell$ \emph{immediate} $\le$-descendants of $E^*$ (those $F$ such that $E^*< F' \le F$ implies $F'=F$)
because every immediate $\le$-descendant has at most 12 $\le$-descendants itself.
If at least $3\ell$ of the immediate $\leq$-descendants of $E^*$ are nested in $E^*$, set $E':=E^*$.
If not, then let $E'$ be the ear in which $E^*$ is nested. Then there 
 must be at least $3\ell$ immediate $\leq$-descendants of $E^*$
that are nested in $E'$. In both cases, 
as  they are immediate $\le$-descendants of $E^*$, their nest intervals have to be edge-disjoint.
We can apply Lemma \ref{alleNebeneinander} and obtain $\ell$ edge-disjoint $K_4$-subdivisions in $H$.
Thus, we may assume that $E^*$ has fewer than $78\ell$ $\le$-descendants.

Next, consider the case when  $E^*$ has at least seven $\nestle$-descendants.
Define $\mathcal E_1$ as the subsequence of $\cE$
of all $F\in\mathcal E$ with $F\gestle E^*$, 
and $\cE_2$ as the subsequence of those ears $F\in\mathcal E$ with $F\not\gestle E^*$.
Define edge-disjoint subgraphs $H_1,H_2$ of $H$ as $H_i=\bigcup_{F\in\mathcal E_i}\overline F$ for $i=1,2$.
Note that $\mathcal E_i$ is a NED of $H_i$ for $i=1,2$. 
We check that both subgraphs contain a $K_4$-subdivision.
By  Lemma~\ref{threeAreEnough}, $H_1$ contains a $K_4$-subdivision.
The number of $x$-ears of $\cE$ that are contained in $H_2$ is at least 
\[
200(\ell-1)-78\ell=122\ell-200 \ge 44
\]
because $\cE$ contains at least $200(\ell-1)$ $x$-ears and $E^*$ has at most $78\ell$ $\le$-descendants
and $\ell\ge 2$. 
Thus, the first ear of $\mathcal E_2$ has at least seven $\le$-descendants (in $\mathcal E_2$)
and these are even $\nestle$-descendants because it is the first ear.
By Lemma \ref{threeAreEnough}, $H_2$ contains a $K_4$-subdivision, as well.

We apply induction to $H_1$ and $H_2$.
The $x$-ear number $\lambda_i$ of $H_i$ is at least 
as large as the number of $x$-ears of $\cE_i$, which implies
 $\lambda_1+\lambda_2\ge \lambda$.
In total, we obtain
\[
\left\lfloor\frac{\lambda_1}{200}\right\rfloor + 1 + \left\lfloor\frac{\lambda_2}{200}\right\rfloor + 1 \ge \left\lfloor\frac{\lambda_1+\lambda_2}{200}\right\rfloor + 1 \ge \left\lfloor\frac{\lambda}{200}\right\rfloor + 1 = \ell
\]
edge-disjoint $K_4$-subdivisions.

Next, consider the case when  $E^*$ has at most six $\nestle$-descendants. 
Then, $E^*$ is not the first ear but nested in some other $x$-ear $E'$.
We decompose $\cE$ into two NEDs $\cE_1$ and $\cE_2$ 
belonging to two edge-disjoint subgraphs $H_1$ and $H_2$.
Let $\cE_1$ consist of $I(E^*)$ as first ear 
and all $F\in \cE$ with $F\geq E^*$ but $F\not\gestle E^*$, in the same order as in $\cE$.
For $\mathcal E_2$, let $E''$ be the ear obtained from $E'$ by replacing the subpath $I(E^*)$
with $E^*$. We then set $\mathcal E_2=\mathcal E\sm (\mathcal E_1\cup\{E'\})\cup \{E''\}$ (where keep the order of $\cE$). 
The graphs $H_1$ and $H_2$ are defined as $H_i=\bigcup_{F\in \cE_i} \overline{F}$. 

Again, the $x$-ear numbers $\lambda_1$ and 
$\lambda_2$ of $H_1$ and $H_2$ are at least the number of $x$-ears in $\cE_1$ 
resp.\ $\cE_2$. 
Since $E^*$ was chosen to have at least 13 $\leq$-descendants, and since it has, in this case, at most 
six $\nestle$-descendants, 
it follows that $\mathcal E_1$ contains $I(E^*)$ as first ear 
and at least seven $x$-ears from $\cE$.
These $x$-ears are $\nestle$-descendants of $I(E^*)$ in $\cE_1$. 
Furthermore, as $E^*$ contains fewer than $78\ell$ $\le$-descendants, the sequence $\cE_1$ contains fewer than $78\ell$ $x$-ears from $\cE$.
Together with $I(E^*)$ as new $x$-ear in $\cE_1$, we have $\lambda_1\le 78\ell$.
Every $x$-ear of $\cE\setminus\{E^*,E'\}$ appears either in $\cE_1$ or $\cE_2$.
Thus, $\cE_2$ contains at least $200(\ell-1) - (78\ell-1)-2 = 122\ell-201 \ge 43$ $x$-ears from $\cE$
where we used $\ell\ge 2$.
Hence, both $\cE_1$ and $\cE_2$ contain at least seven $\nestle$-descendants of their respective first ear,
and Lemma~\ref{threeAreEnough} shows that both subgraphs $H_1$ and $H_2$ contain a $K_4$-subdivision.

Every $x$-ear of $\cE$ except for $E^*$ and $E'$ appears either in $\cE_1$ or $\cE_2$.
We lose the two $x$-ears $E^*$ and $E'$ and gain the new $x$-ear $E''$.
The other new ear, $I(E^*)$, however, is not necessarily an $x$-ear. 
Thus $\lambda_1+\lambda_2\ge \lambda-1$.

We prove that nevertheless $
\left\lfloor\frac{\lambda_1}{200}\right\rfloor + 1 + \left\lfloor\frac{\lambda_2}{200}\right\rfloor + 1 \le \left\lfloor\frac{\lambda_1+\lambda_2}{200}\right\rfloor + 1 = \left\lfloor\frac{\lambda}{200}\right\rfloor + 1
$ holds.
Let $\ell_i=\left\lfloor\frac{\lambda_i}{200}\right\rfloor + 1$, i.e. $\lambda_i=200(\ell_i-1)+r_i$ with $0\le r_i\le 199$, for $i=1,2$.
Then $\lambda\le \lambda_1+\lambda_2+1\le 200(\ell_1+\ell_2-2)+(r_1+r_2+1)$.
As $r_i\le 199$, we have $r_1+r_2+1<2\cdot 200$ which means $\left\lfloor \frac{\lambda}{200} \right\rfloor + 1 \le (\ell_1+\ell_2-2)+1+1=\ell_1+\ell_2$.
Thus, applying induction on $H_1$ and $H_2$ yields $\ell_1+\ell_2\ge \left\lfloor \frac{\lambda}{200} \right\rfloor + 1$ edge-disjoint $K_4$-subdivisions,
which finishes the proof.
\end{proof}

We finally observe that, because $G$ contains a $K_4$-subdivision by~\eqref{standardG},
Lemma~\ref{manylem} is a special case of Lemma~\ref{lem:manyears}, which means that it is proved as well. 
In particular, we may from now on assume that the $x$-ear number of $G$ is smaller than $200k$.\sloppy

\section{Few $x$-ears}
\label{smallS}

In the previous section we saw that the $x$-ear number of $G$ is smaller than $200k$.
We will next find a hitting set for $K_4$-subdivisions 
whose size is bounded by some function in the $x$-ear number $\lambda$.
We construct the hitting set by induction on $\lambda$,
which means we will need to decompose the graph into
\emph{parts} with smaller $x$-ear number. 


Recall that $G$ is fixed and satisfies \eqref{standardG}.
The main lemma in this section is the following.
\begin{lemma}\label{mainfewlem}
Let $\lambda$ be the $x$-ear number of $G$, and let $k$ be a positive integer. 
Then $G$ either contains $k$ edge-disjoint $K_4$-subdivisions or 
$G$ has an edge hitting set of size $\bigO(\lambda k^3)$.
\end{lemma}

With Lemma \ref{mainfewlem} we can finally prove Lemma \ref{lem:twoConn}.

\begin{proof}[Proof of Lemma \ref{lem:twoConn}]
Let $\lambda$ be the $x$-ear number of $G-x$.
If $\lambda\ge 200k$, Lemma \ref{manylem} shows that $G$ contains $k$ edge-disjoint $K_4$-subdivisions.
Hence, $\lambda<200k$ and Lemma \ref{mainfewlem} shows
that $G$ either contains $k$ edge-disjoint $K_4$-subdivisions
or a hitting set of size at most $\bigO(\lambda k^3) = \bigO(k^4)$.
\end{proof}

The proof of Lemma~\ref{mainfewlem} will occupy the rest of the article. 
In the next section we
  investigate how the $x$-ear number relates to the natural decomposition
of $G-X$ as a series-parallel graph. In Section~\ref{laddersfans}, 
we take a short detour and discuss two configurations that will 
play a role in the hardest case of the induction. In Section~\ref{moduleSection}
we look at the traces of $K_4$-subdivisions in an arbitrary 
part of the decomposition of $G-X$. 

\subsection{Parts and ears}
\label{sec:parts}

Let $H$ be a subgraph of $G-x$. 
Any vertex in $H$ that is incident with an edge $e\in E(G-x)\setminus E(H)$ is a \emph{terminal} of $H$.
If $H=G-x$, then $H$ contains no such vertices
but by \eqref{standardG}, there are two vertices $s,t$ such
that $G-x$ is series-parallel with terminals $s,t$.
Then, we use $s,t$ as terminals of $H=G-x$.

The graph $H$ is a \emph{part} (of $G$) 
if it has at most two terminals. 
This means that a part $H$ with terminals $s,t$ is  an induced subgraph except for possibly the edge $st$, which may 
be a chord of $H$.
A part $H$ is \emph{trivial}
if $H-X$ consists of a single vertex (the terminal).
Two parts $H_1,H_2$ are \emph{internally disjoint} if $H_1$ and $H_2$ meet at most in their terminals
and have no common edge. 
Using the facts that $G-x$ is connected and $G-X$ is $2$-connected from \eqref{standardG}
we deduct
\begin{equation}\label{partProps}
\begin{minipage}[c]{0.8\textwidth}\em
A part is a connected graph, it has at least one terminal and the only parts with exactly one terminal are trivial parts.
\end{minipage}\ignorespacesafterend 
\end{equation} 
%
%


Mimicking the notions of series-parallel graphs, we define series and parallel decomposition for parts as well.
Let $H$ be a part with terminals $s,t$, and let $H_1,H_2$ be parts, too. 
Then $H=H_1\cup H_2$ is a \emph{series decomposition} of $H$ into parts 
if $H_1$ and $H_2$ are internally disjoint and
have a common terminal $w$ such that $H_1$ has $s,w$ as terminals and $H_2$ has $w,t$ as terminals. 
Note that we explicitly allow $H_1$ or $H_2$ (or both) to be trivial parts, which entails $s=w$ or $w=t$.

We say that $H=H_1\cup H_2$ is a \emph{parallel decomposition} of $H$ into parts if $H_1$ and $H_2$
are internally disjoint and both have $s,t$ as terminals.

A part $H$ with terminals $s,t$ is \emph{substantial} if it contains a vertex in $X$ that is adjacent
to a vertex in $H-\{s,t\}$. In particular,  $H-X$ must contain more 
than a single edge.
A part $B\subseteq H$ is a \emph{block-part} of $H$ if $B-X$ is a block of $H$.

\begin{lemma}\label{partsplem}
If $H$ is a part with terminals $s,t$, then $H-X$ is series-parallel with terminals $s,t$. 
\end{lemma}
\begin{proof}
If $H-X$ is a single vertex $s=t$ or a single edge $st$, it is trivially series-parallel with terminals $s,t$.

If $H=G-x$, then the definition of the terminals of $H$ imply the statement of the lemma. 
We therefore assume that $H\subsetneq G-x$.

Let $s,t$ be the terminals of $H$ and let $H':=(H-X)+st$.
Assume that $H'$ is not $2$-connected and let $v$ be a cutvertex in $H'$.
In $H'-v$, the adjacent vertices $s$ and $t$ belong to the same component 
and are separated by $v$ from some other component $F$.
As $\{s,t\}$ separates $H-X$ from $G-H-X$, the vertex $v$ separates $F$ from the rest of $G-X$.
This, however, contradicts the fact that $G-X$ is $2$-connected, by \eqref{standardG}.
We conclude that $H'$ is $2$-connected.

Assume that $H'$ contains a $K_4$-subdivision $K$.
If $H'=H$, the graph $G-x\supsetneq H$ would contain one as well which contradicts \eqref{standardG}.
If $H'\neq H$, the fact that $G-x$ is 2-connected by \eqref{standardG}
and $\{s,t\}$ separates $H$ from $G-E(H)$
imply that there is an $s$--$t$-path $P$ in $G-x$ that is internally disjoint from $H$.
Replacing the edge $st$ in $K$ by $P$ yields a $K_4$-subdivision in $G-x$ which is a contradiction to \eqref{standardG}.
We conclude that $H'$ contains no $K_4$-subdivision.

Thus, by Lemma~\ref{diamonds}~(i), it follows that $H'$ with terminals $s,t$ is series-parallel with terminals $s,t$
because $st$ is an edge in $H'$.
Obviously, then also $H-X=H'-st$ is series-parallel with terminals $s,t$.
\end{proof}

Let $H$ be a part such that $H-X$ contains an edge. 
By Lemma \ref{partsplem} and Theorem \ref{eppstein}, 
$H-X$ has a (good) NED. 
As mentioned in Section \ref{section:neds}, we also say that $H$ has a NED $\cE$ if it is one of $H-X$.
Let $H=H_1\cup H_2$ be a series or parallel decomposition into parts $H_1$ and $H_2$.
Given a good NED $\cE=(E_1,\ldots, E_n)$ of $H$, we define an \emph{induced  NED} $\cE_j$ of $H_j$
as the subsequence of $E_1\cap H_j,\ldots, E_n\cap H_j$ of all $E_i\cap H_j$ that are non-trivial
paths. 
The NEDs $\cE_j$ do not have to be good.
We define the \emph{$\cE$-induced $x$-ear number} as the number of $x$-ears in $\cE_j$.
Note that  $H_i$ is substantial if and only if its induced $x$-ear number is at least~$1$.

\begin{lemma}\label{pushoutlem}
Let $H$ be a part with terminals $s,t$ and $x$-ear number~$\lambda$.
If there is an $s$--$t$-path $P$ in $H$ that does not contain any vertex from $N(X)$ in its interior,
then there is a NED $\mathcal E$ of $H$ with $\lambda$ $x$-ears that has $P$ as first ear.
\end{lemma}
\begin{proof}
Among all NEDs of $H$ with $\lambda$  $x$-ears choose a NED $E_1,\ldots, E_m$
such that the largest $r$ such that $E_r$ contains an edge 
of $P$ is as small as possible. 

Let $Q$ be a maximal subpath of $E_r$ with $Q\subseteq P$, and suppose that $Q\neq E_r$.
Then $Q$ has an endvertex $u$ that is not an endvertex of $E_r$. 
Moreover, $u\notin\{s,t\}$ as $u$ is an internal vertex of $E_r$.
Thus, $u$ is incident with an edge of $P$ outside $E_r$, which then 
already belongs to $\bigcup_{\ell=1}^{r-1}E_\ell$ by choice of $r$. 
On the other hand, this edge is incident with an inner vertex of $E_r$, 
which implies that it belongs to an ear nested in $E_r$.
Therefore, the edge cannot be contained in $\bigcup_{\ell=1}^{p-1}E_\ell$.
Because of this contradiction, we deduce that 
\begin{equation}\label{inside}
E_r\subseteq P.
\end{equation}

Suppose that $r\geq 2$.
Let   $u,v$ be the endvertices of $E_r$, and let $E_r$ be nested in $E_q$. 
Then omitting $E_r$ from the sequence
\[
E_1,\ldots, E_{q-1},E_{q}uPvE_{q},uE_{q}v,E_{q+1},\ldots,E_m
\]
results in a NED  of $H$. 

Observe that the number of $x$-ears does not decrease as $uPv=E_r$
is not an $x$-ear by assumption. 
If $uE_qv$ is edge-disjoint from $P$, we immediately obtain a contradiction to 
the minimal choice of $r$. 
Thus, assume that $uE_qv$ contains an edge from $P$.
As $E_r\subseteq P$ has endvertices $u,v$ and $P$ does not contain a cycle,
it follows that $uE_qv$ contains an 
internal vertex that is incident with an edge $e$ that lies in $P$ but outside
$E_q$. Then, $e$ must lie in an ear $E_p$ with $p>q$. Since $e$ lies in $P$
we also get $p<r$. Thus, $q<p<r$, i.e.\ $q+1<r$. In the NED above, the last edge of $P$
appears in the $(q+1)$th ear, which again contradicts the choice of $r$. 
Thus, $r=1$.

As $E_1$ has endvertices $s$ and $t$, the same as $P$,
 we get  $P=E_1$ from~\eqref{inside}.
\end{proof}

\begin{lemma}\label{inducedned}
Let $H$ be a part, $\cE$ a good NED of $H$ and $H=H_1\cup H_2$ be a series or parallel decomposition into parts $H_1$ and $H_2$.
Then, for $j=1,2$, the $\cE$-induced $x$-ear number $\lambda^{\cE}_j$ of $H_j$
equals the $x$-ear number $\lambda_j$ of $H_j$.
\end{lemma}
\begin{proof}
Observe that we only have to prove $\lambda^{\cE}_j\geq \lambda_j$ for $j=1,2$, and by symmetry
it suffices to prove that $\lambda_2^{\mathcal E}\geq \lambda_2$. 
Let $E_1,\ldots,E_m$ be the NED $\mathcal E$, and denote by $i_1<\ldots<i_n$ the indices $i$
such that $E_i\cap H_1$ is a non-trivial path. Note that 
$E_{i_1}\cap H_1,\ldots,E_{i_n}\cap H_1$ has $\lambda_1^{\mathcal E}$ many $x$-ears.
Let $\lambda$ be the $x$-ear number of~$H$.

First, assume that $H=H_1\cup H_2$ is a parallel decomposition.
Then $E_i\subseteq H_1$ or $E_i\subseteq H_2$ for $i=1,\ldots,m$.
This implies that $\lambda=\lambda_1^{\mathcal E}+\lambda_2^{\mathcal E}$.
Let $E'_1,\ldots,E'_\ell$ be a NED of $H_2$ with $\lambda_2$ many $x$-ears.
Then
\[
E_{i_1},\ldots,E_{i_n},E'_1,\ldots,E'_\ell
\]
is a NED of $H$ with $\lambda_1^{\mathcal E}+\lambda_2$ many $x$-ears. 
With $\lambda=\lambda_1^{\mathcal E}+\lambda_2^{\mathcal E}$ it follows that
$\lambda_2\leq\lambda_2^{\mathcal E}$.

\medskip

Second, assume $H=H_1\cup H_2$ to be a series decomposition,
and observe that we may assume $\lambda_2\geq 1$ as otherwise $\lambda_2\leq\lambda_2^{\mathcal E}$
trivially holds. Set $\delta_1=1$ if $E_1\cap H_1$ is an $x$-ear, and $\delta_1=0$ otherwise.
We distinguish two cases. In both cases, we will pick a specific NED $E'_1,\ldots,E'_\ell$ of $H_2$
and form a NED $\mathcal E^*$
\[
(E_1\cap H_1)\cup E_1',E_{i_2},\ldots, E_{i_n},E'_2,\ldots, E'_\ell
\]
of $H$. This is a NED as $i_1=1$ and $E_{i_j}\cap H_1=E_{i_j}$ for $j=2,\ldots, n$.

We first treat the case when $E_1\cap H_2$ is not an $x$-ear. 
That implies that $\mathcal E$ has 
\begin{equation}\label{earcount}
\lambda=1+(\lambda^{\cE}_1-\delta_1)+\lambda^{\cE}_2
\end{equation}
many $x$-ears. 

Moreover, since $E_1\cap H_2$ is not an $x$-ear, it follows from Lemma~\ref{pushoutlem} 
(with $E_1\cap H_2$ as $P$) that
there is a NED $E'_1,\ldots,E'_\ell$ of $H_2$ with $\lambda_2$ many $x$-ears such that $E'_1$ is not an $x$-ear.
 As $E_1$ is an $x$-ear (in $H$) but $E_1\cap H_2$ is not (in $H_2$), it follows
that $(E_1\cap H_1)\cup E_1'$ is an $x$-ear (in $H$). 
The NED $\mathcal E^*$ then has $1+(\lambda^{\cE}_1-\delta_1)+\lambda_2$ many $x$-ears. 
We get
\[
1+(\lambda^{\cE}_1-\delta_1)+\lambda_2\leq \lambda=1+(\lambda^{\cE}_1-\delta_1)+\lambda^{\cE}_2,
\]
by~\eqref{earcount}, which implies $\lambda_2\leq \lambda^{\cE}_2$.

Finally, we treat the case when $E_1\cap H_2$ is an $x$-ear. 
That implies that $\mathcal E$ has 
\begin{equation}\label{earcount2}
\lambda=1+(\lambda^{\cE}_1-\delta_1)+(\lambda^{\cE}_2-1)
\end{equation}
many $x$-ears.

Let $E'_1,\ldots,E'_\ell$
be a good NED of $H_2$. Then it has $\lambda_2$ many $x$-ears, and in particular, by Lemma~\ref{goodNEDlem}~(ii)
the first ear $E'_1$ is an $x$-ear. Consequently $(E_1\cap H_1)\cup E_1'$ is an $x$-ear, and we get that 
the NED $\mathcal E^*$
has
\[
1+(\lambda^{\cE}_1-\delta_1)+(\lambda_2-1)\leq \lambda=1+(\lambda^{\cE}_1-\delta_1)+(\lambda^{\cE}_2-1)
\]
many $x$-ears, where we have used~\eqref{earcount2}. Again we deduce $\lambda_2\leq \lambda^{\cE}_2$.
 \end{proof}

The next lemma describes how $x$-ear numbers behave 
in  series or parallel decompositions in relation to the $x$-ear number of the original part.
We will use it in Section~\ref{earinduction}.

Let $H$ be a part with terminals $s,t$, and let $B$ be a block-part of $H$ with terminals $a,b$, where 
we assume that there is an $s$--$a$~path in $H$ that is internally disjoint from $B$. 
Then we define $L_B,R_B\subseteq H$ as the parts  that are internally disjoint from $B$ 
such that $L_B$ has terminals $ \{s,a\}$, $R_B$ has terminals $\{b,t\}$
and such that $H=L_B\cup B\cup R_B$.  \label{LBRBdef}

\begin{lemma}\label{splitearlem}
Let $H$ be a part with $x$-ear number $\lambda$, and let $H=H_1\cup H_2$ be a 
series or parallel decomposition into parts $H_1,H_2$. Let $\lambda_1$ and $\lambda_2$
be the respective $x$-ear numbers of $H_1$ and $H_2$. 
\begin{enumerate}[\rm (i)]
\item\label{para} If $H=H_1\cup H_2$ is a parallel decomposition, then $\lambda_1+\lambda_2\leq\lambda$.
\item\label{substi} If $H=H_1\cup H_2$ is a series decomposition  and if both $H_1$ 
and $H_2$ contain a substantial block-part,
then $1\le \lambda_1,\lambda_2\le \lambda-1$ and $\lambda_1+\lambda_2\le\lambda$,
or $2\leq\lambda_1,\lambda_2\le \lambda-1$ and $\lambda_1+\lambda_2\leq\lambda+1$.
\item\label{terrible} 
If 
$H$ has exactly one substantial block-part $B$ that parallel decomposes into $B=B_1\cup B_2$ such that 
$B_2$ is non-substantial, and if at least one of $L_B-s$, $R_B-t$ contains a vertex that is adjacent to $X$,
then the $x$-ear number of $B$ is at most $\lambda-1$, and $\lambda\geq 2$.
\end{enumerate}
\end{lemma}
\begin{proof}
Given a good NED $\cE=(E_1,\ldots, E_r)$ of $H$, let $\lambda^{\cE}_j$ be the $\cE$-induced $x$-ear number of $H_j$,
while $\lambda_j$ is the $x$-ear number of $H_j$.
By Lemma~\ref{inducedned}, the two numbers are equal.
However, the distinction between these two meanings makes this proof clearer to understand.

Assume first that $H_1\cup H_2$ is a parallel decomposition. Then every ear
of $E_1,\ldots, E_r$ lies either in $H_1$ or in $H_2$.
Consequently, $\lambda=\lambda^{\cE}_1+\lambda^{\cE}_2$.
By Lemma~\ref{inducedned}, also $\lambda=\lambda_1+\lambda_2$ is true.
This proves~\eqref{para}.

\medskip
Let us prove~\eqref{substi} now. 
As both parts $H_1$ and $H_2$ contain a substantial block-part, we have $\lambda_1,\lambda_2\ge 1$.
Let us first treat the case that $E_1\cap H_j$ is an $x$-ear for at most one $j\in \{1,2\}$.
Then, no $x$-ear of $\mathcal E$ is counted twice in the sum $\lambda^{\cE}_1+\lambda^{\cE}_2$
and therefore, using Lemma \ref{inducedned}, we get $\lambda_1+\lambda_2\le \lambda$.
Using $\lambda_1,\lambda_2\ge 1$ we also conclude $\lambda_1,\lambda_2\le \lambda-1$.

Second, we treat the case when both $E_1\cap H_1$ and $E_1\cap H_2$ are $x$-ears.
No $x$-ear of $\mathcal E$ except for $E_1$ is counted twice in $\lambda^{\cE}_1+\lambda^{\cE}_2$, 
so again by Lemma~\ref{inducedned}
we have $\lambda_1+\lambda_2\leq\lambda+1$.
It remains to prove $\lambda_1,\lambda_2\ge 2$.
Suppose  $\lambda_1=1$.
Then, the induced NED $\cE'_1$ of $H_1$ has only one $x$-ear, namely $E_1\cap H_1$.
However, as $H_1$ contains a substantial block-part $B$,
there must be vertex $b$ in $B$ that is not a terminal and that is adjacent to a vertex in $X$. 
As $\mathcal E'_1$ has a unique $x$-ear, that vertex $b$ lies in the interior of $E_1\cap H_1$.
Since $B$ is a block-part, there is, moreover, some ear $E_i$ that is nested on $E_1$ 
such that $b$ lies in the interior of its nest interval $I(E_i)$. Note that $E_i$ cannot be an $x$-ear.

From Lemma~\ref{goodNEDlem}~(i) we get that the interior of $I(E_i)$ must contain all neighbours
of $X$ that are contained in the interior of $E_1$. This, however, is impossible as $E_1\cap H_2$ 
is an $x$-ear, too. We have proved \eqref{substi}.

 \medskip
At last consider the situation in~\eqref{terrible}.
Suppose that $E_1\cap B$ is an $x$-ear, i.e.\ it contains
a neighbour of $X$ in its interior. In particular, it follows that $E_1\cap B\subseteq B_1$. 
Let $a,b$ be the terminals of $B$. Observe that $\mathcal E$ contains an ear $F$ with 
endvertices $a,b$ that is contained in $B_2$. Then, as $B_2$ is not substantial, the ear $F$ is not an $x$-ear.
Moreover, $F$ is nested on $E_1$ since $E_1$ passes through $a$ and $b$. Note that, since at least one
of $L_B-s$ and $R_B-t$ contains a neighbour of $X$ and since neither $L_B$ nor $R_B$
contains a substantial block-part it follows that the interior of $E_1$ contains a neighbour of $X$
outside $B-\{a,b\}$. If we apply Lemma~\ref{goodNEDlem}~(i) to $E_1$ and $F$, we obtain, however, a contradiction,
since $E_1\cap B$ is supposed to contain a neighbour of $X$ as well in its interior. 
Thus $E_1\cap B$ is not an $x$-ear. 

As a consequence, the induced NED of $B$, of which $E_1\cap B$ is an ear, contains 
at least one $x$-ear less than $\mathcal E$ (we lose the first ear as $x$-ear).
With Lemma~\ref{inducedned}, we obtain $\lambda_B\leq\lambda-1$. As, by assumption, 
$B$ is substantial, we have, on the other hand, $\lambda_B\geq 1$, which implies $\lambda\geq 2$.
\end{proof}

\subsection{Ladders and fans}\label{laddersfans}
A \emph{well-connected ladder} is a subgraph $H$ of $G-x$ that is the union of internally disjoint 
non-trivial parts
$Q_1,\ldots, Q_{3k+2},R_1,\ldots,R_{3k+2},S_1\ldots,S_{3k+3}$ such that 
\begin{itemize}
\item there are distinct vertices $s_1,\ldots,s_{3k+3},t_1,\ldots, t_{3k+3}$ 
such that for $i=1,\ldots, 3k+2$ the terminals of $Q_i$ are $s_i$ and $s_{i+1}$, 
the terminals of $R_i$ are $t_i$ and $t_{i+1}$, and such that for $i=1,\ldots,3k+3$
the terminals of $S_i$ are $s_i$ and $t_i$;
\item there are $a\in\{s_1,t_1\}$, $b\in\{s_{3k+3},t_{3k+3}\}$ and $8k$ 
edge-disjoint $a$--$b$-paths in $H$, there are $k$ edge-disjoint $a$--$x$-paths that meet $H$ at most in $\{s_1,t_1\}$
and there are $k$ edge-disjoint $b$--$x$-paths that meet $H$ at most in $\{s_{3k+3},t_{3k+3}\}$.
\end{itemize}

We define a second structure that is very similar. For later use, we keep the definition 
a bit more flexible, though.

A \emph{fan-graph} is a subgraph $H$ of $G-x$ that is the union of internally disjoint 
non-trivial parts $Q_1,\ldots, Q_\ell,S_1,\ldots,S_{\ell+1}$ such that 
there are distinct vertices $s_1,\ldots,s_{\ell+1},t$ 
such that for $i=1,\ldots, \ell$ the terminals of $Q_i$ are $s_i$ and $s_{i+1}$, 
and such that for $i=1,\ldots,\ell+1$
the terminals of $S_i$ are $s_i$ and $t$.
The integer~$\ell$ is the \emph{size} of the fan-graph
and we call $s_1,s_{\ell+1},t$ the terminals of the fan-graph.

The fan-graph $H$ is a \emph{well-connected fan} if 
the size of $H$ is $\ell=3k$, and if 
there are $6k$ edge-disjoint $s_1$--$s_{3k+1}$-paths in $H-t$, and for each $c\in\{s_1,s_{3k+1}\}$
there are $k$ edge-disjoint $c$--$x$-paths in $G$ that are internally disjoint from $H$.

\begin{lemma}\label{ladderfanlem}
If $G$ contains a well-connected ladder or a well-connected fan,  then $G$
contains $k$ edge-disjoint $K_4$-subdivisions.
\end{lemma}
\begin{proof}
First assume that $G$ contains a well-connected ladder $H$ consisting of parts 
$Q_1,\ldots, Q_{3k+1},R_1,\ldots,R_{3k+1},S_1\ldots,S_{{3k}+2}$. 
Let $a\in\{s_1,t_1\}$, $b\in\{s_{3k+2},t_{3k+2}\}$ such that $H$ contains $8k$ edge-disjoint 
$a$--$b$-paths $\mathcal P$. 

For each $i=0,\ldots, k-1$ pick in each of the parts $Q_{3i+2}$, $Q_{3i+3}$, $R_{3i+2}$, $R_{3i+3}$,
$S_{3i+2}$, $S_{3i+3}$, $S_{3i+4}$ a path between the respective terminals and denote their union 
by $L_{i+1}$. Choose $L_{i+1}$ in such a way such that $L_{i+1}$ contains edges from at most seven distinct paths 
in $\mathcal P$. Then $L_1,\ldots, L_k$ are pairwise edge-disjoint, no $L_i$ contains any vertex of $s_1,t_1,s_{3k+2},t_{3k+2}$
as the parts $S_1,Q_1,R_1,S_{3k+2},Q_{3k+1},R_{3k+1}$ are never used for any $L_i$.
Furthermore, in $\mathcal P$ there are 
still at least $k$ paths that are also edge-disjoint from $L_1,\ldots, L_k$; denote these paths by~$P_1,\ldots,P_k$.

Let $L'_i$ be the union of $L_i$ with the initial segment of $P_i$ from $a$ to $L_i$ and 
with the terminal segment from $L_i$ to $b$. In particular, the graphs $L'_1,\ldots,L'_k$
are still pairwise edge-disjoint. By definition of a well-connected ladder, there are 
furthermore $k$ edge-disjoint $a$--$x$-paths $\mathcal P_a$ that meet $H$ only in $\{s_1,t_1\}$
and $k$ edge-disjoint $b$--$x$-paths $\mathcal P_b$ that meet $H$ only in $\{s_{3k+3},t_{3k+3}\}$.
Hence, no path in $\cP_a$ or $\cP_b$ meets any $L_i$.
Note that also no path in $\mathcal P_a$ may meet any path in $\mathcal P_b$ outside $x$ since 
otherwise $G-x$ would contain a  $K_4$-subdivision, which  contradicts~\eqref{standardG}. 
Thus, each $L'_i$ together with a distinct path in $\mathcal P_a$
and a distinct path $\mathcal P_b$ yields a $K_4$-subdivision, and all of these are pairwise edge-disjoint.

\comment{
\medskip
Second assume that $G$ contains a well-connected fan $H$ consisting of parts\mymargin{we can skip second case as it is very similar} 
$Q_1,\ldots, Q_{3k},S_1\ldots,S_{{3k}+1}$. Let $\mathcal P$ be the set of $6k$
pairwise edge-disjoint $s_1$--$s_{3k+1}$-paths in $H-t$ that are guaranteed by the definition.

For each $i=0,\ldots, k-1$ pick in each of the parts $Q_{3i+1}$, $Q_{3i+2}$,
$S_{3i+1}$, $S_{3i+2}$, $S_{3i+3}$ a path between the respective terminals and denote their union 
by $L_{i+1}$. Choose $L_{i+1}$ in such a way such that $L_{i+1}$ contains edges from at most five distinct paths 
in $\mathcal P$.
Then $L_1,\ldots, L_k$ are pairwise edge-disjoint, and in $\mathcal P$ there are 
still at least $k$ paths that are also edge-disjoint from $L_1,\ldots, L_k$; denote these paths by~$P_1,\ldots,P_k$.

Let $L'_i$ be the union of $L_i$ with the initial segment of $P_i$ from $s_1$ to $L_i$ and 
with the terminal segment from $L_i$ to $s_{3k+1}$. In particular, the graphs $L'_1,\ldots,L'_k$
are still pairwise edge-disjoint. By definition of a well-connected fan, there are 
furthermore $k$ edge-disjoint $s_1$--$x$-paths $\mathcal P_a$ 
and $k$ edge-disjoint $s_{3k+1}$--$x$-paths $\mathcal P_b$ that are all internally disjoint from $H$. 
Note first that also no path in $\mathcal P_a$ may meet any path in $\mathcal P_b$ outside $x$ since 
otherwise $G-x$ would have $K_4$ as a minor. Thus, each $L'_i$ together with a distinct path in $\mathcal P_a$
and a distinct path $\mathcal P_b$ yields a $K_4$-subdivision, and all of these are pairwise edge-disjoint.
}
We omit the proof for well-connected fans as it is very similar. 
\end{proof}

We will apply the lemma in the proof of Lemma \ref{type5lem}.

\subsection{Modules}
\label{moduleSection}

Let $H$ be a part. A non-empty subgraph $M\subseteq H$ is a \emph{module of $H$} if there is a $K_4$-subdivision $K$
such that $M$ is obtained from $K\cap H$ by deleting isolated vertices.
(Deleting isolated vertices is merely a matter of convenience: it helps to reduce
the number of of possible types of modules.)
Although there are infinitely many possible modules
we will prove that they can be classified into only finitely many types of them, the  blueprints.
The aim of this subsection is to define blueprints, to prove that every module has a blueprint
and to investigate
how modules of a part $H$ behave in a series or parallel decomposition.

A \emph{labelled graph} is a graph in which some of the vertices are endowed with a label from 
some alphabet~$\Sigma$ (while other vertices remain unlabelled). We only use the labels \texttt s, \texttt t
and \texttt x. We fix a set of labelled graphs that we call \emph{basic blueprints}. 
These are the labelled graphs in Figure~\ref{basicfig},  
as well
as some more graphs obtained from them: namely, we also allow to contract any dashed 
edge of a graph in Figure~\ref{basicfig}, where we require that the resulting vertex receives the 
label of the labelled endvertex of the contracted edge. For instance, the labelled graph in Figure~\ref{accidentfig}~(a)
is obtained from (h) in Figure~\ref{basicfig} by contracting the unique dashed edge. Moreover, 
we allow that the labels \texttt s and \texttt t are exchanged. That means, for instance, that 
an edge with endvertices labelled \texttt t and \texttt x is a basic blueprint, too.

\begin{figure}[ht]
\centering
\begin{tikzpicture}[scale=0.65]

\def\xshift{6}
\def\yshift{-3.8}
\def\hh{0.7}
\def\bb{0.37}

\newcounter{plbl}
\def\lbly{-1.2}
\newcommand{\patternlabel}{
\stepcounter{plbl}
\node at (0,\lbly) {(\alph{plbl})};
}
\newcommand{\expatternlabel}{
\stepcounter{plbl}
\node at (0,\lbly) {(\alph{plbl})$^*$};
}

\begin{scope}[shift={(0,-1)}]
\draw[hedge] (0,-0.5) \svx -- (0,0.5) \xvx;
\patternlabel
\end{scope}

\begin{scope}[shift={(\xshift,-1)}]
\draw[hedge] (-0.5,-0.5) \svx -- (-0.5,0.5) \xvx;
\draw[hedge] (0.5,-0.5) \tvx -- (0.5,0.5) \xvx;
\patternlabel
\end{scope}

\begin{scope}[shift={(2*\xshift,-1)}]
\draw[hedge] (-0.5,-0.5) \svx -- (0.5,-0.5) \tvx;
\patternlabel
\end{scope}

\begin{scope}[shift={(0,\yshift)}]
\begin{scope}[shift={(-1,-0.5)}]
\draw[ced] (0,0) \svx  -- (1,0) node[hvertex] (v) {};
\draw[hedge] (v) -- (2,0) \tvx;
\draw[hedge] (v) -- (1,1) \xvx;
\end{scope}
\patternlabel
\end{scope}

\begin{scope}[shift={(\xshift,\yshift)}]
\begin{scope}[shift={(-1.5,-0.5)}]
\draw[ced] (0,0) \svx  -- (1,0) node[hvertex] (v) {};
\draw[hedge] (v) -- (2,0) node[hvertex] (w) {};
\draw[ced] (w) -- (3,0) \tvx;
\draw[hedge] (v) -- (1,1) \xvx;
\draw[hedge] (w) -- (2,1) \xvx;
\end{scope}
\patternlabel
\end{scope}

\begin{scope}[shift={(2*\xshift,\yshift)}]
\begin{scope}[shift={(-2,-0.5)}]
\draw[ced] (0,0) \svx  -- (1,0) node[hvertex] (v) {};
\draw[hedge] (v) -- (2,0) node[hvertex] (w) {} -- (3,0) node[hvertex] (u) {};
\draw[ced] (u) -- (4,0) \tvx;
\draw[hedge] (v) -- (1,1) \xvx;
\draw[hedge] (w) -- (2,1) \xvx;
\draw[hedge] (u) -- (3,1) \xvx;
\end{scope}
\expatternlabel
\end{scope}

\begin{scope}[shift={(0,2*\yshift)}]
\draw[hedge] (0,\hh) node[hvertex] (u) {} -- (-30:\hh) node[hvertex] (v){} -- (210:\hh) node[hvertex] (w){} -- cycle;
\draw[hedge] (u) -- (0,\hh+1) \xvx;
\draw[hedge] (v) -- (-30:\hh+1) \xvx;
\draw[hedge] (w) -- (210:\hh+1) \xvx;
\patternlabel
\end{scope}

\begin{scope}[shift={(\xshift,2*\yshift)}]
\draw[hedge] (0,\hh) node[hvertex] (u) {} -- (-30:\hh) node[hvertex] (v){} -- (210:\hh) node[hvertex] (w){} -- cycle;
\draw[hedge] (u) -- (0,\hh+1) \xvx;
\draw[hedge] (v) -- (-30:\hh+1) \xvx;
\draw[ced] (w) -- (210:\hh+1) \svx;
\patternlabel 
\end{scope}

\begin{scope}[shift={(2*\xshift,2*\yshift)}]
\draw[hedge] (0,\hh) node[hvertex] (u) {} -- (-30:\hh) node[hvertex] (v){} -- (210:\hh) node[hvertex] (w){} -- cycle;
\draw[hedge] (u) -- (0,\hh+1) \xvx;
\draw[ced] (v) -- (-30:\hh+1) \tvx;
\draw[ced] (w) -- (210:\hh+1) \svx;
\patternlabel
\end{scope}

\begin{scope}[shift={(0,3*\yshift)}]
\draw[hedge] (0,\hh) node[hvertex] (u) {} -- (-30:\hh) node[hvertex] (v){} -- (210:\hh) node[hvertex] (w){} -- cycle;
\draw[ced] (u) -- (0,\hh+1) \tvx;
\draw[hedge] (v) -- (-30:2*\hh) node[hvertex] (z){};
\draw[ced] (z) -- ++(-30:1) \svx;
\draw[hedge] (w) -- (210:\hh+1) \xvx;
\draw[hedge] (z) -- ++(60:1) \xvx;
\expatternlabel
\end{scope}

\begin{scope}[shift={(\xshift,3*\yshift)}]
\draw[hedge] (0,\hh) node[hvertex] (u) {} -- (-30:\hh) node[hvertex] (v){} -- (210:\hh) node[hvertex] (w){} -- cycle;
\draw[hedge] (u) -- (0,\hh+1) \xvx;
\draw[hedge] (v) -- (-30:\hh+1) \xvx;
\draw[ced] (w) -- (210:\hh+1) \svx;
\draw[hedge] (210:\hh+1 |- 0,\hh+1) \xvx -- ++(0,-1) \tvx;
\expatternlabel
\end{scope}

\begin{scope}[shift={(2*\xshift,3*\yshift)}]

\def\cc{0.6}
\draw[hedge] (45:\cc) node[hvertex] (a) {} -- (135:\cc) node[hvertex] (b) {} -- (225:\cc) node[hvertex](c) {} -- (-45:\cc) node[hvertex] (d) {} -- cycle;
\draw[ced] (a) -- (45:1+\cc) \tvx;
\draw[hedge] (b) -- (135:1+\cc) \xvx;
\draw[ced] (c) -- (225:1+\cc) \svx;
\draw[hedge] (d) -- (-45:1+\cc) \xvx;
\expatternlabel

\end{scope}

\begin{scope}[shift={(0,4*\yshift+0.4)}]

\begin{scope}[shift={(-\bb,0)}]
\draw[hedge] (60:\hh) node[hvertex] (a){} -- (180:\hh) node[hvertex] (b){} -- (-60:\hh) node[hvertex](c){} -- (a) -- (2*\bb+\hh,0) node[hvertex] (d){} -- (c);

\draw[hedge] (b) -- ++(0,1) \xvx;
\draw[hedge] (d) -- ++(0,1) \xvx;
\end{scope}
\patternlabel
\end{scope}

\begin{scope}[shift={(\xshift,4*\yshift+0.4)}]

\begin{scope}[shift={(-\bb,0)}]
\draw[hedge] (60:\hh) node[hvertex] (a){} -- (180:\hh) node[hvertex] (b){} -- (-60:\hh) node[hvertex](c){} -- (a) -- (2*\bb+\hh,0) node[hvertex] (d){} -- (c);

\draw[ced] (b) -- ++(0,1) \svx;
\draw[hedge] (d) -- ++(0,1) \xvx;
\end{scope}
\patternlabel
\end{scope}

\begin{scope}[shift={(2*\xshift,4*\yshift+0.4)}]
\begin{scope}[shift={(-\bb,0)}]
\draw[hedge] (60:\hh) node[hvertex] (a){} -- (180:\hh) node[hvertex] (b){} -- (-60:\hh) node[hvertex](c){} -- (a) -- (2*\bb+\hh,0) node[hvertex] (d){} -- (c);

\draw[ced] (b) -- ++(0,1) \svx;
\draw[hedge] (d) -- ++(0,1) \xvx;

\path (b) ++(-1,0) \xvx [draw,hedge] -- ++(0,1) \tvx;

\end{scope}
\expatternlabel
\end{scope}
\end{tikzpicture}
\caption{The basic blueprints. Starred blueprints are exceptional; see 
discussion before Lemma~\ref{hittinglem}.}\label{basicfig}
\end{figure}
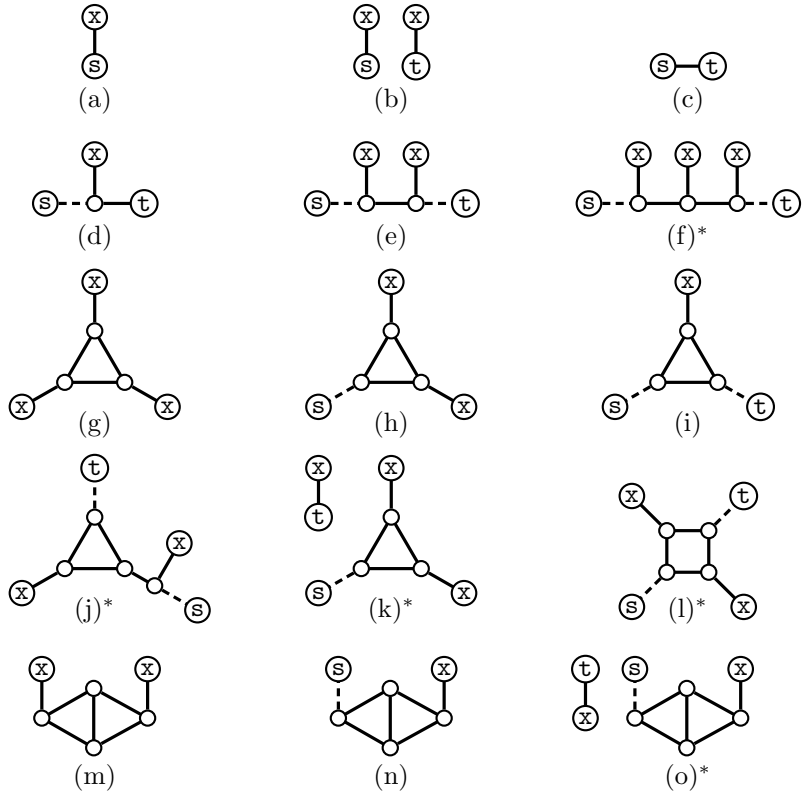

A \emph{blueprint} is either a basic blueprint, or it is \emph{derived} from a basic blueprint~$B$
as follows: for each label $\alpha\in\{\texttt s,\texttt t\}$ that does not appear in $B$
we may either label an unlabelled vertex with $\alpha$, or we may subdivide an edge and label the 
resulting subdividing vertex with $\alpha$, or we may choose to omit~$\alpha$. (That is, if both labels
\texttt s and \texttt t are missing, we may add one or both of the missing labels.) 
In this case, we say that $\alpha$ is an \emph{accidental}
label of the blueprint. 
As an example, consider 
Figure~\ref{accidentfig} where we derive four more blueprints from 
basic blueprint~(a) by introducing an accidental label (\texttt t here).

\begin{figure}[ht]
\centering
\begin{tikzpicture}[scale=0.8]
\tikzstyle{hvertex}=[solid,thick,circle,inner sep=0mm, minimum size=2mm, fill=white, draw=black]
\tikzstyle{labvx}=[solid,thick,circle,inner sep=0.3mm, minimum size=2mm, fill=white, draw=black]
\tikzstyle{ced}=[hedge,densely dashed]

\def\xshift{3.7}
\def\yshift{-3.8}
\def\hh{0.7}
\def\bb{0.37}

\setcounter{plbl}{0}
\def\lbly{-1.2}
\newcommand{\patternlabel}{
\stepcounter{plbl}
\node at (0,\lbly) {(\alph{plbl})};
}

\begin{scope}[shift={(1.5*\xshift,0)}]
\draw[hedge] (0,\hh) node[hvertex] (u) {} -- (-30:\hh) node[hvertex] (v){} -- (210:\hh) \svx -- cycle;
\draw[hedge] (u) -- (0,\hh+1) \xvx;
\draw[hedge] (v) -- (-30:\hh+1) \xvx;
\patternlabel
\end{scope}

\begin{scope}[shift={(0,\yshift)}]
\draw[hedge] (0,\hh) node[hvertex] (u) {} -- (-30:\hh) node[hvertex] (v){} -- (-90:\bb) \tvx -- (210:\hh) \svx -- cycle;
\draw[hedge] (u) -- (0,\hh+1) \xvx;
\draw[hedge] (v) -- (-30:\hh+1) \xvx;
\patternlabel
\end{scope}

\begin{scope}[shift={(\xshift,\yshift)}]
\draw[hedge] (0,\hh) node[hvertex] (u) {} -- (-30:\hh) coordinate (v) -- (210:\hh) \svx -- cycle;
\draw[hedge] (u) -- (0,\hh+1) \xvx;
\draw[hedge] (v) -- (-30:\hh+1) \xvx;
\path (v) \tvx;
\patternlabel
\end{scope}

\begin{scope}[shift={(2*\xshift,\yshift)}]
\draw[hedge] (0,\hh) node[hvertex] (u) {} -- (30:\bb) \tvx -- (-30:\hh) node[hvertex] (v){} -- (210:\hh) \svx -- cycle;
\draw[hedge] (u) -- (0,\hh+1) \xvx;
\draw[hedge] (v) -- (-30:\hh+1) \xvx;
\patternlabel
\end{scope}

\begin{scope}[shift={(3*\xshift,\yshift)}]
\draw[hedge] (0,\hh) node[hvertex] (u) {} -- (-30:\hh) node[hvertex] (v){} -- (210:\hh) \svx -- cycle;
\draw[hedge] (u) -- (0,\hh+1) \xvx;
\draw[hedge] (v) -- (-30:\hh+0.6) \tvx -- (-30:\hh+1.3) \xvx;
\patternlabel
\end{scope}
\end{tikzpicture}
\caption{We obtain four more blueprints from the basic blueprint~(a) by accidental use of \texttt t}\label{accidentfig}
\end{figure}
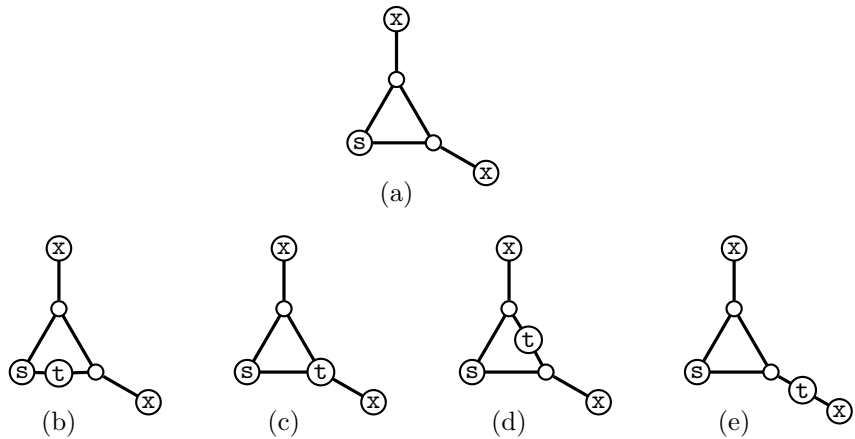

Let $H$ be a part with terminals $s,t$ (where we think of $s$ as the first terminal, and $t$ as the second terminal,  
i.e.\ the order of the terminals matters) and let $B$ be a blueprint.
We say that a subgraph $M\subseteq H$ (that is usually, but not necessarily, a module) 
\emph{has blueprint} $B$ 
if $M$ is a subdivision of $B$ such that any branch vertex of $M$ that has label \texttt x in $B$ lies in $X$, 
and such that a vertex of $M$ is equal to $s$ (resp.\ to~$t$) if and only if it is a branch vertex that is labelled with 
\texttt s (resp.~\texttt t) in $B$.
If the blueprint $B$ 
contains \texttt s or \texttt t as an accidental label, then $M$ \emph{uses} $s$, resp.\ $t$, \emph{accidentally}.

We will show in the next lemma that every module has a blueprint.
To talk more easily about the different kinds of modules, we give names to most of them. 
Let $M$ be a subgraph of some part that has blueprint $B$. 
If $B$ was obtained from 
basic blueprint~(a) in Figure~\ref{basicfig}, then $M$ is an \emph{appendix}. If $B$ came from~(b),
then $M$
is a \emph{double appendix}, if $B$ came from one of~(c)--(f), then $M$ is a \emph{comb}, 
and it is a \emph{$0$-, $1$-,
$2$- or $3$-comb} depending on the number of vertices from $X$ it contains. 
Finally, if $B$ arose from one of (m)--(o), it is a \emph{rooted diamond}.

\begin{lemma} \label{moduleHasBlueprint}
Every module of a part has a blueprint.
\end{lemma}
\begin{proof}

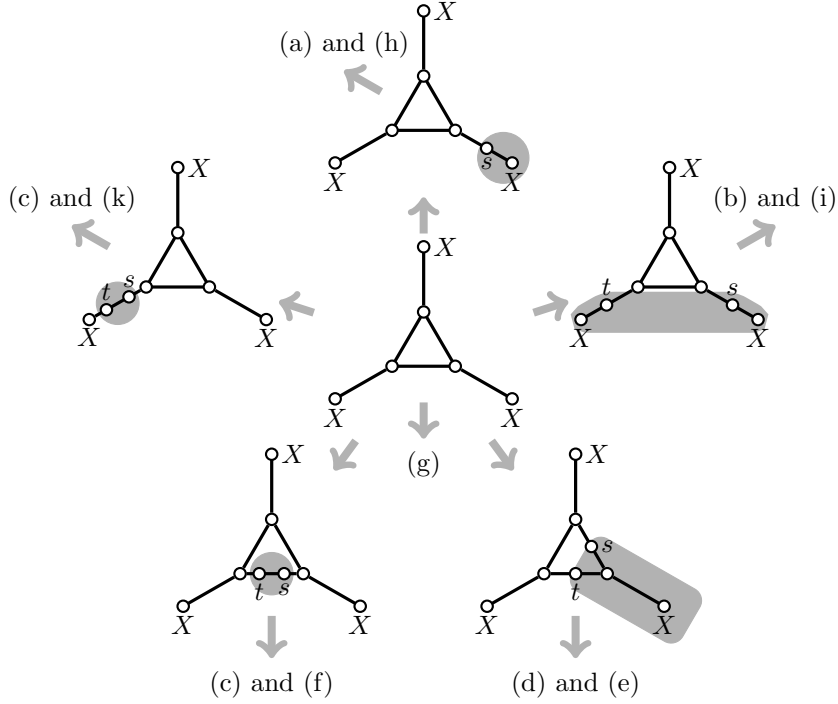
\begin{figure}[ht]
\centering
\begin{tikzpicture}[scale=0.8,label distance=-1mm]
\tikzstyle{hvertex}=[solid,thick,circle,inner sep=0.cm, minimum size=1.5mm, fill=white, draw=black]
\tikzstyle{part}=[fill=hellgrau, draw=hellgrau,rounded corners=5pt]
\tikzstyle{thickarrow}=[line width=3pt,color=hellgrau,->]

\newcommand{\ssvx}{node[labvx]{\tiny $s$}}
\newcommand{\ttvx}{node[labvx]{\tiny $t$}}
\newcommand{\xxvx}{node[labvx]{\tiny $X$}}

\def\hh{0.6}
\def\gg{2.8}
\def\radius{1.7}

\draw[hedge] (0,\hh) node[hvertex] (u) {} -- (-30:\hh) node[hvertex] (v){} -- (210:\hh) node[hvertex] (w){} -- cycle;
\draw[hedge] (u) -- (0,\gg*\hh) node[hvertex,label=right:$X$]{};
\draw[hedge] (v) -- (-30:\gg*\hh) node[hvertex,label=below:$X$]{};
\draw[hedge] (w) -- (210:\gg*\hh) node[hvertex,label=below:$X$]{};

\foreach \i in {0,1,2,3,4}{
\draw[thickarrow] (90+\i*72:\radius+0.2) -- (90+\i*72:\radius+0.8);
}

\draw[thickarrow] (90:2.5*\radius) ++ (-0.7,0) -- ++(150:0.7) node[color=black,above]{(a) and (h)};
\draw[thickarrow] (90+1*72:3.2*\radius) -- ++(150:0.7) node[color=black,above]{(c) and (k)};
\draw[thickarrow] (90+2*72:2.5*\radius) ++(0,-1) -- ++(0,-0.7) node[color=black,below]{(c) and (f)};
\draw[thickarrow] (90+3*72:2.5*\radius) ++(0,-1) -- ++(0,-0.7) node[color=black,below]{(d) and (e)};
\draw[thickarrow] (90+4*72:3.2*\radius) -- ++(30:0.7) node[color=black,above]{(b) and (i)};

\draw[thickarrow] (-90:\radius-0.8) -- (-90:\radius-0.2) node[color=black,below] {(g)};

\begin{scope}[shift={(0,2.3*\radius)}]
\draw[hedge] (0,\hh) node[hvertex] (u) {} -- (-30:\hh) node[hvertex] (v){} -- (210:\hh) node[hvertex] (w){} -- cycle;
\draw[hedge] (u) -- (0,\gg*\hh) node[hvertex,label=right:$X$]{};
\draw[hedge] (v) -- (-30:\gg*\hh) node[hvertex,label=below:$X$](X42){} node[midway,hvertex,label=below:$s$](s){};
\draw[hedge] (w) -- (210:\gg*\hh) node[hvertex,label=below:$X$]{};
\begin{scope}[on background layer]
\draw[part] (-30:0.9*\gg*\hh) circle (12pt);
\end{scope}
\end{scope}

\begin{scope}[shift={(90+72:2.5*\radius)}]
\draw[hedge] (0,\hh) node[hvertex] (u) {} -- (-30:\hh) node[hvertex] (v){} -- (210:\hh) node[hvertex] (w){} -- cycle;
\draw[hedge] (u) -- (0,\gg*\hh) node[hvertex,label=right:$X$]{};
\draw[hedge] (v) -- (-30:\gg*\hh) node[hvertex,label=below:$X$]{};
\node[hvertex,label=below:$X$] (X1) at (210:\gg*\hh) {};
\draw[hedge] (w) -- (X1)  node[near start,hvertex,label=above:$s$](s){} coordinate[midway] (H1)
node[near end,hvertex,label=above:$t$](t){};
\end{scope}

\begin{scope}[shift={(90+2*72:2.5*\radius)}]
\node[hvertex] (u) at (0,\hh){};
\node[hvertex] (v) at (-30:\hh){};
\node[hvertex] (w) at (210:\hh){};
\draw[hedge] (u) -- (v) -- (w) node[near start,hvertex,label=below:$s$](s){} coordinate[midway] (H2)
node[near end,hvertex,label=below:$t$](t){} -- (u);
\draw[hedge] (u) -- (0,\gg*\hh) node[hvertex,label=right:$X$]{};
\draw[hedge] (v) -- (-30:\gg*\hh) node[hvertex,label=below:$X$]{};
\draw[hedge] (w) -- (210:\gg*\hh) node[hvertex,label=below:$X$]{};
\end{scope}

\begin{scope}[shift={(90+3*72:2.5*\radius)}]
\node[hvertex] (u) at (0,\hh){};
\node[hvertex] (v) at (-30:\hh){};
\node[hvertex] (w) at (210:\hh){};
\draw[hedge] (u) -- (v)  node[midway,hvertex,label=right:$s$](s){} -- (w) 
  node[midway,hvertex,label=below:$t$](t){} -- (u);
\draw[hedge] (u) -- (0,\gg*\hh) node[hvertex,label=right:$X$]{};
\draw[hedge] (v) -- (-30:\gg*\hh) node[hvertex,label=below:$X$]{} coordinate[midway] (H3);
\draw[hedge] (w) -- (210:\gg*\hh) node[hvertex,label=below:$X$]{};
\begin{scope}[on background layer, shift={(H3)},rotate=-30]
\draw[part] (-1.1,-0.5) rectangle (1,0.5);
\end{scope}
\end{scope}

\begin{scope}[shift={(90+4*72:2.5*\radius)}]
\draw[hedge] (0,\hh) node[hvertex] (u) {} -- (-30:\hh) node[hvertex] (v){} -- (210:\hh) node[hvertex] (w){} -- cycle;
\draw[hedge] (u) -- (0,\gg*\hh) node[hvertex,label=right:$X$]{};
\draw[hedge] (v) -- (-30:\gg*\hh) node[hvertex,label=below:$X$](Xs){} node[midway,hvertex,label=above:$s$](s){};
\draw[hedge] (w) -- (210:\gg*\hh) node[hvertex,label=below:$X$](Xt){} node[midway,hvertex,label=above:$t$](t){};
\begin{scope}[on background layer]
\draw[part,line width=10pt,rounded corners=3pt] (s.center) -- (Xs.center) -- (Xt.center) -- (t.center) -- cycle;
\end{scope}
\end{scope}

\begin{scope}[on background layer]
\draw[part] (H1) circle (10pt);
\draw[part] (H2) circle (10pt);
\end{scope}

\end{tikzpicture}
\caption{Where the blueprints come from if $x$ is a branch vertex. The two terminals
separate $G-x$ into two parts, a gray part and the rest.}
\label{fig:k4toblueprints1}
\end{figure}

Let $M$ be a module of a part $H$, and let $K$ be a $K_4$-subdivision such that $M$
is obtained from $H\cap K$ by deleting isolated vertices. 
If $s,t$ are the terminals of $H$, then there is a smallest set $S\subseteq \{s,t\}$ 
such that $S\cup (H\cap X)$ separates $M$ from $x$ in $K$. 
Figure~\ref{fig:k4toblueprints1} shows schematically how $M$ may be situated 
with respect to $H$ if $x$ is a branch vertex of $K$, and Figure~\ref{fig:k4toblueprints2} 
does the same if $x$ is  a subdividing vertex. In the figures, $M$ and thus $H$ may 
be either represented by the gray area, or by everything outside the gray area (except 
for the terminals). The different cases arise from the different locations of $S$ with respect 
to $M$, where we have omitted terminals outside $S$ that meet $M$ from the drawing --- 
their inclusion  would multiply the number of configurations considerably. Moreover, 
if only one of $s,t$ lies in $S$ then, by symmetry, we assume that it is $s\in S$. 
There is a final simplification in the drawings: $s,t$ are always shown as subdividing vertices in $K$ 
but they may obviously be branch vertices. Thus, the figures should be understood so that 
each terminal might move to the closest branch vertex. 
The reader might want to check that we have covered all possible configurations: here, 
we note that in Figure~\ref{fig:k4toblueprints2} it is not possible that $s,t$ separates 
the two vertices in $M\cap V(X)$
from the rest of $M$ as this would contradict Lemma~\ref{diamonds}~(ii).

\begin{figure}[ht]
\centering
\begin{tikzpicture}[scale=0.8,label distance=-1mm]
\tikzstyle{hvertex}=[solid,thick,circle,inner sep=0.cm, minimum size=1.5mm, fill=white, draw=black]
\tikzstyle{thickarrow}=[line width=3pt,color=hellgrau,->]
\tikzstyle{part}=[fill=hellgrau, draw=hellgrau,rounded corners=5pt]

\newcommand{\ssvx}{node[labvx]{\tiny $s$}}
\newcommand{\ttvx}{node[labvx]{\tiny $t$}}
\newcommand{\xxvx}{node[labvx]{\tiny $X$}}

\def\hh{0.7}
\def\bb{0.37}
\def\radius{1.5}

\begin{scope}[shift={(-\bb,-0.3)}]
\draw[hedge] (60:\hh) node[hvertex] (a){} -- (180:\hh) node[hvertex] (b){} -- (-60:\hh) node[hvertex](c){} -- (a) -- (2*\bb+\hh,0) node[hvertex] (d){} -- (c);

\draw[hedge] (b) -- ++(0,1) node[hvertex,label=right:$X$] {};
\draw[hedge] (d) -- ++(0,1) node[hvertex,label=left:$X$] {};
\end{scope}

\foreach \i in {0,1,2,3,4}{
\draw[thickarrow] (90+\i*72:\radius) -- (90+\i*72:\radius+0.7);
}

\draw[line width=3pt,color=hellgrau,->] (90-72-10:3.2*\radius) -- ++(-45:0.7) 
node[below,color=black]{(c) and (l)};
\draw[line width=3pt,color=hellgrau,->] (90+72+10:3.2*\radius) -- ++(180+45:0.7)
node[below,color=black]{(a) and (n)};
\draw[line width=3pt,color=hellgrau,->] (90+2*72:3.2*\radius) -- (90+2*72:3.2*\radius+0.7)
node[below,color=black]{(c) and (o)};
\draw[line width=3pt,color=hellgrau,->] (90+3*72:3.2*\radius) -- (90+3*72:3.2*\radius+0.7)
node[below,color=black]{(c) and (j)};
\draw[line width=3pt,color=hellgrau,->] (110:2.7*\radius) -- ++ (-0.7,0){} node[left,color=black]{(d) and (i)};

\draw[thickarrow] (200:\radius) -- (200:\radius+0.7) node[color=black,left] {(m)};

\begin{scope}[shift={(-\bb,0)}]

\begin{scope}[shift={(90+72:2.6*\radius)}]
\draw[hedge] (60:\hh) node[hvertex] (a){} -- (180:\hh) node[hvertex] (b){} -- (-60:\hh) node[hvertex](c){} -- (a) -- (2*\bb+\hh,0) node[hvertex] (d){} -- (c);

\draw[hedge] (b) -- ++(0,1) node[hvertex,label=right:$X$] {};
\path (d) -- ++(0,1) node[hvertex,label=left:$X$] (X0) {};
\draw[hedge] (d) -- (X0) node[midway,hvertex,label=right:$s$](s){};
\begin{scope}[on background layer]
\draw[part] (d) ++ (0,0.9) circle (12pt);
\end{scope}
\end{scope}

\begin{scope}[shift={(90+2*72:2.6*\radius)}]
\draw[hedge] (60:\hh) node[hvertex] (a){} -- (180:\hh) node[hvertex] (b){} -- (-60:\hh) node[hvertex](c){} -- (a) -- (2*\bb+\hh,0) node[hvertex] (d){} -- (c);

\draw[hedge] (b) -- ++(0,1) node[hvertex,label=right:$X$] {};
\path (d) -- ++(0,1) node[hvertex,label=left:$X$] (X1) {};
\draw[hedge] (d) -- (X1) node[near start,hvertex,label=right:$s$](s){}
coordinate[midway] (H1) node[near end,hvertex,label=right:$t$](t){};
\begin{scope}[on background layer]
\draw[part] (H1) circle (10pt);
\end{scope}
\end{scope}

\begin{scope}[shift={(90+3*72:2.6*\radius)}]
\node[hvertex] (a) at (60:\hh){};
\node[hvertex] (b) at (180:\hh){};
\node[hvertex] (c) at (-60:\hh){};
\node[hvertex] (d) at (2*\bb+\hh,0){};
\draw[hedge] (a) -- (b) -- (c) node[pos=0.3,hvertex,label=below:$s$](s){}
coordinate[midway] (H2) node[pos=0.7,hvertex,label=below:$t$](t){} -- (a) -- (d) -- (c);
\draw[hedge] (b) -- ++(0,1) node[hvertex,label=right:$X$] {};
\draw[hedge] (d) -- ++(0,1) node[hvertex,label=left:$X$] {};
\begin{scope}[on background layer]
\draw[part] (H2) circle (10pt);
\end{scope}
\end{scope}

\begin{scope}[shift={(90-72:2.6*\radius)}]
\node[hvertex] (a) at (60:\hh){};
\node[hvertex] (b) at (180:\hh){};
\node[hvertex] (c) at (-60:\hh){};
\node[hvertex] (d) at (2*\bb+\hh,0){};
\draw[hedge] (a) -- (b) -- (c) -- (a) node[near start,hvertex,label=right:$s$](s){}
coordinate[midway] (H3) node[near end,hvertex,label=right:$t$](t){} -- (d) -- (c);
\draw[hedge] (b) -- ++(0,1) node[hvertex,label=right:$X$] {};
\draw[hedge] (d) -- ++(0,1) node[hvertex,label=left:$X$] {};
\begin{scope}[on background layer]
\draw[part] (H3) circle (10pt);
\end{scope}
\end{scope}

\begin{scope}[shift={(90:2.3*\radius)}]
\draw[hedge] (60:\hh) node[hvertex] (a){} -- (180:\hh) node[hvertex] (b){} -- (-60:\hh) node[hvertex](c){} -- (a)  --
node[midway, hvertex,label=above:$s$](s){} (2*\bb+\hh,0) node[hvertex] (d){} -- node[midway, hvertex,label=below:$t$](t) {} (c);

\draw[hedge] (b) -- ++(0,1) node[hvertex,label=right:$X$] {};
\draw[hedge] (d) -- ++(0,1) node[hvertex,label=left:$X$] (X4){};
\begin{scope}[on background layer]
\draw[part,line width=8pt,rounded corners=3pt] (t.center) -- (s  |- X4) -- (X4.center) -- (X4 |- t) -- cycle;
\end{scope}

\end{scope}

\end{scope}

\end{tikzpicture}
\caption{Where the blueprints come from if $x$ is a subdividing vertex}
\label{fig:k4toblueprints2}
\end{figure}
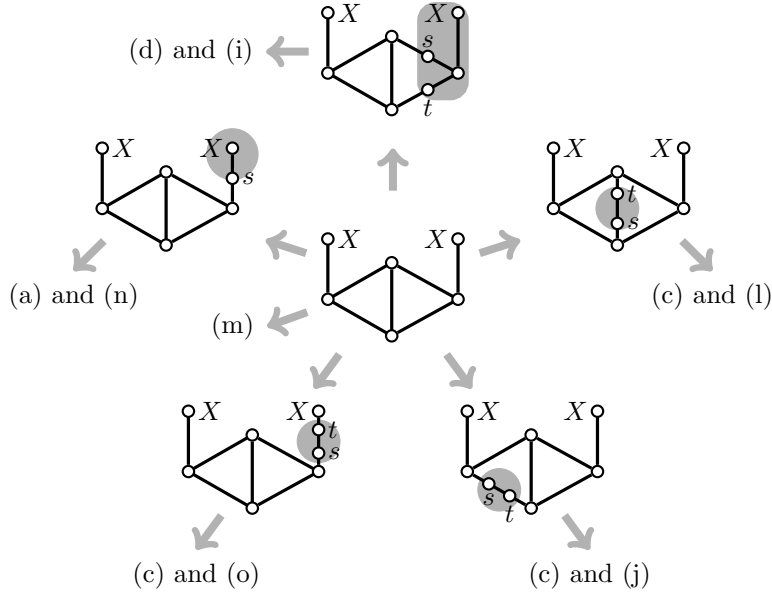

Now, if $S=\{s,t\}\cap V(K)$, then contracting edges in $M$ we arrive at the
basic blueprints as indicated in the figures. If not, then the terminals 
in $(\{s,t\}\cap V(K))\sm S$ will appear as accidental terminals. 
\end{proof}

We also fix an observation that is immediate from the definition of a module:

\begin{lemma}\label{smallerlem}
Let $H$ be a part, let $H'\subseteq H$ be a part and let $M$ be a module of $H$.   
If $M'$ is obtained from $M\cap H'$ by deleting isolated vertices, then
 $M'$ is a module of $H'$.
\end{lemma}

\begin{lemma}\label{replacelem}
Let $H$ be a part with terminals $s,t$,  
let $M$ be a module of $H$, and let $P\subseteq M$ be a path 
such that no interior vertex lies in $\{s,t\}\cup X$ or has degree~$3$ or larger in $M$.
Let $P'\subseteq H$ be a path between the endvertices of $P$ that is internally disjoint from $M-P$
and also from  $\{s,t\}$. Then replacing $P$ by $P'$ in $M$ results in a module $M'$ of $H$
that has the same blueprint as $M$.
\end{lemma}
\begin{proof}
Let  $K$ be a $K_4$-subdivision such that $M$ arises from $K\cap H$ by deleting isolated vertices. 
Then, if we replace $P$ by $P'$ in $K$, we obviously obtain a $K_4$-subdivision $K'$.
(Here it matters that $P'$ is internally disjoint from $M-P$ and from $\{s,t\}$.)
As $M'$ is equal to $K'\cap H$, up to isolated vertices, we see that $M'$ is a module. 
That $M$ and $M'$ have the same blueprint follows from the fact that $K$ and $K'$ differ
only in the interior of a subdivided edge, and only along a path inside $H$ that avoids its terminals.
\end{proof}

\begin{lemma}\label{replacelem2}
Let $H_1$ be a part, let $H_2\subseteq H_1$ be a part with terminals $s_2,t_2$,
and let $M_1$ be a module of $H_1$. Let $M_2$ be the module of $H_2$ obtained from $M_1\cap H_2$ by deleting isolated vertices, 
let $B_2$ be the blueprint of $M_2$.
Let $M'_2$ be a subgraph of $H_2$ that has blueprint $B_2$ (a priori not necessarily a module).
Then 
\begin{enumerate}[\rm (i)]
\item $M_1'=M_2'\cup (M_1-(H_2-\{s_2,t_2\}))$ has blueprint $B_1$; and
\item $M_2'$ and $M_1'$ are modules.
\end{enumerate}
\end{lemma}
\begin{proof}
Set $\overline H_2:=H_1-(H_2-\{s_2,t_2\})$, where we additionally 
omit $s_2t_2$ from $\overline H_2$ if it is an edge of $H_2$.
For a subdivision $M$ we write $U(M)$ for the set of branch vertices. 
If $M_1$ meets $s_2$ but $s_2$ is not already a branch vertex of $B_1$, then 
subdivide the corresponding edge of $B_1$, and do the same for $t_2$. In this 
way we obtain a new (labelled) graph $B'_1$ of which $M_1$ is a subdivision.
Let us first  show that $M'_1$ is also a subdivision of $B'_1$.
Note that, ignoring labels, we may see $B_2$ as a subgraph of $B'_1$.

For a vertex $u$ of $B'_1$ denote by $u^*$ the corresponding branch vertex in $M_1$.
As $M_2$ and $M_2'$ have the same blueprint,
there is a bijection $\phi_2:U(M_2)\to U(M_2')$ such that for every $u\in V(B_2)$
the vertex $\phi_2(u^*)$ is the branch vertex corresponding to $u$ in $M_2'$. 
We extend $\phi_2$ via the identity on $U(M_1\cap V(\overline H_2))$ 
to an injective function $\phi_1$ with domain $U(M_1)$. We claim that $\phi_1(U(M_1))$
as the set of branch vertices makes $M'_1$ a subdivision of $B'_1$.

Consider an edge $uv$ in $B'_1$. Since $M_1$ is a subdivision of $B'_1$
it contains a subdivided edge between $u^*$ and $v^*$, i.e.\ a
$u^*$--$v^*$-path $P$ whose internal vertices all have degree~$2$ in $M_1$.
Moreover, as each vertex in $\{s_2,t_2\}\cap V(M_1)$ is a branch vertex of $M_1$, 
it follows that $P\subseteq \overline H_2$ or $P\subseteq H_2$. 
In the latter case, it follows that $P\subseteq M'_1$ as $M'_1\cap\overline H_2=M_1\cap\overline H_2$,
and we see that $P$ is a subdivided edge between $\phi_1(u^*)$ and $\phi_1(v^*)$ in $M'_1$.
If, on the other hand, $P\subseteq H_2$, then $P$ is a subdivided edge between the 
two branch vertices $u^*$ and $v^*$ of $M_2$. Therefore, $M'_2\subseteq M'_1$ contains a $\phi_1(u^*)$--$\phi_1(v^*)$-path $P'$ whose internal vertices have degree~$2$ in $M'_2$ anthen also in $M'_1$. 

As $u\in X$ if and only if $\phi_2(u)\in X$ for every $u\in U(M_2)$, and 
as $M_1$ and $M'_1$ coincide on $\overline H_2$, it follows that $M_1'$ contains 
a terminal of $H_1$ or a vertex in $X$ if and only if the corresponding vertex in $B_1$
bears the appropriate label. Therefore, $M_1$ and $M'_1$ have the same blueprint. 

Let $K$ be a $K_4$-subdivision such that $M_1$ is obtained from $M_1\cap H_1$ by deleting 
isolated vertices. Then, replacing $M_2$ with $M_2'$ in $K$ we obtain another $K_4$-subdivision $K'$
(this can be seen via $\phi_2$ as above). Moreover,
 $M'_1$ is equal to $K'\cap H_1$, up to isolated vertices, and thus $M'_1$ is a module.
\end{proof}

The previous lemmas show that modules are compatible with parts and thus with series and parallel decomposition. 
This is summarised as follows

\begin{lemma}\label{legolem}
Let $H$ be a part, and let $H=H_1\cup H_2$ be a series or parallel decomposition of $H$ into parts. 
Let $M$ be a module of $H$ and for $i=1,2$, let $M_i$ be obtained from $M\cap H_i$ by deleting isolated vertices. Then
\begin{enumerate}[\rm (i)]
\item $M_i$ is either empty or a module of $H_i$, with blueprint $B_i$ say, for $i=1,2$; and 
\item if $M'_i$ has blueprint $B_i$ in $H_i$ for $i=1,2$, then $M'_1\cup M'_2$
is a module of $H$ with the same blueprint as $M$. 
\end{enumerate} 
\end{lemma}
\begin{proof}
(i) follows directly from Lemma~\ref{smallerlem}.
(ii) follows from two applications of Lemma~\ref{replacelem2} (i).
\end{proof}

If $M$ is a module with a blueprint that is derived from basic blueprints (f), (j), (k), (l) or (o),
$M$ is called an \emph{exceptional} module. A look at the blueprints shows that 
\begin{equation}\label{twoDisjPaths}
\begin{minipage}[c]{0.8\textwidth}\em
	an exceptional module in a part $H$ with terminals $s,t$ 
	contains two disjoint $\{s,t\}$--$X$-paths.
\end{minipage}\ignorespacesafterend 
\end{equation} 
In the next two lemmas we make two observations about exceptional 
and unexceptional modules that will become important much later, in Lemma~\ref{type5lem}. 

%

A module $M$ of a part $H$ is \emph{hit} by an edge set $F\subseteq E(G)$ 
if every $K_4$-subdivision $K$ with $E(K\cap H)=E(M)$ meets $F$.
In particular, if $F$ meets $M$ then it hits it. 
There are other ways, however, how $F$ can hit $M$:

\begin{lemma}\label{hittinglem}
Let $H$ be a part with terminals $s,t$, and 
let $M$ be an unexceptional module of $H$. 
Let $y\in \{s,t\}$ be used non-accidentally in $M$.
If $F\subseteq E(G)$ meets every $y$--$x$-path that is internally disjoint from $H-\{s,t\}$
then $F$ hits $M$.
\end{lemma}
\begin{proof}
	Let $K$ be a $K_4$-subdivision such that $M$ is obtained from $K\cap H$ by deleting isolated vertices.
	By checking the list of blueprints of unexceptional modules,
see in particular Figures~\ref{fig:k4toblueprints1} and~\ref{fig:k4toblueprints2},
	we see that  $K-(M- \{s,t\})$
is connected and contains~$x$.
\end{proof}

\begin{lemma}\label{exceptlem}
	Let $H$ be a part with terminals $s,t$ and $H_2\subseteq H$ 
	be  a part with terminals $s',t'$.
Let $H_1$ be a subgraph of $H$ such that $H$  
is the edge-disjoint union of $H_1$ and $H_2$, and 
$H_1\cap H_2=\{s',t'\}$.
	If $M$ is an exceptional module in $H$,
	then $M\cap H_1$ is the disjoint union of two $\{s,t\}$--$\{s',t'\}$-paths.
\end{lemma}
\begin{proof}
The vertices $s',t'$ separate $s,t$ from $X$ in $M$ as $H_1$ does not contain any vertex of $X$. Considering Figure~\ref{basicfig}, we see that this immediately implies
that  $M\cap H_1$ is the disjoint union of two $\{s,t\}$--$\{s',t'\}$-paths,
unless the blueprint of $M$ is derived from (l) or (o).
If, in that case, $M\cap H_1$ is not as desired, then
we find ourselves in the situation as shown in Figure~\ref{fig:excepts}.
	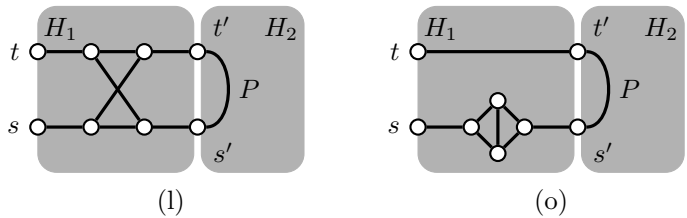
\begin{figure}[ht]
		\centering
		\begin{tikzpicture}
\tikzstyle{hinter}=[rounded corners=7, color=hellgrau, fill=hellgrau]
\def\step{0.7}
\def\ystep{1}
\def\pad{0.6}

\draw[hinter] (0,2*\ystep+\pad) rectangle (3*\step-0.05,\ystep-\pad);
\draw[hinter] (5*\step,2*\ystep+\pad) rectangle (3*\step+0.05,\ystep-\pad);

\node[hvertex,label=left:$s$] (s) at (0,\ystep){};
\node[hvertex,label=left:$t$] (t) at (0,2*\ystep){};
\node[hvertex] (a) at (\step,\ystep){};
\node[hvertex] (b) at (\step,2*\ystep){};
\node[hvertex] (c) at (2*\step,\ystep){};
\node[hvertex] (d) at (2*\step,2*\ystep){};
\node[hvertex,label=below right:$s'$] (ss) at (3*\step,\ystep){};
\node[hvertex,label=above right:$t'$] (tt) at (3*\step,2*\ystep){};
\draw[hedge] (s) -- (a) -- (c) -- (ss);
\draw[hedge] (t) -- (b) -- (d) -- (tt);
\draw[hedge] (a) -- (d);
\draw[hedge] (b) -- (c);
\draw[hedge,out=0, in=0] (tt) to node[auto,midway] {$P$} (ss);

\node at (0.5*\pad,2*\ystep+0.5*\pad) {$H_1$};
\node at (5*\step-0.5*\pad,2*\ystep+0.5*\pad) {$H_2$};
\node at (2.5*\step,0) {(l)};

\begin{scope}[shift={(5,0)}]
\draw[hinter] (0,2*\ystep+\pad) rectangle (3*\step-0.05,\ystep-\pad);
\draw[hinter] (5*\step,2*\ystep+\pad) rectangle (3*\step+0.05,\ystep-\pad);

\node[hvertex,label=left:$s$] (s) at (0,\ystep){};
\node[hvertex,label=left:$t$] (t) at (0,2*\ystep){};
\node[hvertex] (a) at (\step,\ystep){};
\node[hvertex] (b) at (2*\step,\ystep){};
\node[hvertex] (d) at (1.5*\step,\ystep+0.5*\step){};
\node[hvertex] (e) at (1.5*\step,\ystep-0.5*\step){};
\node[hvertex,label=below right:$s'$] (ss) at (3*\step,\ystep){};
\node[hvertex,label=above right:$t'$] (tt) at (3*\step,2*\ystep){};

\draw[hedge] (s) -- (a) (b) -- (ss);
\draw[hedge] (a) -- (d) -- (b) -- (e) -- (a);
\draw[hedge] (d) -- (e);
\draw[hedge] (t) -- (tt);
\draw[hedge,out=0, in=0] (tt) to node[auto,midway] {$P$} (ss);
\node at (0.5*\pad,2*\ystep+0.5*\pad) {$H_1$};
\node at (5*\step-0.5*\pad,2*\ystep+0.5*\pad) {$H_2$};
\node at (2.5*\step,0) {(o)};
\end{scope}

		\end{tikzpicture}
		\caption{Obtaining an $s$--$t$-diamond in the proof of Lemma \ref{exceptlem}.}
		\label{fig:excepts}
	\end{figure}
	Since every part is connected by~\eqref{partProps},
	there is an $s'$--$t'$-path $P$ in $H_2$.
	Then,  $(M\cap H_1) \cup P$ is an $s$--$t$-diamond in $H-X$,
which  contradicts Lemma~\ref{diamonds}~(ii) because $H-X$ is series-parallel,
by Lemma~\ref{partsplem}.
\end{proof}

An edge set $F\subseteq E(G)$ is a \emph{hit-or-miss set for $H$} if for any module $M$ of $H$ that 
is not hit by $F$ there are $k$ edge-disjoint modules in $H$ that have the same blueprint as $M$.
Note that a hit-or-miss set for all of $G-x$ is either an edge hitting set, or indicates that $G$ contains 
$k$ edge-disjoint $K_4$-subdivisions.

A part $H$ is \emph{simple} if it is disjoint from $X$ or if it is trivial, i.e.\ if 
$H-X$ consists of a single vertex.


\begin{lemma}\label{simplelem}
Let $H$ be a simple part, and let $M$ be a module of $H$. Then $M$ is either a path between the terminals
of $H$  or it is an edge between the terminal of $H$ and a vertex in $X$.
In particular, $H$ has a hit-or-miss set of size at most~$k-1$. 
\end{lemma}
\begin{proof}
If $H$ is disjoint from $X$, then,
by consulting the list of basic blueprints in Figure~\ref{basicfig},
we see that the only possible modules of $H$ are paths 
between the terminals of $H$. 
If it is possible to separate the terminals with at most $k-1$ edges, then 
such an edge set is a hit-or-miss set; otherwise the empty set is a hit-or-miss set.

If $H-X$ consists of a single vertex $s$, then $H$ is a star with centre $s$ and all leaves in $X$. 
If $H$ has at least $k$ vertices in $X$, then the empty set is a hit-or-miss set; otherwise
$E(H)$ is a hit-or-miss set.
\end{proof}

\begin{lemma}\label{substantiallem}
Let $M$ be a module of a part $H$. If $H$ has no substantial block-part, then $M$ is 
either a comb, an appendix or  a double appendix.
\end{lemma}
\begin{proof}
Any neighbour of $X$ in $H$ is either a cutvertex of $H-X$, 
or a terminal of $X$. 
Therefore, there is a sequence $B_1,\ldots, B_\ell$  of edge-disjoint subgraphs of $H$
such that $H=\bigcup_{i=1}^\ell B_i$ is an iterated series decomposition of parts, 
and such that every $B_i$ is either a block of $H$ (and then disjoint from $X$),
 or a star whose centre is a cutvertex of $H-X$ and whose leaves lie in $X$.
In other words, every $B_i$ is a simple part.
It follows from Lemma~\ref{simplelem} that each $B_i\cap M$ is a path. 
Consequently, $M-X$
is a forest, and a look at the basic blueprints in Figure~\ref{basicfig} confirms that $M$
must be a comb, an appendix or a double appendix.
\end{proof}

\subsection{Ear induction}
\label{earinduction}

We start with a tool with which we can inductively construct hit-or-miss sets.

\begin{lemma}\label{hitormisslem}
Let $H$ be a part, and let $H=H_1\cup H_2$ be a series or parallel decomposition of $H$
into parts $H_1$ and $H_2$. If $F_1,F_2$ are hit-or-miss sets for $H_1$ and $H_2$, respectively, 
then $F_1\cup F_2$ is a hit-or-miss set for $H$.
\end{lemma}
\begin{proof}
Let $M$ be a module of $H$ that is not hit by $F_1\cup F_2$,  
that is, there is a $K_4$-subdivision $K$ with $M\subseteq K\subseteq G-(F_1\cup F_2)$.
If $M\subseteq H_i$ for some $i=1,2$, the fact that the hit-or-miss set $F_i$ for $H_i$ 
does not hit $M$ implies that there are $k$ edge-disjoint modules of $H_i$ that have the same blueprint as $M$.
They are also modules of $H$ of the same blueprint as $M$.

Therefore we may assume that $M$ contains an edge of both, $H_1$ and $H_2$.
For $i=1,2$, let $M_i$ be obtained from $M\cap H_i$ by deleting isolated vertices.
By Lemma~\ref{legolem} (i), $M_i$ is a module of $H_i$ for $i=1,2$.
Because $M_i\subseteq M\subseteq K\subseteq G-F_i$ holds,
the set $F_i$ does not hit $M_i$ for $i=1,2$.
As $F_i$ is a hit-or-miss set for $H_i$, 
there are $k$ edge-disjoint modules $M_i^1,\ldots, M_i^k$ in $H_i$
that have the same blueprint as $M_i$, for $i=1,2$.
By Lemma \ref{legolem} (ii), for every $j=1,\ldots, k$, 
the graph $M^j=M_1^j\cup M_2^j$ is a module that has the same blueprint as $M$.
The modules $M^1,\ldots, M^k$ are edge-disjoint because $H_1$ and $H_2$ are edge-disjoint and 
because the smaller modules $M_1^j,\ldots, M_k^j$ are edge-disjoint for $j=1,2$. 
This proves that $F_1\cup F_2$ is a hit-or-miss set for $H$.
\end{proof}

Let $H$ be a part with terminals $s\neq t$, and let $B$ be a block-part of $H$.
As defined in Section~\ref{LBRBdef}, we use $L_B$ and $R_B$
to decompose $H=L_B\cup B\cup R_B$ twice in series into edge-disjoint parts.
Moreover, we say that $H$ has
\begin{itemize}
\item \emph{type I:} if $H$ does not have any substantial block-part;
\item \emph{type II:} if $H$ has at least two substantial block-parts;
\item \emph{type III:} if 
$H$ has exactly one substantial block-part $B$ that parallel decomposes into $B=B_1\cup B_2$ such that 
$B_2$ is non-substantial, and if at least one of $L_B-s,R_B-t$ contains a vertex that is adjacent to $X$;
\item \emph{type IV:} if $H$ has exactly one substantial block-part $B$ that parallel decomposes into 
two substantial parts $B=B_1\cup B_2$; and 
\item \emph{type V:} if $H$ has not type~I,~II,~III or~IV.
\end{itemize}

Note that if the $x$-ear number $\lambda$ of $H$ is 0, then $H$ is type I 
and $\lambda=1$ implies that $H$ is type I or V.
Define 
\[
f_*(\lambda,k)=\begin{cases}
(\lambda-2)\cdot 2000k^3+1000k^3 &\emtext{for }\lambda\geq 2\\
6k^2&\emtext{for }\lambda\leq 1
\end{cases}
\]
and 
\[
f(\lambda,k)=
f_*(\lambda,k)+250k^3
\]
%
The following lemma is the main lemma in this section.

\begin{lemma}\label{fewlem}
Let $H$ be a part with $x$-ear number~$\lambda$. Then
there is a hit-or-miss set $F$ for $H$ such that 
$|F|\leq f_*(k,\lambda)$ if $H$ has type~I--IV and such that
$|F|\leq f(k,\lambda)$ if $H$ has type~V. 
\end{lemma}

Before we prove the lemma, let us first show how we can finish the proof 
of  Lemma~\ref{mainfewlem} with it.

\begin{proof}[Proof of Lemma \ref{mainfewlem}]
Let $\lambda$ be the $x$-ear number of $G$.
By Lemma~\ref{fewlem} there is a hit-or-miss set $Y$ for the part $G-x$ of size at most $2000\lambda k^3$.
If $Y$ meets every $K_4$-subdivision of $G$, then we are done. So,
let $K$ be a $K_4$-subdivision in $G-Y$.
Then $K-x$ is a module of $G-x$ (either with a blueprint as in Figure~\ref{basicfig}~(g),
or a rooted diamond with blueprint as in~(m))
that is not hit by $Y$.
Therefore, there are $k$ edge-disjoint modules that have the same blueprint as $K-x$
and thus, as the vertices in $X$ have degree~$2$, these can be completed to 
 $k$ edge-disjoint $K_4$-subdivisions.
\end{proof}

We start  with the proof of Lemma~\ref{fewlem}, the final
piece in the proof of the main theorem.
We proceed by induction on $\lambda$
and  treat each  type of $H$ in a separate lemma.
In the first lemma we deal with the 
 base case, when the $x$-ear number of $H$ is~$0$.

\begin{lemma}\label{type1lem}
Let $H$ be a part of type~I. Then
there is a hit-or-miss set $F$ for $H$ of size~$|F|\leq 6k^2$. 
\end{lemma}
\comment{Note if we do this lemma a bit more carefully, we should get a linear bound here. 
How is this done? In the second case, when there are $k$ edge-disjoint $s$--$t$-paths, check the following quantity:
\[draw
\sum_{v\in V(H-\{s,t\})}\min(k,|N_H(v)\cap X|)
\]
If that quantity is at least $3k$, then we find $k$ edge-disjoint $3$-combs. }
\begin{proof}
Let $s,t$ be the terminals of $H$.
Note that, by Lemma~\ref{substantiallem}, any non-trivial module of $H$ is a comb, an appendix or a double appendix
(modules of blueprints (a)--(f) in Figure \ref{basicfig}).

Assume first that there is an edge set $F_1$ of size $|F_1|<k$ that separates $s$ from $t$ in $H$. 
Then $F_1$ hits any comb of $H$, as well as any appendix that contains both terminals (one 
of them accidentally).  Thus, the only modules we still need to consider are appendices that do not accidentally
contain the other terminal as well as double appendices.

If there is a set  of at most $k-1$ edges
separating $s$ from $X$ in $H-F_1$, then denote the set by $F_2$; otherwise put $F_2=\emptyset$. 
Define $F_3$ for $t$ in the analogous way. Then either $F=F_1\cup F_2\cup F_3$ hits every $s$--$X$-path in 
$H$ or there are $k$ edge-disjoint ones, and the same holds for $t$--$X$-paths. 
Assume that there is an $s$-appendix $M$ in $H$ (i.e. a $K_4$-subdivision $K$ meets $H$ exactly in an $s$--$X$-path).
Then, Lemma \ref{replacelem2} (ii) shows that every other $s$--$X$-path in $H$ (that does not contain $t$ accidentally)
is a module (an $s$-appendix). 
Thus, there are $k$ edge-disjoint $s$-appendices in $H$. The same holds for $t$ instead of $s$.

Consider a double appendix $M$ of $H$ that is not  hit by $F$. 
By the choice of $F$, there must be $k$ edge-disjoint $s$--$X$-paths in $H-F_1$ and also $k$ edge-disjoint $t$--$X$-paths in $H-F_1$,
and none of the $s$--$X$-paths can meet any of the the $t$--$X$-paths as $F_1$ separates $s$ from $t$ in $H$. 
The union of any $s$--$X$-path and any $t$--$X$-path is a double appendix by Lemma \ref{replacelem2} (ii).
Hence, $H$ contains $k$ edge-disjoint double appendices.
Thus $F$ is a hit-or-miss set for $H$
of size $|F|\leq 3k$.

\medskip
Second assume that there are $k$ edge-disjoint $s$--$t$-paths in $H$. 
Recall that every neighbour of $X$ in $H$ is a  cutvertex of $H$ or a terminal
and therefore there is an $s$--$t$-path in $H$ that contains every neighbour of $X$.
If $H$ contains at least $3k+2$ neighbours of $X$, we let $F$ consist of all edges between $s$ and $X$
if there are at most $k-1$ of them and of all $t$--$X$-edges if there are at most $k-1$ of them;
otherwise we set $F=\emptyset$.
Then, $F$ has size at most $2k$ and we claim that $F$ is a hit-or-miss set in this case.

As $H$ contains $3k$ neighbours of $X$ which are not adjacent to $s$ or $t$,
the graph $H$ contains $k$ edge-disjoint  subdivisions of the blueprint~(f) in Figure \ref{basicfig} 
in which the terminals each have  degree~$1$ (neither of the dashed edges is contracted)
and by Lemma \ref{replacelem2} (ii) they are $3$-combs if there is any $3$-comb in $H$.
Any such $3$-comb contains any module of blueprint (a)--(f) where the dashed edges are not contracted
if there is at least one such module in $H$ (again by Lemma \ref{replacelem2} (ii)).

%

Any other module in $H$ thus contains an $\{s,t\}$--$x$-edge and is either hit by $F$ or we find $k$ such edges.
Let $B$ be any blueprint of (a)--(f) where  (some or all of) the dashed edges are contracted
and such that $H$ contains a module of blueprint $B$.
Then, any $3$-comb together with an $s$--$x$-edge and/or a $t$--$x$-edge contains 
another module of  blueprint $B$.

Therefore, we may assume that $H-X$ has at most $3k+1$ vertices that have a neighbour in $X\cap V(H)$.
Let the set of these be $v_1,\ldots, v_\ell$, and assume them to be enumerated
in the order they appear on an $s$--$t$-path $R$ in $H$. For each $v_i$, let $P_i$ be 
the simple part consisting of $v_i$ and its neighbours in $X\cap V(H)$ together with 
the edges between them. For $i=1,\ldots,\ell-1$, let $Q_i$ be the union of all blocks
in $H-X$ that share an edge with $v_iRv_{i+1}$. Then $Q_i$ is a simple part, too, 
and $H=\bigcup_{i=1}^\ell P_i\cup\bigcup_{i=1}^{\ell-1}Q_i$ is an iterated series decomposition
into simple parts. (Note that $G-X$ is $2$-connected, by~\eqref{standardG}.)
By Lemma~\ref{simplelem}, there is a hit-or-miss set of size at most $k-1$ in each of these
simple parts, and by Lemma~\ref{hitormisslem} their union $F$ is a hit-or-miss set of $H$, which 
then has size $|F|\leq (3k+1)\cdot 2\cdot (k-1)\leq 6k^2$.
\end{proof}

The proofs for types II and IV need similar calculations that are extracted in the following lemma.

\begin{lemma}\label{rechnung}
Let $\lambda_1,\lambda_2,\lambda$ be non-negative integers.
Let $\lambda_i<\lambda$ and either $2\le \lambda_i$ and $2\le \lambda_1+\lambda_2\le \lambda+1$, 
or $1\le \lambda_i$ and $2\le \lambda_1+\lambda_2\le \lambda$.
Then $f(\lambda_1,k)+f(\lambda_2,k)\le f_*(\lambda,k)-12k^2$.
\end{lemma}
\begin{proof}
Assume first that $\lambda_1,\lambda_2\geq 2$.  Then:
\begin{align*}
f(\lambda_1,k)+f(\lambda_2,k)&\leq (\lambda_1-2)2000k^3+1250k^3+(\lambda_2-2)2000k^3+1250k^3\\
&\leq(\lambda-2)2000k^3+2500k^3-2000k^3\leq f_*(\lambda,k) -12k^2,
\end{align*}
as $\lambda_1+\lambda_2\leq\lambda+1$.

Second, assume that one of $\lambda_1,\lambda_2$ is at most one, and that the other is at least~$2$. 
Then, by assumption, $3\le \lambda_1+\lambda_2\leq\lambda$. Thus
\begin{align*}
f(\lambda_1,k)+f(\lambda_2,k)&\leq f(1,k)+f(\lambda-1,k) = 
256k^2+(\lambda-3)2000k^3+1250k^3\\
&\leq(\lambda-2)2000k^3+1506k^3-2000k^3\leq f_*(\lambda,k) - 12k^2
\end{align*}

Finally, assume that $\lambda_1,\lambda_2\leq 1$. Then 
\[
f(\lambda_1,k)+f(\lambda_2,k)\leq 256k^3+256k^3\leq 1000k^3-12k^2=f_*(2,k)-12k^2
\]
Note that $\lambda\geq 2$, by assumption.
\end{proof}

\begin{lemma}\label{type2lem}
Let $H$ be a part with $x$-ear number~$\lambda$ of type~II. Then
there is a hit-or-miss set $F$ for $H$ of size~$|F|\leq f_*(\lambda,k)$. 
\end{lemma}
\begin{proof}
As $H$ is of type~II, there is a series decomposition
 $H=H_1\cup H_2$ such that each part $H_i$ contains a substantial block-part.
If $\lambda_i$ is the $x$-ear number of $H_i$, for $i=1,2$, we see with 
 Lemma~\ref{splitearlem}~\eqref{substi} that $\lambda_1,\lambda_2$
satisfy the conditions of Lemma~\ref{rechnung}.
As $\lambda_1,\lambda_2<\lambda$, we can apply induction in $H_1$ and in $H_2$,
and get hit-or-miss sets $F_1,F_2$ for $H_1,H_2$ of sizes 
$|F_1|\leq f(\lambda_1,k)$ and $|F_2|\leq f(\lambda_2,k)$. 
Then $F=F_1\cup F_2$ is a hit-or-miss set
for $H$, by Lemma~\ref{hitormisslem}, of size
\[
|F|\leq f(\lambda_1,k)+f(\lambda_2,k)\leq f_*(\lambda,k),
\]
by Lemma~\ref{rechnung}.
\end{proof}

\begin{lemma}\label{type3lem}
Let $H$ be a part with $x$-ear number~$\lambda$ of type III. Then
there is a hit-or-miss set $F$ for $H$ of size~$|F|\leq f_*(\lambda,k)$. 
\end{lemma}
\begin{proof}
Let $B$ the unique substantial block-part of $H$ and
let $L_B,R_B$ be defined as before.
We can apply Lemma \ref{splitearlem}~\eqref{terrible}
and get $\lambda_B\leq\lambda-1$ and $\lambda\geq 2$. 
Induction yields a hit-or-miss set $F_B$ for $B$ of size $|F_B|\le f(\lambda-1,k)$.

As $B$ is the unique substantial block-part,
it follows that  $L_B,R_B$ are of type~I. Thus, by Lemma~\ref{type1lem},
we obtain hit-or-miss sets $F_L,F_R$ for $L_B,R_B$ of sizes $|F_L|,|F_R|\le 6k^2$.
By Lemma~\ref{hitormisslem}, the union $F=F_B\cup F_L\cup F_R$ is a hit-or-miss set for $H$ of size
$f(\lambda-1,k)+2\cdot 6k^2$.
If $\lambda=2$, then $f(\lambda-1,k)\le 256k^3 \le 1000k^3-744k^3 \le f_*(\lambda,k)-12k^2$
and if $\lambda\ge 3$, then $f(\lambda-1,k)=2000k^3 (\lambda-3) + 1250k^3 \le 2000k^3 (\lambda-2) +1000k^3 - 750k^3 \le f_*(\lambda,k)-12k^2$. Thus, $|F|\le f(\lambda-1,k)+2\cdot 6k^2 \le f_*(\lambda,k)$.
\end{proof}

\begin{lemma}\label{type4lem}
Let $H$ be a part with $x$-ear number~$\lambda$ of type~IV. Then
there is a hit-or-miss set $F$ for $H$ of size~$|F|\leq f_*(\lambda,k)$. 
\end{lemma}
\begin{proof}
Let $B$ be the unique substantial block-part of $H$.
Then $L_B$ and $R_B$ do not contain any substantial block-parts.
By Lemma~\ref{type1lem}, there are thus hit-or-miss sets $F_L$ and $F_R$
for $L_B$ and $R_B$ such that $|F_L|,|F_R|\leq 6k^2$.
Below we will find a  hit-or-miss set $F_B$ for $B$ of size at most $|F_B|\leq f_*(\lambda,k)-12k^2$.
Then, by two applications of Lemma~\ref{hitormisslem} we see that $F_L\cup F_B\cup F_R$ is 
a hit-or-miss set for $H$ of size at most $f_*(\lambda,k)$.

So, let us construct $F_B$. Denote the $x$-ear number of $B$ by $\lambda_B$
and observe that $\lambda_B\leq\lambda$ as $B\subseteq H$.
By assumption, there is a parallel decomposition $B=B_1\cup B_2$ of $B$ into parts $B_1$
and $B_2$ such that the respective $x$-ear numbers $\lambda_1,\lambda_2$ are positive. 
By Lemma~\ref{splitearlem}~\eqref{para}, it follows that $2\leq\lambda_1+\lambda_2\leq\lambda_B\leq\lambda$.
In particular, we have $\lambda_1,\lambda_2<\lambda$, which means by induction 
we find a hit-or-miss set $F_1$ for $B_1$, and a hit-or-miss set $F_2$ for $B_2$
such that $|F_1|\leq f(\lambda_1,k)$ and $|F_2|\leq f(\lambda_2,k)$. Put $F_B=F_1\cup F_2$,
and observe that it is a hit-or-miss set for~$H$, by Lemma~\ref{hitormisslem}.

Now, we can apply Lemma~\ref{rechnung} to $\lambda_1,\lambda_2, \lambda$ and obtain
\(
|F_B|\le f_*(\lambda,k)-12k^2.
\)
\end{proof}

Unfortunately, the last case, type~V, needs a bit more work and we will need to prove
two more lemmas first.
The reader might notice that exceptional modules are excluded within the next two lemmas.
The first reason lies in the difference between exceptional and unexceptional modules that is explained in Lemma \ref{hittinglem} 
and the second reason in the fact, that a hit-or-miss set for exceptional modules
in type V can be easily constructed using Lemma \ref{exceptlem}
and therefore, we do not have to care about them.

We first define \emph{simple pseudo-parts} as a generalization of simple parts (compare Lemma \ref{simplelem})
and then, the next lemma plays the same role for simple pseudo-parts as Lemma \ref{simplelem} does for simple parts.
Note that when we speak about separators $\{a,b\}$ of $G-X$ in the next definition and Lemma \ref{pseudolem},
neither $\{a\}$ nor $\{b\}$ is a separator of $G-X$ because $G-X$ is $2$-connected by \eqref{standardG}.

Let $H$ be a part, and let $F\subseteq E(G)$ be an edge set. An induced subgraph $H'$ of $H$ is 
a \emph{simple pseudo-part} of $(H,F)$ if $H'$ has three terminals, $s'$, $t'$ and $w'$ 
such that both $\{s',t'\}$ and $\{w',t'\}$ are separators of $G-X$, 
if $H'$ is disjoint from $X$ 
and if
\begin{enumerate}[\rm (a)]
\item either $F$ separates $w'$ from $x$ in $G-E(H')$ and
there is an edge set $D\subseteq E(H')$ such that an $s'$--$t'$-path in $H'$ passes through $w'$ 
if and only if it is disjoint from $D$; 
\item or $F$ separates $w'$ from $s'$ in $H'-t'$.
\end{enumerate}
Even though pseudo-parts have three terminals they are quite similar 
to ordinary parts.
The next lemma shows how to construct one part of a hit-or-miss set in pseudo-parts.

\begin{lemma}\label{pseudolem}
Let $H$ be a part, $F$ an edge set, and let $H'$ be a simple pseudo-part of $(H,F)$
with terminals $s',t',w'$.
Then there is an edge set  $F'$ of size at most $2k$ such that: 
if $M$ is an unexceptional module of $H$ that contains an edge of $H'$ but is not hit by $F\cup F'$,
then  
\begin{enumerate}[\rm (i)]
\item  
$M\cap H'$ consists of a path $P$ between two terminals $r_1,r_2$ of $H'$
plus possibly the third terminal $r_3$ as isolated vertex; and
\item there are $k$ edge-disjoint $r_1$--$r_2$-paths $P_1,\ldots, P_k$ in $H'$
such that $r_3\in V(P)$ if and only if $r_3\in V(P_i)$ for all $i$.
\end{enumerate}
\end{lemma}
\begin{proof}
First note that $M\cap H'$ has a unique component that contains an edge.
If there were two of them, one of them would have to be an appendix because $H'$ has only three terminals.
This, however, is impossible as $H'$ contains no neighbour of $X$.

For the purpose of this proof, let us say that $H'$ is of \emph{$x$-type}
if $F$ separates $w'$ from $x$ in $G-E(H')$, and that it is of \emph{$s'$-type}
if $F$ separates $w'$ from $s'$ in~$H'-t'$.
Let $s,t$ be the terminals of $H$. We first show:
\begin{equation}\label{simplemod}
\begin{minipage}[c]{0.8\textwidth}\em
If $M$ is an unexceptional module of $H$ that meets $E(H')$ and
that is not hit by $F$, then the unique component $P$ of $M\cap H'$ with an edge is a path between 
two terminals of $H'$. 
\end{minipage}\ignorespacesafterend 
\end{equation} 

Let $M$ be such a module. 
First let  $M$ be a rooted diamond (Figure \ref{basicfig} (m), (n))
consisting of two cycles $C_1,C_2$ that intersect along a non-trivial path $Q$
with two disjoint paths $Q_1,Q_2$ from $C_1-Q$, $C_2-Q$
to $\{s,t\}\cup X$ and suppose that $C_1\cup C_2\subseteq H'$. 
We will show that there is $q_i \in V(Q_i)\cap\{s',t',w'\}$ for $i=1,2$ such that $\{q_1,q_2\}$
is a separator of $G-X$. 
As $G-X$ is $2$-connected (by \eqref{standardG}), then $G-x$ contains a $q_1$--$q_2$-path that avoids $H'$.
That, however, will lead to a contradiction to Lemma~\ref{diamonds} (i) as follows:
any such $q_1$--$q_2$-path together with $M$ yields a $K_4$-subdivision that is contained in $G-x$,
which is series-parallel.

To find such $q_1,q_2$ consider first the case when $H'$ is of $x$-type. 
Let $q_1$ be the last vertex on $Q_1$
such that $Q_1q_1\subseteq H'$ and define $q_2$ in the analogous way for $Q_2$.
Let $K$ be a $K_4$-subdivision such that $C_1\cup C_2\cup Q_1q_1 \cup Q_2q_2$
is obtained from $K\cap H'$ by deleting isolated vertices.
As $M$ is a rooted diamond, $K$ contains a $q_1$--$x$-path and $q_2$--$x$-path that are edge-disjoint from $H'$.
Since $F$ does not hit $M$ and since thus the ends of $Q_1,Q_2$ outside $H'$ cannot be separated from $x$
and since $H'$ is of $x$-type,
it follows that $\{q_1,q_2\}=\{s',t'\}$, which is a separator of $G-X$ by assumption.
If, on the other hand, $H'$ is of $s'$-type, then, as $F$ separates $w'$ from $s'$, and as both $Q_1,Q_2$
contain a vertex with neighbours outside $H'$, it follows that one of  $Q_1,Q_2$ meets $t'$,
while the other meets $w'$ or $s'$. Since both $\{s',t'\}$ and $\{w',t'\}$ are separators
of $G-X$ we are done with this case. In particular, if $M$ is a rooted diamond, then $M\cap H'$ 
cannot contain two cycles.

Second,  suppose there is a vertex $v$ in $M\cap H'$ that has degree~$3$ in $M$ and at least degree~$2$
in $H'$. If $M$ is not a rooted diamond then it contains three internally disjoint paths from $v$ to 
three vertices in $\{s,t\}\cup X$ such that none of these is separated by $F$ from $x$ in $G-E(H')$. 
In particular,  there are also three internally disjoint paths from $v$ to $\{s',t',w'\}$ since $H'$ is disjoint from $X$,
which is  impossible when $H'$ is of $s'$-type. If $H'$ is of $x$-type, then none of these paths 
can end in $w'$ as $F$ separates $w'$ from $x$ in $G-E(H')$, which means this is impossible as well. 
If $M$, on the other hand, is a rooted diamond, then, as $M\cap H'$ cannot contain two cycles, there are 
still three internally disjoint paths from $v$ to $s',t',w'$, and each of these cannot be separated from $x$ by $F$. 
Again, this is impossible.

Therefore, the unique component $P$ of $M\cap H'$ that contains an edge 
is either a path or a cycle. 
If $P$ is a cycle, however, then we see by inspecting the blueprints that $P$ must contain a vertex that has
degree~$3$ in $M$, which we already had excluded. 
Thus we have shown~\eqref{simplemod}.

Now let us define $F'$. First assume that $H'$ is of $x$-type. 
In particular, for any module $M$ as in~\eqref{simplemod} the path $P$ is an $s'$--$t'$-path that 
may pass through $w'$. Indeed, $P$ cannot start or end in $w'$ as $F$ separates $w'$ from $x$ in $G-E(H')$.
If it is possible to separate $s'$ from $t'$
in $H'-w'$ by at most $k$ edges, then let $F_1$ be such an edge set; otherwise set $F_1=\emptyset$.
If it is possible to separate $s'$ from $t'$ in $H'-D$ (where $D$ is as in the definition of a simple pseudo-part)
by at most $k$ edges, then let $F_2$ be such an edge set; otherwise put $F_2=\emptyset$. Then $F'=F_1\cup F_2$
is as desired.

Finally, consider the case when $H'$ is of $s'$-type. Again if $M$ is a module as in~\eqref{simplemod},
then the path $P$ is either an $s'$--$t'$-path, a $t'$--$w'$-path, or an $s'$--$w'$-path that passes through $t'$
(note that $F$ does not intersect $M$). Now, for $a\in\{s',w'\}$ if it is possible to separate $a$ from $t'$ 
in $H'-F$ by at most $k$ edges, then let $F_a$ be such a set, and otherwise set $F_a=\emptyset$. Then $F'=F_{s'}\cup F_{w'}$
is as desired. 
\end{proof}

We say that an edge set $F$  \emph{fails} a module $M$ of a part $H$
if $F$ does not hit  $M$ and if there are no 
$k$ edge-disjoint modules in $H$ that have the same blueprint as $M$. Clearly,  a set $F$ that 
does not fail any module of $H$ is a hit-or-miss set of $H$.

\begin{lemma}\label{badcase1lem}
Let $H$ be a part, and let $F$ be an edge set such that $H$
is the edge-disjoint union $H=L\cup B$
such that $B$ is a part, and such that $L$ is the edge-disjoint union of $r$ simple non-trivial parts or simple
pseudo-parts of $(H,F)$. If $F$ does not fail any  exceptional module of $H$ and if 
$B$ has a hit-or-miss set $F_B$,
then there is a hit-or-miss set for $H$ of size at most $|F\cup F_B|+2rk$.
\end{lemma}
\begin{proof}
Put $P_0=B$, $Y_0=F_B$ and let $P_1,\ldots, P_r$ be an enumeration of the simple parts and pseudo-parts
that make up $L$. We apply Lemma~\ref{simplelem} to every simple part in $P_1,\ldots, P_r$, 
and Lemma~\ref{pseudolem} to every simple pseudo-part in $P_1,\ldots, P_r$.
This results in $r$ edge sets
$Y_1,\ldots, Y_r$, each of size at most $2k$. We set $Y=F\cup Y_1\cup \ldots \cup Y_r \cup F_B$ and
observe that $|Y|\le |F\cup F_B|+2rk$. Let us check that $Y$ does not fail any module $M$ of $H$.
We may assume that $M$  is unexceptional and does not meet $Y$ as then $Y$ does not fail $M$
(note that $F\subseteq Y$).

Inductively, we define for $j=-1,0,\ldots, r$ modules $M^j_1,\ldots, M^j_k$ of $H$ such that 
all have the same blueprint as $M$ and such that $M^j_1,\ldots, M^j_k$ are edge-disjoint on $\bigcup_{s=0}^jP_j$,
and such that $M^j_i\cap P_s=M\cap P_s$ for $i=1,\ldots, k$ and $s=j+1,\ldots, r$.
Note that $M^r_1,\ldots, M^r_k$ will then be $k$ edge-disjoint modules of $H$ of the same blueprint as $M$,
which means that $Y$ does not fail $M$.

To start the induction, we set $M^j_1,\ldots, M^j_k:=M$ for $j=-1$. Assume now that $j\in\{0,\ldots, r\}$
and that we have achieved the construction for smaller $j$. 
Consider $M^{j-1}_1\cap P_j=\ldots=M^{j-1}_k\cap P_j=M\cap P_j$. If $M\cap P_j$ is edgeless, 
then we may simply put $M^j_i=M^{j-1}_i$ for $i=1,\ldots, k$. Thus, assume $M\cap P_j$ to contain an edge.

If $P_j$ is a  part then, 
after deletion of isolated vertices,  $M\cap P_j$ 
is a module of $P_j$ with blueprint $B'$, by Lemma~\ref{smallerlem}. Moreover, as $Y_j$ is a hit-or-miss set
there are $k$ edge-disjoint modules $N_1,\ldots, N_k$ of $P_j$ with blueprint $B'$. 
With Lemma~\ref{replacelem2} we may, for $i=1,\ldots, k$ replace $M^{j-1}_i\cap P_j$ by $N_i$
and thus obtain modules $M^j_1,\ldots, M^j_k$ as desired. 

If, on the other hand, $P_j$ is a simple pseudo-part of $(H,F)$ then, by Lemma~\ref{pseudolem},
$M\cap P_j$ is a path between two terminals of $P_j$, and there are $k$ edge-disjoint paths $R_1,\ldots, R_k$
in $P_j$ that meet the terminals of $P_j$ in the same way as $M\cap P_j$. With Lemma~\ref{replacelem} we may,
for $i=1,\ldots, k$ replace $M^{j-1}_i\cap P_j$ by $R_i$
and thus obtain modules $M^j_1,\ldots, M^j_k$ as desired. 
\end{proof}

\begin{lemma}\label{type5lem}
Let $H$ be a part with $x$-ear number~$\lambda$ of type~V. Then
there is a hit-or-miss set $F$ for $H$ of size~$|F|\leq f(\lambda,k)$. 
\end{lemma}
\begin{proof}
Let $s_0,t_0$ be the terminals of $H$. 
The terminals might be neighbours of vertices in $X$ within $H$. To exclude this case, 
we consider the simple parts  induced by $\{s_0\}\cup (N(s_0)\cap X)$ and by $\{t_0\}\cup (N(t_0)\cap X)$.
By Lemma~\ref{simplelem}, these have hit-or-miss sets of size at most $k$ each. Since they together with 
a hit-or-miss set for $H-(N(s_0,t_0)\cap X)$ form a hit-or-miss set for $H$ by Lemma \ref{hitormisslem}, 
we may from now on assume that
\begin{equation}
\emtext{
neither $s_0$ nor $t_0$ has a neighbour in $X$ within $H$, 
}
\end{equation}
provided we find a hit-or-miss set for $H$
that has size at most $f(\lambda,k)-2k$. (Which we will obviously do.)

Next, we will show that $H$ has a ladder-like structure (we will apply Lemma \ref{ladderfanlem} later).
Let $\ell$ be maximal such that there are 
internally disjoint parts
$Q_0,\ldots, Q_\ell$, $R_0,\ldots,R_\ell$, $S_1\ldots,S_{\ell}$ and $B^*$ such that
\begin{itemize}
\item there are vertices $s_1,\ldots,s_{\ell+1}$ such that consecutive vertices $s_i$ might coincide
and vertices $t_1,\ldots, t_{\ell+1}$ such that consecutive vertices $t_i$ might coincide and
such that $s_i\neq t_j$ for all $i,j\in\{1,\ldots,\ell+1\}$;
\item for $i=0,\ldots, \ell$ the terminals of $Q_i$ are $s_i$ and $s_{i+1}$, 
the terminals of $R_i$ are $t_i$ and $t_{i+1}$, and such that for $i=1,\ldots,\ell+1$
the terminals of $S_i$ are $s_i$ and $t_i$, and such that the part $B^*$ has  terminals 
$s_{\ell+1}$ and $t_{\ell+1}$;
\item the parts $S_1\ldots,S_{\ell}$ are non-trivial (but not necessarily the other parts); 
\item setting $L=\bigcup_{i=0}^\ell Q_i\cup R_i\cup \bigcup_{i=1}^\ell S_i$ we 
have $H=L\cup B^*$; and
\item $L$ is disjoint from $X$ (what makes every $Q_i,R_i,S_i$ a simple part).  
\end{itemize}
First note that there is indeed such an $\ell$ to begin with as we can always choose $\ell=0$,
$Q_0=\{s_0\}$,
$R_0=\{t_0\}$, $s_1=s_0$, $t_1=t_0$ and $B^*=H$. 

We claim that 
\begin{equation}\label{Bstar}
\emtext{$B^*$ has one of the types I--IV}.
\end{equation}
Suppose that $B^*$ has type~V.
As $B$ is neither of type~I, nor of type~II, 
it has a unique substantial block-part $B$, with terminals $s_{\ell+2},t_{\ell+2}$.
In particular, $B^*$ decomposes into parts $Q_{\ell+1},B,R_{\ell+1}$ 
such that $Q_{\ell+1}$ has terminals $s_{\ell+1}$
and $s_{\ell+2}$, and the part $R_{\ell+1}$ has terminals $t_{\ell+1}$ and $t_{\ell+2}$.
Moreover, since $B^*$ is not of type~IV the block-part $B$ has a
parallel decomposition $B=S_{\ell+1}\cup B'$ into parts $S_{\ell+1}$ and $B'$ such that 
$S_{\ell+1}$ is not substantial. Moreover, as $B^*$ is not of type~III, neither of $Q_{\ell+1},R_{\ell+1}$
can have a vertex in $N(X)$. Then, however, the new sequence of parts 
satisfies all the required conditions, but is longer. This contradicts the maximality of $\ell$
and proves~\eqref{Bstar}. 

\medskip
To save a bit on indices, we put $s=s_0$, $t=t_0$, $s'=s_{\ell+1}$ and $t'=t_{\ell+1}$.
By considering the different possible modules for $H$ and by recalling that $L$ is disjoint from $X$, 
we can see (by consulting Figure~\ref{basicfig}) that each module $M$ of $H$
falls into one of the following categories:
\begin{enumerate}[\rm (a)]
\item $M$ is an $s$--$t$-path;
\item $M\cap L$ is an $s'$--$t'$-path;
\item $M$ is exceptional; 
\item $M\subseteq B^*$; and
\item $M$ is neither of (a)--(d), and then contains a $\{s,t\}$--$\{s',t'\}$-path.
\end{enumerate}

By Lemmas~\ref{type1lem}--\ref{type3lem} and \eqref{Bstar}, there is a hit-or-miss set $F_*$ for $B^*$ of size 
at most $|F_*|\leq f_*(\lambda,k)$. Moreover, let $F'$ be an edge set of size at most $4k$
such that $F'$ separates $s$ from $t$ in $H$ if $s,t$ can be separated by at most $k$ edges,
such that $F'$ separates $s'$ from $t'$ in $L$ if $s',t'$ can be separated by at most $k$ edges,
such that $F'$ separates $s$ from $s'$ in $\bigcup_{i=1}^\ell Q_i$ if that is possible with 
at most $k$ edges, and such that 
that $F'$ separates $t$ from $t'$ in $\bigcup_{i=1}^\ell R_i$ if that is possible with 
at most $k$ edges. 
Put $F_1=F_*\cup F'$ and observe that $|F_1|\leq f_*(\lambda,k)+4k$. We claim that:
\begin{equation}\label{failure}
\begin{minipage}[c]{0.8\textwidth}\em
if $F_1$ fails a module $M$ of $H$, then it is as in {\rm (e)}.
In particular, $M$ is unexceptional.
\end{minipage}\ignorespacesafterend 
\end{equation} 
Indeed, if $M$ is as in (a), i.e.\ if $M$ is an $s$--$t$-path, then if $F_1$
fails $M$ there must be, by choice of $F'$, $k$ edge-disjoint $s$--$t$-paths in $H$ which contradicts the fact that $F_1$ fails $M$.
If $M$ is as in (d), i.e.\ if $M\subseteq B^*$, then $F_1$ does not fail $M$
as $F_*\subseteq F_1$ is a hit-or-miss set for $B^*$. 

Thus, assume that $M$ is as
in (b) or (c). In both cases, $M\cap B^*$ is a module of the block-part $B^*$, by Lemma~\ref{smallerlem}. 
As $F_*$ is a hit-or-miss set
for $B^*$, if $F_1$ does not hit $M$, then there must be $k$ edge-disjoint
modules $M^*_1,\ldots,M^*_k$ of the block-part $B^*$ that have the same blueprint as $M\cap B^*$. 
If $M$ is as in (b), i.e.\ if $M\cap L$ is an $s'$--$t'$-path, then, by choice of $F'$,
there are $k$ edge-disjoint $s'$--$t'$-paths $P_1,\ldots,P_k$ in $L$. 
It now follows from Lemma~\ref{replacelem2} 
that $M^*_1\cup P_1,\ldots,M^*_k\cup P_k$ are edge-disjoint modules of $H$ of the same blueprint as $M$
and thus, $F_1$ does not fail $M$.

If $M$ is as in (c), that is, if $M$ is exceptional, 
then $L$ in the role of $H_1$ and $B^*$ in the role of $H_2$
satisfy the conditions of Lemma \ref{exceptlem}.
Its application shows that $M\cap L$ is
the union of two disjoint $\{s,t\}$--$\{s',t'\}$-path.
As $L$ does not contain the disjoint union of an $s$--$t'$-path and an $s'$--$t$-paths,
we see that $M\cap L$ is the disjoint union of an $s$--$s'$-path and an $t$--$t'$-path.
The rest of the argument is similar to case~(b).
Hence, we have shown that $F_1$ does not fail any module of types (a), (b), (c) or (d).
This proves~\eqref{failure}.

Next let $F''$ be the edge set consisting, for each $a\in\{s,t\}$ and $b\in\{s',t'\}$,  
of at most $8k$ edges separating $a$ from $b$ in $L$, if possible, 
and of at most $k$ edges meeting every $a$--$x$-path that is disjoint from $V(L)-\{s,t\}$, if possible, 
and of at most $k$ edges separating $b$ from $x$ in $B^*$, if possible. 
Then we have $|F''|\le 32k+2k+2k=36k$.

Put $F_2=F_1\cup F''$. Then 
\begin{equation}\label{sizeF2}
|F_2|\leq f_*(\lambda,k)+40k.
\end{equation}
If $F_2$ does not fail any module of $H$, then we are clearly done. 
Thus, assume it fails a module $M$. By~\eqref{failure},
$M$ contains an $\{s,t\}$--$\{s',t'\}$-path, and therefore, 
by definition of $F''$, we see that 
there is $a\in\{s,t\}$ and $b\in\{s',t'\}$ such that there are 
$8k$ edge-disjoint $a$--$b$-paths in $L$.
Since $M$ is unexceptional, it follows from Lemma~\ref{hittinglem}
that $F_2$ cannot meet every $a$--$x$-path that is disjoint from $V(L)-\{s,t\}$. 
Thus there must exist
$k$ edge-disjoint $a$--$x$-paths
that meet $L$ at most in $\{s,t\}$.
As the hit-or-miss set $F_*\subseteq F_2$ for $B^*$  does not hit 
the module $M\cap B^*$ (see Lemma~\ref{smallerlem}),
it follows that there are $k$ edge-disjoint modules $M^*_1,\ldots,M^*_k$ of $B^*$
of the same blueprint as $M\cap B^*$. 
Consulting Figure~\ref{basicfig} and Lemma \ref{moduleHasBlueprint}, we see that every module of $H$ that is not an $s$--$t$-path
contains a path between any vertex to some vertex in $X$. Consequently, 
each $M^*_i$ contains 
a $b$--$X$-path that is then contained in $B^*$ (the path may contain both $s',t'$).


Let $L^*$ be an auxiliary graph defined on 
$\{s_i,t_i:i=1\ldots,\ell+1\}$
as vertex set and with $\{s_it_i: i=1,\ldots, \ell\}$
as edge set. 
Assume that $L^*$ contains a matching of size $3k+3$.
Then $L$ contains $3k+3$ disjoint parts $S_i$.
Moreover, by the previous paragraph,
there vertices $a,b$ as in the definition of well-connected ladders.
Hence, we conclude that $L$ contains a well-connected ladder.
But then, Lemma~\ref{ladderfanlem} implies $k$ edge-disjoint $K_4$-subdivisions which contradicts \eqref{nok}.
Hence, by K\H onig's theorem, the graph $L^*$ has a vertex cover of size at most $3k+2$
and therefore, 
\begin{equation}\label{smallZ}
\begin{minipage}[c]{0.8\textwidth}\em
there is a set $Z\subseteq\{s_0,\ldots, s_{\ell+1},t_0,\ldots,t_{\ell+1}\}$ with $|Z|\le 3k+2$
such that every $S_i$ is incident with a vertex in $Z$.
\end{minipage}\ignorespacesafterend 
\end{equation} 

Consider a vertex $z\in Z$, and assume that $z\in\{t_0,\ldots, t_{\ell+1}\}$. 
Let $i$ be the smallest index such that $S_i$ has terminals $s_i$ and $z$, and let $j$
be the largest index such that $S_{j+1}$ has terminals $s_{j+1}$ and $z$. Then 
the union $W_z$ of $S_i,\ldots,S_{j+1}$ and $Q_i,\ldots, Q_j$ is a fan-graph with terminals $z,s_i,s_{j+1}$. 
In the analogous way, we get
a fan-graph if $z\in\{s_0,\ldots, s_{\ell+1}\}$.
Now, the fan-graphs $W_z$ and $W_{z'}$ for two distinct $z,z'\in Z$ might overlap.
By shortening fan-graphs, we can find a set $\mathcal W$ of such fan-graphs 
that are pairwise edge-disjoint and that contains all parts $S_i$ of $L$
and such that $|\mathcal W|\leq |Z|$. Moreover, for the set $\mathcal R$ of non-trivial $Q_i,R_j$
that are not contained in any $W\in\mathcal W$ we get that $|\mathcal R|\leq 2|Z|$.
By construction, it follows that 
$L=\bigcup_{R\in\mathcal R}R\cup\bigcup_{W\in \mathcal W}W$ is an edge-disjoint union.

Let $\mathcal W'$ be the subset of $\mathcal W$ of fan-graphs of size at least $3k$,
and let $\mathcal R'$ be the union of $\mathcal R$ with all non-trivial $Q_i,R_i,S_i$
contained in some $W\in\mathcal W\sm\mathcal W'$. Then $|\mathcal R'|\leq (3k+2)|Z|\leq 25k^2$,
by~\eqref{smallZ}, 
and we still have that 
$L=\bigcup_{R\in\mathcal R'}R\cup\bigcup_{W\in \mathcal W'}W$ is an edge-disjoint union.
Note that $|\cW'|\le  |\cW|\le |Z|$.
For later use, we state that
\begin{equation}\label{sizeRW}
|\mathcal R'|+|\mathcal W'|\leq 25k^2+5k\leq 30k^2.
\end{equation}

Consider $W\in\mathcal W'$. By symmetry, we may assume that $W$ consists of $S_i,\ldots, S_j$ and $R_i,\ldots R_{j-1}$,
and that each $S_r$, $r=i,\ldots,j$ has terminals $s_r$ and $t'\in\{t_0,\ldots, t_{\ell+1}\}$. 
Assume first that there is an edge set $F_W$ of size at most $k$ that separates
a $c\in\{s_i,s_j\}$ from $x$ in $G-(W-c)$. We may assume that $c=s_i$. 
Then $W$ is a simple pseudo-part of $(H,F_W)$ as in~(a) of the definition of simple pseudo-part, 
with $s_i$ in the role of $w'$, $s_j$ in the role of $s'$ and $\bigcup_{r=i+1}^jE(S_r)$ in the role of $D$.

If there is no $c\in\{s_i,s_j\}$ that can be separated from $x$ in $G-(W-c)$ by at most $k$ edges,
then, 
as we are done when $G$ does  contains a well-connected fan, by Lemma~\ref{ladderfanlem}, 
we may assume that there is a set $F_W$ of size at most $6k$ such that $s_i$ is 
separated from $s_j$ in $W-t'$. Then $W$ is a simple pseudo-part of $(H,F_W)$ as in~(b)
of the definition.
In both cases, $W$ is a simple pseudo-part of $(H,F_W)$ with $|F_W|\le 6k$.

Put $F_3=F_2\cup \bigcup_{W\in\mathcal W'}F_W$ and observe with~\eqref{sizeF2} and~\eqref{smallZ} that 
\[
|F_3|\leq f_*(\lambda,k)+40k+|Z|\cdot 6k\leq f_*(\lambda,k)+ 40k+30k^2\leq f_*(\lambda,k)+70k^2.
\]
Note that because of~\eqref{failure} and because $L$ is the edge-disjoint union 
of the simple (and non-trivial) parts in $\mathcal R'$ with the simple pseudo-parts in $\mathcal W'$ 
of $(H,F_3)$ we may apply
Lemma~\ref{badcase1lem} and then obtain  a hit-or-miss set $F$ for $H$ of size at most 
\begin{align*}
|F|& \le |F_3|+2k(|\mathcal R'|+|\mathcal W'|)\leq f_*(\lambda,k)+70k^2+60k^3 \\
&\le f_*(\lambda,k)+130k^3\le f(\lambda,k)-2k,
\end{align*} 
where we have used~\eqref{sizeRW} in the second inequality.
\end{proof}

We now have proved that for all types I-V of $H$ 
we find a hit-or-miss set of size $f_*(\lambda,k)$ resp.\ $f(\lambda,k)$
and thus, we have proved Lemma \ref{fewlem}.
This completes the proof of our main theorem.

\section{Size of the hitting set}

What is a lower bound on the size of the edge hitting set for $K_4$-subdivisions?
Fiorini et al.\ note in the introduction of \cite{FJW12}
that there are graphs $G$ on $n$ vertices with treewidth $\Omega(n)$
and girth $\Omega(\log n)$ and Raymond et al.~\cite{RST13}
mention that these graphs even can be chosen cubic.
Every $K_4$-subdivision in such a graph $G$ 
contains at least $\Omega(\log n)$ vertices as it contains a cycle and the girth of $G$ is $\Omega(\log n)$.
Because $G$ is cubic,
no vertex is contained in two edge-disjoint $K_4$-subdivisions
and therefore, $G$ contains at most $k=\Omega(\frac{n}{\log n})$ edge-disjoint $K_4$-subdivisions.
If $X$ is an edge hitting set for $K_4$-subdivisions,
then the treewidth of $G-X$ is at most 2
and as the deletion of an edge decreases the treewidth by at most 1,
the set $X$ has to contain $\Omega(n)$ edges.
We conclude that $|X|\ge \Omega(k \log k)$ which is a lower bound for the edge hitting set for $K_4$-subdivisions.
We do not know a better lower bound.

The size of the hitting set in Theorem \ref{mainThm} is far away from this lower bound.
Let us recall that the power of $k$ in our theorem comes from the following parts:
the single vertex hitting set reduction gives a factor of $k\log k$;
the reduction to $2$-connected graphs adds a factor of $k^3$,
the number of $x$-ears contributes a factor of $k$,
and for each $x$-ear the induction adds a final factor of $k^3$.
Clearly, some of these steps could be optimised to lower the 
size of the hitting set but it seems very doubtful that a 
better bound than $O(k^4)$ could be reached. 

%
%
%

\bibliographystyle{amsplain}
\bibliography{erdosposa}

\vfill

\small
\vskip2mm plus 1fill
\noindent
Version \today{}
\bigbreak

\noindent
Henning Bruhn
{\tt <henning.bruhn@uni-ulm.de>}\\
Matthias Heinlein
{\tt <matthias.heinlein@uni-ulm.de>}\\
Institut f\"ur Optimierung und Operations Research\\
Universit\"at Ulm\\
Germany\\

\end{document}